\documentclass{amsart}
\usepackage{amsthm,amsmath,amssymb,amscd, mathrsfs}
\usepackage[latin1]{inputenc}
\usepackage[all]{xy}
\usepackage{url}
\usepackage[dvipsnames]{xcolor}
\usepackage{color}
\usepackage{framed}

\numberwithin{equation}{section}

\theoremstyle{plain}
\newtheorem{thm}{Theorem}[section]
\newtheorem*{thm*}{Theorem}
\newtheorem{prop}[thm]{Proposition}
\newtheorem{lem}[thm]{Lemma}
\newtheorem{cor}[thm]{Corollary}
\newtheorem{conj}[thm]{Conjecture}

\newtheorem{lem-defn}[thm]{Lemma-Definition}

\newtheorem*{claim*}{Claim}

\theoremstyle{definition}
\newtheorem{defn}[thm]{Definition}
\newtheorem{example}[thm]{Example}

\newtheorem{construction}[thm]{Construction}
\newtheorem{set-up}[thm]{Set-up}

\theoremstyle{remark}
\newtheorem{rem}[thm]{Remark}

\DeclareMathOperator{\Frac}{Frac}

\DeclareMathOperator{\Ker}{Ker}
\DeclareMathOperator{\Cok}{Coker}
\DeclareMathOperator{\Image}{Im}
\DeclareMathOperator{\Hom}{Hom}
\DeclareMathOperator{\End}{End}

\DeclareMathOperator{\id}{id}

\DeclareMathOperator{\tr}{tr}
\DeclareMathOperator{\rank}{rank}
\DeclareMathOperator{\Fil}{Fil}
\DeclareMathOperator{\gr}{gr}

\DeclareMathOperator{\ad}{ad}

\DeclareMathOperator{\Gal}{Gal}

\DeclareMathOperator{\Spec}{Spec}
\DeclareMathOperator{\Spa}{Spa}
\DeclareMathOperator{\Spf}{Spf}

\newcommand{\ra}{\rightarrow}
\newcommand{\lra}{\longrightarrow}

\newcommand{\hra}{\hookrightarrow}

\newcommand{\pf}{\mathrm{pf}}

\newcommand{\can}{\mathrm{can}}

\newcommand{\cont}{\mathrm{cont}}

\newcommand{\ur}{\mathrm{ur}}
\newcommand{\cris}{\mathrm{cris}}
\newcommand{\st}{\mathrm{st}}
\newcommand{\pst}{\mathrm{pst}}
\newcommand{\dR}{\mathrm{dR}}

\newcommand{\logcris}{\mathrm{log}\, \mathrm{cris}}

\newcommand{\et}{\mathrm{\acute{e}t}}
\newcommand{\proet}{\mathrm{pro\acute{e}t}}

\newcommand{\MF}{\mathrm{MF}}

\newcommand{\Rep}{\mathrm{Rep}}

\newcommand{\OB}{\mathcal{O}\mathbb{B}}

\newcommand{\N}{\mathbb{N}}
\newcommand{\Z}{\mathbb{Z}}
\newcommand{\Q}{\mathbb{Q}}
\newcommand{\R}{\mathbb{R}}

\newcommand{\A}{\mathbb{A}}
\newcommand{\B}{\mathbb{B}}
\newcommand{\rmA}{\mathrm{A}}
\newcommand{\rmB}{\mathrm{B}}

\newcommand{\fkf}{\mathfrak{f}}

\newcommand{\fkp}{\mathfrak{p}}

\newcommand{\fkX}{{\mathfrak X}}
\newcommand{\fkY}{{\mathfrak Y}}

\newcommand{\bB}{\mathbb{B}}

\newcommand{\bL}{\mathbb{L}}

\newcommand{\bT}{\mathbb{T}}

\newcommand{\calE}{\mathcal{E}}

\newcommand{\calG}{\mathcal{G}}
\newcommand{\calH}{\mathcal{H}}

\newcommand{\calK}{\mathcal{K}}
\newcommand{\calL}{\mathcal{L}}
\newcommand{\calM}{\mathcal{M}}

\newcommand{\calO}{\mathcal{O}}

\newcommand{\calX}{\mathcal{X}}

\begin{document}
\title{
A $p$-adic monodromy theorem for de Rham local systems
}
\date{\today}
\author{Koji Shimizu}
\address{School of Mathematics, Institute for Advanced Study, Princeton, NJ, USA}
\email{kshimizu@ias.edu}

\begin{abstract}
We study horizontal semistable and horizontal de Rham representations of the absolute Galois group of a certain smooth affinoid over a $p$-adic field. In particular, we prove that a horizontal de Rham representation becomes horizontal semistable after a finite extension of the base field.
As an application, we show that every de Rham local system on a smooth rigid analytic variety becomes horizontal semistable \'etale locally around every classical point.
We also discuss potentially crystalline loci of de Rham local systems and cohomologically potentially good reduction loci of smooth proper morphisms.
\end{abstract}

\maketitle

\section{Introduction}

Let $p$ be a prime and let $K$ be a complete discrete valuation field of characteristic zero with perfect residue field of characteristic $p$.

Fontaine introduced the notions of crystalline, semistable, and de Rham representations of the absolute Galois group of $K$ to study $p$-adic \'etale cohomology of algebraic varieties over $K$.  It is now known that Galois representations arising from geometry in this way  are de Rham.
He also conjectured that every de Rham representation is potentially semistable. This reflects the semistable reduction conjecture stating that every smooth proper algebraic variety over $K$ acquires semistable reduction after a finite field extension.
Fontaine's conjecture was proved by Berger \cite{Berger} and is often referred to as $p$-adic monodromy theorem.

This article discusses a similar question for families of $p$-adic Galois representations parametrized by a variety. More precisely, let $X$ be a smooth rigid analytic variety over $K$ (regarded as an adic space) and let $\bL$ be an \'etale $\Z_p$-local system on $X$. For every classical point $x\in X$ (i.e., a point whose residue class field $k(x)$ is finite over $K$) and a geometric point $\overline{x}$ above $x$, the stalk $\bL_{\overline{x}}$ is a Galois representation of $k(x)$. In this way, we regard $\bL$ as a family of Galois representations. In particular, we can ask whether $\bL$ is de Rham, semistable, etc., at every classical point. Note that building on earlier works \cite{Faltings-p-adicHodge, Hyodo-Hodge-Tate, Faltings-almostetaleextensions, Brinon-crisdeRham, AI-cris}, Scholze \cite{Scholze-p-adicHodge} defines the notion of de Rham local systems, and Liu and Zhu \cite{Liu-Zhu} prove that $\bL$ is a de Rham local system if and only if $\bL$ is de Rham at every classical point of $X$ (or even at a single classical point on each connected component). In this context, we prove the following $p$-adic monodromy theorem:

\begin{thm}[cf.~Theorem~\ref{thm:p-adic monodromy}]\label{thm:main thm in introduction}
Let $\bL$ be a de Rham $\Z_p$-local system on a smooth rigid analytic variety $X$ over $K$.
For every $K$-rational point $x\in X$, there exist an analytic open neighborhood $U\subset X$ of $x$ and a finite extension $L$ of $K$ such that $\bL|_{U\otimes_KL}$ is semistable at every classical point.
Moreover, if $\bL_{\overline{x}}$ is potentially crystalline as a Galois representation of $k(x)$, then we can choose $U$ and $L$ so that $\bL|_{U\otimes_KL}$ is crystalline at every classical point.
\end{thm}

We obtain a similar result for every classical point.
Note that $X$ can have a proper open subset containing all the $K$-rational points or even all the classical points in general. Hence the union of open sets $U$ in the theorem may not be the entire $X$.

\begin{rem}
Currently, there seems no good way to define crystalline or semistable local systems on $X$ without fixing a good formal model of $X$. This is the reason why our $p$-adic monodromy theorem concerns pointwise semistability.
See \cite{AI-cris, AI-st, Faltings-cryscohom, Faltings-F-isocrystals, Faltings-almostetaleextensions, Tan-Tong, Tsuji-cris} for definitions of crystalline or semistable local systems when starting with a nice integral or formal model.
\end{rem}

\begin{rem}
 In \cite{Kisin-LocalConstancy}, Kisin studies torsion and $\ell$-adic \'etale  local systems on $X$ for a prime $\ell\neq p$, and shows that around each $K$-rational point, such a local system is locally constant in the analytic topology of $X$.
In contrast, $p$-adic \'etale local systems on $X$ are not necessarily locally constant in the analytic topology. Our result can be regarded as a $p$-adic analogue of Kisin's result.
\end{rem}

\begin{rem}
 Note that Theorem~\ref{thm:main thm in introduction} is of local nature. We may also ask a global version of $p$-adic monodromy conjecture (cf.~\cite[Remark 1.4]{Liu-Zhu}):
 Let $X$ be a smooth rigid analytic variety over $K$, and let $\bL$ be a de Rham $\Z_p$-local system on $X$.
Does there exist a rigid analytic variety $X'$ finite \'etale over $X$ such that
$\bL|_{X'}$ is semistable at every classical point?
We know an affirmative answer in the following two cases:
\begin{enumerate}
 \item the case where $\bL$ is the $p$-adic Tate module of an abelian variety over $X$
(use Raynaud's criterion \cite[Expos\'e IX, Proposition 4.7]{SGA7-1}).
 \item the case where $\bL$ has a single Hodge--Tate weight \cite[Theorem 1.6]{const_HTwts}.
\end{enumerate}
On the other hand, Lawrence and Li \cite{Lawrence-Shizhang} show that there exist an algebraic variety $X$ over $K$ and a de Rham $\Z_p$-local system $\bL$ on $X$ such that for every finite extension $L$ of $K$, $\bL|_{X\otimes_KL}$ is not semistable at some closed point.
\end{rem} 

\begin{rem}
  It is an interesting attempt to formulate a semistable reduction conjecture in families over a field of mixed characteristic;
Theorem~\ref{thm:main thm in introduction} will be $p$-adic Hodge-theoretic evidence for such a conjecture. 
In the case of fields of characteristic zero, 
 Abramovich and Karu \cite{Abramovich-Karu} make a semistable reduction conjecture in families, and Adiprasito, Liu, and Temkin \cite{Adiprasito-Liu-Temkin} prove the conjecture of Abramovich and Karu in a larger generality. See these articles for the details.
\end{rem}

\bigskip

Let us turn to applications of Theorem~\ref{thm:main thm in introduction}.
Note that Liu and Zhu's result implies that if an \'etale $\Q_p$-local system $\bL$ on $X$ is de Rham at a classical point on each connected component, then $\bL$ is de Rham at every classical point. Such a strong result does not hold for potentially crystalline representations.
However, we prove that if $\bL$ is potentially crystalline at a classical point $x$, then it is so at every classical point in an analytic open neighborhood of $x$.
\begin{thm}[Theorem~\ref{thm:potentially crystalline loci}]
If $\bL$ is a de Rham $\Q_p$-local system on a smooth rigid analytic variety $X$ over $K$, then there exists an analytic open subset $U\subset X$ such that the following holds for every classical point $x\in X$:
\begin{enumerate}
 \item if $x\in U$, then $\bL$ is potentially crystalline at $x$.
 \item if $x\not\in U$, then $\bL$ is not potentially crystalline at $x$.
\end{enumerate}
We call such $U$ a potentially crystalline locus of $\bL$.
\end{thm}

Note that our definition takes only classical points into consideration and thus a potentially crystalline locus of $\bL$ is not unique.
On the other hand, the union of potentially crystalline loci is also a potentially crystalline locus. Hence we obtain the largest potentially crystalline locus of $\bL$.

The second application concerns the special case where $\bL$ comes from geometry.
For example, suppose that we have a family of elliptic curves $f\colon E\ra X$.
For a point $x\in X(\overline{K})$, the elliptic curve $E_{x}$ has good reduction if and only if the $j$-invariant of $E_x$ is a $p$-adic integer. The set of such points may well be called the (potentially) good reduction locus of $E$.
In general, if we have a smooth proper family of algebraic varieties, the \'etale cohomology of the fibers yields both $\ell$-adic and $p$-adic \'etale local systems on the base.
By combining known cases of $\ell$-adic/$p$-adic monodromy-weight conjecture and the $\ell$-independence with our result or Kisin's result, we obtain the following theorem:

\begin{thm}[Theorem~\ref{thm:potentially good reduction locus}]
Assume that the residue field of $K$ is finite. Let $f\colon Y\ra X$ be a smooth proper morphism between smooth algebraic varieties over $K$ such that the relative dimension of $f$ is at most two.
For each $m\in\N$, there exists an analytic open subset $U$ of the adic space $X^{\ad}$ associated to $X$ such that the following holds for every classical point $x$ of $X^{\ad}$:
\begin{enumerate}
 \item if $x\in U$, then $R^mf_\ast\Q_\ell$ is potentially unramified at $x$ for $\ell\neq p$ and 
$R^mf_\ast\Q_p$ is potentially crystalline at $x$.
 \item if $x\not\in U$, then $R^mf_\ast\Q_\ell$ is not potentially unramified at $x$ for $\ell\neq p$ and $R^mf_\ast\Q_p$ is not potentially crystalline at $x$. 
\end{enumerate}
We call such $U$ a cohomologically potentially good reduction locus of $f$ of degree $m$.
\end{thm}

As in the previous theorem, such a locus is not unique but we obtain the largest cohomologically potentially good reduction locus.

\begin{rem}
Imai and Mieda \cite{Imai-Mieda} introduce a notion of the potentially good reduction locus of a Shimura variety using automorphic local systems. They prove the existence of the potentially good reduction loci of Shimura varieties of preabelian type, and use these loci to study the \'etale cohomology of Shimura varieties.
\end{rem}

\medskip

We now explain ideas of the proof of Theorem~\ref{thm:main thm in introduction}.
We start with an example that motivates our proof:
Let $\fkf\colon \fkY\ra\fkX$ be a smooth proper morphism of smooth $p$-adic formal schemes over $\Spf\calO_K$, and let $f\colon Y\ra X$ denote its adic generic fiber.
Then $\bL:=R^mf_\ast\Q_p$ is a de Rham $\Q_p$-local system. In fact, we even know that it is crystalline at every classical point (hence there is nothing to prove in this case). Let $z$ be a closed point in the special fiber of $\fkX$ and let $U\subset X$ be the residue open disk of $z$.
For $K$-rational points $x_1,x_2\in U$, there is a canonical identification of $\varphi$-modules\footnote{They need not be isomorphic as \emph{filtered} $\varphi$-modules, and thus $\bL_{\overline{x_1}}$ and $\bL_{\overline{x_2}}$ may not be isomorphic as Galois representations of $K$.}
\[
 D_{\cris}(\bL_{\overline{x_1}})\cong H^m_{\cris}(\fkY_z/W(k))[p^{-1}]\cong D_{\cris}(\bL_{\overline{x_2}}).
\]
One key property of $U$ is that the Gauss--Manin connection on the vector bundle $H^m_{\dR}(Y/X)$ on $X$ has a full set of solutions when restricted to $U$ (cf.~\cite{Katz-Dwork} and \cite[Remark 2.9]{Berthelot-Ogus-isomI}).
Heuristically, this means that we can compare $\varphi$-modules associated to stalks of $\bL$ at different points on $U$ by parallel transport. 
Our proof articulates this idea despite the absence of integral models of $X$ nor motives for $\bL$.

Let us return to the general setting of Theorem~\ref{thm:main thm in introduction}.
Liu and Zhu \cite{Liu-Zhu} introduce a natural functor $D_{\dR}$ from the category of \'etale $\Q_p$-local systems on $X$ to the category of filtered vector bundles with integrable connections on $X$. When $X=\Spa(K,\calO_K)$, this agrees with Fontaine's $D_{\dR}$-functor for Galois representations of $K$. The functor commutes with arbitrary pullbacks, and an \'etale $\Q_p$-local system $\bL$ on $X$ is de Rham if and only if $\rank D_{\dR}(\bL)=\rank \bL$.
Note that in the above geometric situation, $D_{\dR}(R^mf_\ast\Q_p)$ coincides with $H_{\dR}^m(Y/X)$ equipped with the Gauss--Manin connection and the Hodge filtration.

Suppose that $\bL$ is a de Rham $\Z_p$-local system and fix a $K$-rational point $x$ of $X$.
Since $X$ is smooth, the theory of $p$-adic differential equations over a polydisk (Theorem~\ref{thm:pDF on polydisk}) tells that there is an analytic open neighborhood $U\subset X$ of $x$ such that $D_{\dR}(\bL|_U)$ has a full set of solutions.
By shrinking $U$ if necessary, we may further assume that $U$ is isomorphic to a rigid torus
$\Spa(A,A^\circ)$, where $A^\circ:=\calO_K\langle T_1^{\pm 1},\ldots,T_n^{\pm 1}\rangle$, the $p$-adic completion of the Laurent polynomial ring $\calO_K[T_1^{\pm 1},\ldots,T_n^{\pm 1}]$, and $A:=A^\circ[p^{-1}]$. 

Let $\calG_A$ denote the absolute Galois group of $A$. 
Following works of Colmez \cite{Colmez} and Brinon \cite{Brinon-crisdeRham}, we define horizontal crystalline, semistable, and de Rham period rings $\rmB_{\max}^\nabla(A^\circ), \rmB_{\st}^\nabla(A^\circ), \rmB_{\dR}^\nabla(A^\circ)$. Using these period rings, we define the notions of horizontal crystalline, semistable, and de Rham representations of $\calG_A$ (Definition~\ref{def:horizontal de Rham representations}).
In our setting, $\bL|_U$ corresponds to a $p$-adic representation $V$ of $\calG_A$. Since $\bL|_U$ is de Rham and $D_{\dR}(\bL|_U)$ has a full set of solutions, $V$ is horizontal de Rham. Now the first part of Theorem~\ref{thm:main thm in introduction} follows from a $p$-adic monodromy theorem for horizontal de Rham representations:

\begin{thm}[Theorem~\ref{thm:p-adic monodromy for horizontal de Rham representations}]\label{thm:p-adic monodromy for horizontal de Rham representations in intro}
If a $p$-adic representation $V$ of $\calG_A$ is horizontal de Rham, then there exists a finite extension $L$ of $K$ such that $V|_{\calG_{A_L}}$ is horizontal semistable, where $A_L:=A\otimes_KL$.
\end{thm}
The second part of Theorem~\ref{thm:main thm in introduction} follows from the study of the monodromy operator (cf.~Theorem~\ref{thm:horizontal de Rham and one point crystalline}).

To explain the proof of Theorem~\ref{thm:p-adic monodromy for horizontal de Rham representations in intro}, we need some more notation.
Let $\calO_{\calK}$ denote the $p$-adic completion of the localization $A^\circ_{(\pi)}$, where $\pi$ is a uniformizer of $K$. Then $\calK:=\calO_{\calK}[p^{-1}]$ is a complete discrete valuation field with imperfect residue field. We similarly have the notions of horizontal semistable and de Rham representations of the absolute Galois group of $\calK$.
In \cite{Ohkubo}, Ohkubo shows that every horizontal de Rham representation becomes horizontal semistable after restricting to the absolute Galois group of the composite $L\calK$ for some finite extension $L$ of $K$.
Hence Theorem~\ref{thm:p-adic monodromy for horizontal de Rham representations in intro} follows from Ohkubo's result and the following purity theorem for horizontal semistable representations of $\calG_A$:
\begin{thm}[Theorem~\ref{thm:purity for horiztaonlly semistable}]
 If a horizontal de Rham representation of $\calG_A$ is horizontal semistable when restricted to the absolute Galois group of $\calK$, then it is also horizontal semistable as a representation of $\calG_A$.
\end{thm}
 
This is the key technical result of our article and the proof requires detailed analysis of period rings (e.g.~Lemma~\ref{lem:Cartesian diagram of period rings}). Our purity theorem and its proof are inspired by similar results in different contexts:
 a $p$-adic monodromy theorem for de Rham representations of the absolute Galois group of $K$ with values in a reduced affinoid algebra by Berger and Colmez \cite{Berger-Colmez}, and a purity theorem for crystalline local systems in the case of good reduction by Tsuji \cite[Theorem~5.4.8]{Tsuji-cris}.

\medskip

At this point, we emphasize that the main goal of our work is to develop $p$-adic Hodge theory of horizontal crystalline, semistable, and  de Rham representations of $\calG_A$. 
These classes of representations are much more restrictive than those of crystalline and de Rham representations studied by Brinon \cite{Brinon-crisdeRham}. However, it is still worth studying them.
In fact, as we have explained, de Rham local systems come from horizontal de Rham representations (\'etale) locally around each classical point.
For another instance, Faltings \cite{Faltings-coverings} studies the admissible loci of certain $p$-adic period domains using horizontal crystalline representations (see also \cite{Hartl-Rapoport-Zink,Kedlaya-ICM, Fargues-ICM} for relevant works).

This article establishes two results in $p$-adic Hodge theory of $\calG_A$.
First we prove the $p$-adic monodromy theorem for horizontal de Rham representations as we have discussed. Second we define Fontaine's functors $D_{\pst}$ and $V_{\pst}$ in our setting (Section~\ref{section:filtered phi N modules}). We prove that these functors induce an equivalence of categories between the category of horizontal de Rham representations of $\calG_A$ and that of admissible discrete filtered $(\varphi,N,\Gal(\overline{K}/K),A)$-modules (Theorem~\ref{thm:equivalence by Dpst and Vpst}). 
What is missing in our work is to characterize the admissibility condition in the latter category in terms of more intrinsic properties like semistability (weak admissibility). In the case of $p$-adic Hodge theory for $K$, this is first established by Colmez and Fontaine \cite{Colmez-Fontaine}.
Unfortunately, results on $p$-adic period domains already suggest that the admissibility in our setting cannot be characterized by (classical) pointwise weak admissibility (cf.~\cite[Example 6.7]{Hartl-Rapoport-Zink}).

\medskip

Finally, we explain the organization of this paper.
The first part (Sections~\ref{section:period rings}-\ref{section:p-adic monodromy theorem for horizontal de Rham representations}) develops $p$-adic Hodge theory for rings (i.e., representations of $\calG_A$) and the second part (Sections~\ref{section:p-adic Hodge theory for rigid analytic varieties}-\ref{section:potentially crystalline loci and potentially good reduction loci}) studies $p$-adic Hodge theory for rigid analytic varieties.

In the first part, we need to define horizontal period rings such as $\rmB_{\dR}^\nabla(A^\circ)$. We define them in two steps. First we define horizontal period rings such as $\B_{\dR}(\Lambda,\Lambda^+)$ for every perfectoid Banach pair $(\Lambda,\Lambda^+)$ over the completed algebraic closure of $K$. For a base ring $R$ (e.g.~$A^\circ$), the $p$-adic completion $\widehat{\overline{R}}$ of the integral closure of $R$ in the algebraic closure of $R[p^{-1}]$ gives such a perfectoid Banach pair, and we set $\rmB_{\dR}^\nabla(R):=\B_{\dR}(\widehat{\overline{R}}[p^{-1}], \widehat{\overline{R}})$.

In Section~\ref{section:period rings}, we define horizontal period rings for perfectoid Banach pairs, following \cite{Colmez} and \cite{Brinon-crisdeRham}.
Most of the arguments in this section are standard in $p$-adic Hodge theory. One exception might be Lemma~\ref{lem:Cartesian diagram of period rings}, which will be used crucially in the proof of the purity theorem in Section~\ref{section:purity}.
Section~\ref{section:base rings} studies the class of base rings $R$ over which Brinon \cite{Brinon-crisdeRham} develops $p$-adic Hodge theory. In Section~\ref{section:horizontal semistable representations}, we define horizontal crystalline, semistable, and de Rham representations. Section~\ref{section:filtered phi N modules} introduces filtered $(\varphi,N,\Gal(L/K),A)$-modules, and the functors $D_{\pst}$ and $V_{\pst}$. This section is not used in later sections.
After these preparations, we state and prove the purity theorem in Section~\ref{section:purity}. We prove Theorem~\ref{thm:p-adic monodromy for horizontal de Rham representations in intro} in Section~\ref{section:p-adic monodromy theorem for horizontal de Rham representations}. 

The second part starts with Section~\ref{section:p-adic Hodge theory for rigid analytic varieties}, where we review \'etale local systems on smooth rigid analytic varieties and period sheaves on the pro-\'etale site. Our $p$-adic monodromy theorem is proved in Section~\ref{section:p-adic monodromy theorem for de Rham local systems}.
Finally, we discuss potentially crystalline loci and cohomologically potentially good reduction loci in Section~\ref{section:potentially crystalline loci and potentially good reduction loci}.

\medskip
\noindent
\textbf{Acknowledgments.}
The author would like to thank Takeshi Tsuji for sending the preprint \cite{Tsuji-cris} and Shun Ohkubo for email correspondences on the article \cite{Ohkubo}. He would also like to thank H\'el\`ene Esnault, Teruhisa Koshikawa, and Alex Youcis for valuable discussions and comments about this work.
The author was supported by the National Science Foundation under Grant No.~DMS-1638352 through membership at the Institute for Advanced Study.

\medskip
\noindent
\textbf{Notation.}

For a ring $S$ of characteristic $p$, the Frobenius map $x\mapsto x^p$ is denoted by $\varphi$. We say that $S$ is \emph{perfect} if $\varphi$ is an isomorphism. In this case, $W(S)$ denotes the ring of $p$-typical Witt vectors of $S$. By functoriality, $W(S)$ admits the Witt vector Frobenius $W(\varphi)$, which we simply denote by $\varphi$.

A rigid analytic variety over $K$ refers to a quasi-separated adic space locally of finite type over $\Spa(K,\calO_K)$ (cf.~ \cite[Proposition 4.5(iv)]{Huber-gen}).
For a rigid analytic variety $X$ over $K$, an (analytic) open subset of $X$ refers to an open subset of $X$ as an adic space.
Definitions of \'etale $\Z_p$-local systems and \'etale $\Q_p$-local systems on $X$ are explained in Subsection~\ref{subsection:local systems}.

\section{Period rings}\label{section:period rings}

In this section, we introduce horizontal period rings $\B_{\dR}, \B_{\max}, \B_{\st}$, etc., and study their basic properties. We will use these period rings to define horizontal de Rham, crystalline, and semistable representations in Section~\ref{section:horizontal semistable representations}.
Colmez define these period rings for sympathetic algebras in \cite{Colmez}.
For our later technical purpose, we define them in a more general setting. 
Most proofs in this section have essentially appeared in \cite{Fontaine-exposeII, Colmez, Brinon-crisdeRham, KL-I}.
Lemma~\ref{lem:Cartesian diagram of period rings} is a key technical result used in the proof of the purity theorem in Section~\ref{section:purity}.

\subsection{Review of Banach algebras and perfectoid pairs}
In this subsection, let $L$ be a field that is complete with respect to a non-archimedean $p$-adic norm $\lvert \;\rvert$ satisfying $\lvert p\rvert=p^{-1}$.

\begin{defn}[cf. {\cite[\S 2.2]{Colmez}}]\hfill
\begin{enumerate}
 \item  Let $\Lambda$ be an $L$-algebra.
An \emph{$L$-algebra seminorm} $\lvert\,\rvert_\Lambda$ on $\Lambda$ is a function
$\lvert\,\rvert_\Lambda\colon \Lambda\ra \R_{\geq 0}$
such that 
\begin{itemize}
 \item $\lvert 0\rvert_\Lambda=0$ and $\lvert 1\rvert_\Lambda=1$;
 \item $\lvert x+y\rvert_\Lambda\leq \max\{\lvert x\rvert_\Lambda, \lvert y\rvert_\Lambda\}$ 
and $\lvert xy\rvert_\Lambda\leq \lvert x\rvert_\Lambda\cdot\lvert y\rvert_\Lambda$ for $x,y\in\Lambda$;
 \item $\lvert ax\rvert_\Lambda=\lvert a\rvert\cdot \lvert x\rvert_\Lambda$ for $a\in L$ and $x\in\Lambda$.
\end{itemize}
We say that an $L$-algebra seminorm is an \emph{$L$-algebra norm} if 
$\lvert x\rvert_\Lambda=0$ holds only when $x=0$. 
In this case, we call the pair $(\Lambda,\lvert\,\rvert_\Lambda)$ an \emph{$L$-normed algebra}.
If there is no confusion, we simply say that $\Lambda$ is an $L$-normed algebra and denote its fixed $L$-algebra norm by $\lvert\,\rvert_\Lambda$.
 \item An \emph{$L$-Banach algebra} is an $L$-normed algebra $(\Lambda,\lvert\,\rvert_\Lambda)$
such that $\Lambda$ is complete with respect to the topology defined by $\lvert\,\rvert_\Lambda$. A morphism of $L$-Banach algebras will refer to a continuous $L$-algebra homomorphism.
 \item Let $\Lambda$ be an $L$-Banach algebra. An element $x\in\Lambda$ is called \emph{power-bounded} if the subset $S=\{x^n\mid n\in\N\}\subset \Lambda$ is bounded, i.e., for every open neighborhood $U$ of $0$, there exists an open neighborhood $V$ of $0$ such that $VS\subset U$.
The set of power-bounded elements forms a subring of $\Lambda$, which we denote by $\Lambda^\circ$.
 \item (\cite[Definition 2.8.1]{KL-I}) An $L$-Banach algebra $\Lambda$ is called \emph{uniform} if the subring $\Lambda^\circ$ of $\Lambda$ is bounded.
Note that $\Lambda$ is uniform if and only if the norm on $\Lambda$ is equivalent to its spectral (semi)norm\footnote{If $\Lambda$ is uniform, then $\Lambda$ with the spectral norm is uniform in the sense of \cite[p.~16]{Berkovich-book}.}.
 \item (\cite[Definition 2.4.1]{KL-I})
An $L$-Banach pair\footnote{It is called an adic $L$-Banach ring in \cite{KL-I}.} is a pair $(\Lambda,\Lambda^+)$ in which $\Lambda$ is an $L$-Banach algebra and $\Lambda^+$ is a subring of $\Lambda^\circ$ which is open and integrally closed in $\Lambda$. Such a $\Lambda^+$ is also called \emph{a ring of integral elements}.
A morphism of $L$-Banach pairs $(\Lambda,\Lambda^+)\ra (\Lambda',\Lambda'^+)$
is a morphism $f\colon \Lambda\ra \Lambda'$ of $L$-Banach algebras such that $f(\Lambda^+)\subset \Lambda'^+$.
\end{enumerate}
\end{defn}

\begin{rem}
 Note that in \cite{KL-I} a morphism of $L$-Banach algebras is required to be bounded.
\end{rem}

\begin{rem}
 If $(\Lambda,\Lambda^+)$ is an $L$-Banach pair and if $\Lambda$ is uniform, then $\Lambda^+$ is reduced, $p$-torsion free, and $p$-adically complete and separated.
\end{rem}

\begin{defn}
Assume further that $L$ is algebraically closed.
An $L$-Banach pair $(\Lambda,\Lambda^+)$ is called \emph{perfectoid} if $\Lambda$ is uniform and if the Frobenius morphism $x\mapsto x^p$ on $\Lambda^+/p\Lambda^+$ is surjective.
\end{defn}

\begin{rem}
Under the assumption that $L$ is algebraically closed, 
 an $L$-Banach pair $(\Lambda,\Lambda^+)$ is perfectoid in the above sense if and only if it is perfectoid in the sense of \cite[Definition 3.6.1]{KL-I} by \cite[Proposition 3.6.2 (a)]{KL-I}.
In the case where $\Lambda^+=\Lambda^\circ$, this definition is equivalent to the condition that $\Lambda$ is a perfectoid $L$-algebra in the sense of \cite[Definition 5.1 (i)]{Scholze-perfectoid}, in which case we will simply say that $\Lambda$ is perfectoid (over $L$).
By \cite[Proposition 3.6.2 (d)]{KL-I}, $(\Lambda,\Lambda^+)$ is perfectoid if and only if $\Lambda$ is perfectoid, i.e., $(\Lambda,\Lambda^\circ)$ is perfectoid.
\end{rem}

\subsection{Period rings for perfectoid pairs}

Let $k$ be a perfect field of characteristic $p$ and set $K_0:=W(k)[p^{-1}]$.
Let $K$ be a totally ramified finite extension of $K_0$. We fix a uniformizer $\pi$ of $K$ and an algebraic closure $\overline{K}$ of $K$. Let $C$ denote the $p$-adic completion of $\overline{K}$. Then $C$ is algebraically closed by Krasner's lemma. We denote by $\calO_{C}$ the ring of integers of $C$

\begin{set-up}
Let $(\Lambda,\Lambda^+)$ be a perfectoid $C$-Banach pair such that $\Lambda^+$ is an $\calO_{C}$-algebra. Note that $\Lambda=\Lambda^+[p^{-1}]$.
 \end{set-up}

\begin{rem}
When we discuss a morphism $(\Lambda,\Lambda^+)\ra (\Lambda',\Lambda'^+)$ of such objects, we consider a morphism of $K$-Banach pairs as opposed to that of $C$-Banach pairs. For example, we will consider $\sigma\colon(C,\calO_{C})\ra (C,\calO_{C})$ for $\sigma\in \Gal(\overline{K}/K)$.
\end{rem}

We set
\[
\Lambda^{+\flat}:=\varprojlim_{\varphi}\Lambda^+/p\Lambda^+.
\]
By \cite[Remark 3.4.10]{KL-I}, this is a ring of characteristic $p$ (called the \emph{tilt} of $\Lambda^+$) and there is a natural multiplicative monoid isomorphism
\[
 \varprojlim_{\lambda\mapsto \lambda^p}\Lambda^+\stackrel{\cong}{\lra}\Lambda^{+\flat}.
\]
We denote by $\Lambda^{+\flat}\ra \Lambda^+; x\mapsto x^\sharp$ the inverse of the above isomorphism followed by the first projection.
Note that $\Lambda^{+\flat}$ has a norm $\lvert \;\rvert$ defined by $\lvert x\rvert:=\lvert x^\sharp\rvert_{\Lambda}$ and it is complete with respect to this norm (\cite[Lemma 3.4.5]{KL-I}).

We set
\begin{align*}
\A_{\inf} (\Lambda^+)&=\A_{\inf}(\Lambda,\Lambda^+):=W(\Lambda^{+\flat}),\quad\text{and}\\
\A_{\inf,K}(\Lambda^+)&=\A_{\inf,K}(\Lambda,\Lambda^+):=\calO_K\otimes_{W(k)}\A_{\inf}(\Lambda^+). 
\end{align*}
By functoriality of $p$-typical Witt vectors, $\A_{\inf}(\Lambda^+)$ carries the Frobenius $\varphi$ lifting $x\mapsto x^p$.
There exists a unique surjective $W(k)$-algebra homomorphism
\[
 \theta_{\A_{\inf}(\Lambda^+)}\colon \A_{\inf}(\Lambda^+)\ra \Lambda^+
\]
characterized by
\[
 \theta_{\A_{\inf}(\Lambda^+)}\biggl(
\sum_{i=0}^\infty [x_i]p^i
\biggr)
=\sum_{i=0}^\infty x_i^\sharp p^i,
\]
where $[x_i]\in \A_{\inf}(\Lambda^+)$ is the Teichm\"uller lift of $x_i$
(cf.~\cite[Lemma 3.2.2, Definition 3.4.3]{KL-I}).
This map also extends to a surjective $\calO_K$-algebra homomorphism
\[
 \theta_{\A_{\inf,K}(\Lambda^+)}\colon \A_{\inf,K}(\Lambda^+)\ra \Lambda^+.
\]

Choose a compatible system $(p_m)_{m\in\N}$ of $p$-power roots of $p$ in $\overline{K}$, i.e., $p_m\in \overline{K}$ with $p_0=p$ and $p_{m+1}^p=p_m$.
Set $p^\flat:=(p_0\bmod{p},p_1\bmod{p},\ldots)\in \calO_{C}^\flat\subset \Lambda^{+\flat}$.
Similarly, choose a compatible system $(\pi_m)_{m\in\N}$ of $p$-power roots of $\pi$ in $\overline{K}$ and set $\pi^\flat:=(\pi_0\bmod{p},\pi_1\bmod{p},\ldots)\in \calO_{C}^\flat\subset \Lambda^{+\flat}$.

\begin{lem}\label{lem: ker theta is principal}
The ideal $\Ker \theta_{\A_{\inf}(\Lambda^+)}$ is principal.
In fact, it is generated by $[p^\flat]-p$.
Similarly, $\Ker \theta_{\A_{\inf,K}(\Lambda^+)}$ is principal with a generator $[\pi^\flat]-\pi$.
\end{lem}

\begin{proof}
The first assertion is \cite[Lemma 3.6.3]{KL-I}\footnote{It is also a special case of the second assertion by considering the case where $K=K_0$.}.
We prove the second.
Note $[\pi^\flat]-\pi\in \Ker \theta_{\A_{\inf,K}(\Lambda^+)}$.
Since $\A_{\inf,K}(\Lambda^+)$ is $\pi$-adically complete and $\Lambda^+$ is $\pi$-torsion free, it suffices to prove that the ideal $(\Ker \theta_{\A_{\inf,K}(\Lambda^+)}, \pi)$
is equal to $([\pi^\flat]-\pi,\pi)$.
Take $x\in \Ker \theta_{\A_{\inf,K}(\Lambda^+)}$.
Note that $x$ is uniquely written as
$x=\sum_{i=0}^\infty[x_i]\pi^i$ for some $x_i\in \Lambda^{+\flat}$.
Since $0=\theta_{\A_{\inf,K}(\Lambda^+)}(x)=\sum_{i=0}^\infty x_i^\sharp \pi^i$,
we see that $\pi$ divides $x_0^\sharp$ in $\Lambda^+$.
Using the multiplicative monoid isomorphism $\varprojlim_{\lambda\mapsto \lambda^p}\Lambda^+\stackrel{\cong}{\lra}\Lambda^{+\flat}$ and the fact that $\Lambda^+$ is integrally closed, we see that $\pi^\flat$ divides $x_0$ in $\Lambda^{+\flat}$ and thus 
$[\pi^\flat]$ divides $[x_0]$ in $\A_{\inf,K}(\Lambda^+)$.
Write $[x_0]=y[\pi^\flat]$ for some $y\in \A_{\inf,K}(\Lambda^+)$.
Then we have
\[
 x=y([\pi^\flat]-\pi)+\pi\bigl(y+\sum_{i=1}^\infty[x_i]\pi^{i-1}\bigr).
\]
This implies
$(\Ker \theta_{\A_{\inf,K}(\Lambda^+)}, \pi)=([\pi^\flat]-\pi,\pi)$.
\end{proof}

Let us define the horizontal de Rham period rings:

\begin{defn}
 We set
\begin{align*}
\B_{\dR}^+(\Lambda^+)&=\B_{\dR}^+(\Lambda,\Lambda^+) :=\varprojlim_m \A_{\inf}(\Lambda^+)[p^{-1}]/(\Ker \theta_{\A_{\inf}(\Lambda^+)}[p^{-1}])^m, \quad\text{and} \\
\B_{\dR,K}^+(\Lambda^+)&=\B_{\dR,K}^+(\Lambda,\Lambda^+) :=\varprojlim_m \A_{\inf,K}(\Lambda^+)[p^{-1}]/(\Ker \theta_{\A_{\inf,K}(\Lambda^+)}[p^{-1}])^m.
\end{align*}
\end{defn}

We will prove that these two rings are naturally isomorphic in Lemma~\ref{lem:horizontal de Rham period rings are isomorphic}, and write simply $\B_{\dR}^+(\Lambda^+)$ for them afterward.

\begin{lem}\label{lem:BdR is complete}
 The map $\theta_{\A_{\inf}(\Lambda^+)}$ extends to
\[
 \theta_{\B_{\dR}^+(\Lambda^+)}\colon
\B_{\dR}^+(\Lambda^+)\ra \Lambda=\Lambda^+[p^{-1}],
\]
and $\B_{\dR}^+(\Lambda^+)$ is complete with respect to the ideal $\Ker \theta_{\B_{\dR}^+(\Lambda^+)}$.
Similarly, the map $\theta_{\A_{\inf,K}(\Lambda^+)}$ extends to
\[
 \theta_{\B_{\dR,K}^+(\Lambda^+)}\colon
\B_{\dR,K}^+(\Lambda^+)\ra \Lambda,
\]
and $\B_{\dR,K}^+(\Lambda^+)$ is complete with respect to the ideal $\Ker \theta_{\B_{\dR,K}^+(\Lambda^+)}$.
\end{lem}

\begin{proof}
It suffices to prove the second assertion as the first is a special case of the second where $K=K_0$.
Obviously, $\theta_{\A_{\inf,K}(\Lambda^+)}$ extends to
$\theta_{\B_{\dR,K}^+(\Lambda^+)}\colon\B_{\dR,K}^+(\Lambda^+)\ra \Lambda$.
Since $\Lambda$ is $p$-torsion free, the ideal $\Ker \theta_{\A_{\inf,K}(\Lambda^+)}[p^{-1}]$ of $\A_{\inf,K}(\Lambda^+)[p^{-1}]$ is also generated by $[\pi^\flat]-\pi$, and thus finitely generated.
It follows that $\B_{\dR,K}^+(\Lambda^+)$ is complete with respect to $\Ker \theta_{\B_{\dR,K}^+(\Lambda^+)}$. 
\end{proof}

Note that  the natural homomorphisms
\[
 \A_{\inf}(\Lambda^+)[p^{-1}]/(\Ker \theta_{\A_{\inf}(\Lambda^+)}[p^{-1}])^m
\ra \A_{\inf,K}(\Lambda^+)[p^{-1}]/(\Ker \theta_{\A_{\inf,K}(\Lambda^+)}[p^{-1}])^m
\]
give rise to a ring homomorphism
\[
 \B_{\dR}^+(\Lambda^+)\ra\B_{\dR,K}^+(\Lambda^+).
\]

Choose a compatible system $(\varepsilon_m)_{m\in\N}$ of nontrivial $p$-power roots of unity in $\overline{K}$, i.e., $\varepsilon_m\in \overline{K}$ with $\varepsilon_0=1$, $\varepsilon_1\neq 1$, and $\varepsilon_{m+1}^p=\varepsilon_m$.
Set 
\[
 \varepsilon:=(\varepsilon_0\bmod{p},\varepsilon_1\bmod{p},\ldots)\in \calO_{C}^\flat\subset \Lambda^{+\flat}. 
\]
Note that $[\varepsilon]-1\in \A_{\inf}(\Lambda^+)$ satisfies $\theta_{\A_{\inf}(\Lambda^+)}([\varepsilon]-1)=0$.
In particular, we define
\[
t:=\sum_{m=1}^\infty\frac{(-1)^{m-1}}{m}([\varepsilon]-1)^m
\]
as an element of $\B_{\dR}^+(\Lambda^+)$.

\begin{lem}\label{lem: t is a generator}
 The ideal $\Ker \theta_{\B_{\dR}^+(\Lambda^+)}$ is generated by $t$.
Similarly, $\Ker \theta_{\B_{\dR,K}^+(\Lambda^+)}$ is generated by $t$.
\end{lem}

\begin{proof}
Since the first assertion is a special case of the second, we only show the second.
 We know that $\Ker \theta_{\B_{\dR,K}^+(\Lambda^+)}$ is generated by $([\pi^\flat]-\pi)$.
Hence it suffices to prove that $t$ divides $[\pi^\flat]-\pi$ in $\B_{\dR,K}^+(\Lambda^+)$.
By functoriality of our construction, it is enough to prove this for $(\Lambda,\Lambda^+)=(C,\calO_{C})$.
This is \cite[1.5.4]{Fontaine-exposeII}.
\end{proof}

\begin{lem}\label{lem:horizontal de Rham period rings are isomorphic}
 The natural homomorphism
\[
 \B_{\dR}^+(\Lambda^+)\ra\B_{\dR,K}^+(\Lambda^+)
\]
is an isomorphism.
\end{lem}

\begin{proof}
This follows from Lemmas~\ref{lem:BdR is complete} and \ref{lem: t is a generator}.
\end{proof}

As mentioned before, we identify these rings and denote them by $\B_{\dR}^+(\Lambda^+)$ in what follows. In particular, $\B_{\dR}^+(\Lambda^+)$ is a $K$-algebra.

\begin{lem}
 The natural map $\A_{\inf,K}(\Lambda^+)\ra \B_{\dR}^+(\Lambda^+)$ is injective.
\end{lem}

\begin{proof}
 We claim 
\begin{align*}
 \A_{\inf,K}(\Lambda^+)\cap \bigl(\Ker\theta_{\A_{\inf,K}(\Lambda^+)}[p^{-1}]\bigr)^m&=\bigl(\Ker\theta_{\A_{\inf,K}(\Lambda^+)}\bigr)^m,\quad\text{and}\\
\bigcap_m\bigl(\Ker\theta_{\A_{\inf,K}(\Lambda^+)}[p^{-1}]\bigr)^m&=0.
\end{align*}
Note that the assertion follows from the second claim.
The first claim follows from the fact that $\Ker\theta_{\A_{\inf,K}(\Lambda^+)}[p^{-1}]$ is generated by $[\pi^\flat]-\pi$ and $\Lambda^+$  is $p$-torsion free.
By the first claim, the second claim is reduced to $\bigcap_m\bigl(\Ker\theta_{\A_{\inf,K}(\Lambda^+)}\bigr)^m=0$, which follows from the fact that $\A_{\inf,K}(\Lambda^+)/(\pi)=\Lambda^{+\flat}$ has a norm $\lvert\,\rvert$ satisfying $\lvert \pi^\flat\rvert=\lvert \pi\rvert_{\Lambda}>0$.
\end{proof}

\begin{lem}
The ring $\B_{\dR}^+(\Lambda^+)$ is $t$-torsion free.
 \end{lem}

\begin{proof}
It is enough to show that $\B_{\dR}^+(\Lambda^+)$ is $([p^\flat]-p)$-torsion free.
Since $\B_{\dR}^+(\Lambda^+)$ is the completion of $\A_{\inf}(\Lambda^+)[p^{-1}]$ with respect to the ideal $([p^\flat]-p)$, it remains to prove that the latter ring is $([p^\flat]-p)$-torsion free.
Since $\Lambda^+$ is $p$-torsion free, $\Lambda^{+\flat}$ is $p^\flat$-torsion free.
It follows that $\A_{\inf}(\Lambda^+)$ is $([p^\flat]-p)$-torsion free and so is $\A_{\inf}(\Lambda^+)[p^{-1}]$.
\end{proof}

\begin{defn}
 Set
\[
 \B_{\dR}(\Lambda^+)=\B_{\dR}(\Lambda,\Lambda^+):=\B_{\dR}^+(\Lambda^+)[t^{-1}].
\]
We equip $\B_{\dR}(\Lambda^+)$ with the decreasing filtration defined by
\[
 \Fil^m\B_{\dR}(\Lambda^+):=t^m\B_{\dR}^+(\Lambda^+)\quad (m\in\Z).
\]
This is separated and exhaustive.
\end{defn}

\begin{defn}
Let $\A_{\inf}(\Lambda^+)[\frac{\Ker \theta}{p}]$
be the $\A_{\inf}(\Lambda^+)$-subalgebra of $\A_{\inf}(\Lambda^+)[p^{-1}]$ generated by $p^{-1}\Ker \theta_{\A_{\inf}(\Lambda^+)}$.
 We define
$\A_{\max}(\Lambda^+)=\A_{\max}(\Lambda,\Lambda^+)$ to be the $p$-adic completion of $\A_{\inf}(\Lambda^+)[\frac{\Ker \theta}{p}]$.

Similarly, let $\A_{\inf,K}(\Lambda^+)[\frac{\Ker \theta}{\pi}]$
be the $\A_{\inf,K}(\Lambda^+)$-subalgebra of $\A_{\inf,K}(\Lambda^+)[p^{-1}]$ generated by $\pi^{-1}\Ker \theta_{\A_{\inf,K}(\Lambda^+)}$.
 We define
$\A_{\max,K}(\Lambda^+)=\A_{\max,K}(\Lambda,\Lambda^+)$ to be the $p$-adic completion of $\A_{\inf,K}(\Lambda^+)[\frac{\Ker \theta}{\pi}]$.
\end{defn}

\begin{lem}
The Frobenius $\varphi$ on $\A_{\inf}(\Lambda^+)$ extends to an endomorphism on $\A_{\inf}(\Lambda^+)[\frac{\Ker \theta}{p}]$ and thus on $\A_{\max}(\Lambda^+)$.
\end{lem}

We call this endomorphism on $\A_{\max}(\Lambda^+)$ the Frobenius and denote it by $\varphi$.

\begin{proof}
Recall that $\Ker\theta_{\A_{\inf}(\Lambda^+)}$ is generated by $[p^\flat]-p$.
Hence the assertion follows from
\[
 \varphi\biggl(\frac{[p^\flat]-p}{p}\biggr)=\frac{[p^\flat]^p-p}{p}=p^{p-1}\biggl(\frac{[p^\flat]-p}{p}+1\biggr)^p-1\in \A_{\inf}(\Lambda^+)\biggl[\frac{\Ker \theta}{p}\biggr].
\]
\end{proof}

\begin{lem}
 The rings $\A_{\max}(\Lambda^+)$ and $\A_{\max,K}(\Lambda^+)$
are $p$-torsion free.
\end{lem}

\begin{proof}
It is enough to show that $\A_{\max,K}(\Lambda^+)$ is $p$-torsion free.
 Note that $\A_{\inf,K}(\Lambda^+)$ is a finite free module over $\A_{\inf}(\Lambda^+)$ and thus $p$-torsion free.
Hence $\A_{\inf,K}(\Lambda^+)[\frac{\Ker\theta}{\pi}]$ is also $p$-torsion free. The assertion for $\A_{\max,K}(\Lambda^+)$ follows from this.
\end{proof}

\begin{defn}
Set
\begin{align*}
 \B_{\max}^+(\Lambda^+)&=\B_{\max}^+(\Lambda,\Lambda^+):=\A_{\max}(\Lambda^+)[p^{-1}],
\quad\text{and}\\
 \B_{\max,K}^+(\Lambda^+)&=\B_{\max,K}^+(\Lambda,\Lambda^+):=\A_{\max,K}(\Lambda^+)[p^{-1}].
\end{align*}
 \end{defn}

\begin{lem}
 The element
\[
 t=\sum_{m=1}^\infty\frac{(-1)^{m-1}}{m}([\varepsilon]-1)^m
\]
converges in $\frac{1}{p^r}\A_{\max,K}(\Lambda^+)$ for a sufficiently large $r\in\N$.
It also converges in $\A_{\max}(\Lambda^+)$.
\end{lem}

\begin{proof}
 Since the proofs are similar (and the second assertion is easier to prove), we show the first.
Write
\[
 \frac{(-1)^{m-1}}{m}([\varepsilon]-1)^m
=\frac{(-1)^{m-1}\pi^m}{m} \biggl(\frac{[\varepsilon]-1}{\pi}\biggr)^m.
\]
Observe that $\frac{(-1)^{m-1}\pi^m}{m}$ approaches zero $p$-adically when $m\to\infty$ and thus there exists $r\in\N$ such that $\frac{(-1)^{m-1}\pi^m}{m}\in \frac{1}{p^r}\A_{\max,K}(\Lambda^+)$ for every $m$.
Hence the first assertion follows from the above observation and $[\varepsilon]-1\in \Ker \theta_{\A_{\max,K}(\Lambda^+)}$.
\end{proof}

\begin{prop}
We have
\[
\calO_K\otimes_{W(k)}\A_{\inf}(\Lambda^+)\!\biggl[\frac{\Ker\theta}{p}\biggr]
\!\!\subset\! \A_{\inf,K}(\Lambda^+)\!\biggl[\frac{\Ker\theta}{\pi}\biggr]
\!\!\subset\! \frac{1}{p}\!\biggl(\!\calO_K\otimes_{W(k)}\A_{\inf}(\Lambda^+)\!\biggl[\frac{\Ker\theta}{p}\biggr]\biggr).
\]
In particular,
\[
 K\otimes_{K_0}\B_{\max}^+(\Lambda^+)\cong \B_{\max,K}^+(\Lambda^+).
\]
\end{prop}

\begin{proof}
Set $e=[K:K_0]$ and write $p=a\pi^e$ and $p^\flat=a^\flat (\pi^\flat)^e$ for some $a\in \calO_K^\times$ and $a^\flat\in (\calO_{C}^\flat)^\times$.
For every $m\in \N$, we have
\[
 \biggl(\frac{[p^\flat]}{p}\biggr)^m=[a^\flat]^ma^{-m}\biggl(\frac{[\pi^\flat]}{\pi}\biggr)^{em}.
\]
On the other hand, if we write $m=qe+r$ with $0\leq r<e$, we have
\[
 \biggl(\frac{[\pi^\flat]}{\pi}\biggr)^m
=\frac{1}{p}\biggl(a[a^\flat]^{-1}\frac{[p^\flat]}{p}\biggr)^q\frac{p}{\pi^r}[\pi^\flat]^r.
\]
The first assertion follows from these equalities and Lemma~\ref{lem: ker theta is principal}.
Taking the $p$-adic completion of the first part yields
\[
\calO_K\otimes_{W(k)}\A_{\max}(\Lambda^+)
\subset \A_{\max,K}(\Lambda^+)
\subset \frac{1}{p}\biggl(\calO_K\otimes_{W(k)}\A_{\max}(\Lambda^+)\biggr).
\]
Hence the second assertion follows by inverting $p$.
\end{proof}

Let us now turn to Colmez's description of the horizontal period rings $\B_{\dR}^+(\Lambda^+)$ and $\B_{\max,K}^+(\Lambda^+)$ as $K$-vector spaces \cite[\S 8.4]{Colmez}.

\begin{construction}\label{const:section of theta map}
Fix a family of elements $(e_i)_{i\in I}$ of $\Lambda^+$ such that the images of $e_i$'s in $\Lambda^+/\pi\Lambda^+$ form a $k$-basis of $\Lambda^+/\pi\Lambda^+$.
Then every element of $\Lambda$ is written uniquely as $\sum_{i\in I}a_ie_i$ with $a_i\in K$ such that for every $r>0$ there are only finitely many $i$'s with $\lvert a_i\rvert_\Lambda\geq r$.
For each $i\in I$, choose $\widetilde{e}_i\in \A_{\inf,K}(\Lambda^+)$ such that $\theta_{\A_{\inf,K}(\Lambda^+)}(\widetilde{e}_i)=e_i$.
Consider the map
\[
 s\colon \Lambda\ra \A_{\inf,K}(\Lambda^+)[p^{-1}];\quad
\sum_{i\in I}a_ie_i\mapsto\sum_{i\in I}a_i\widetilde{e}_i.
\]
This is a $K$-linear map and satisfies $\theta_{\A_{\inf,K}(\Lambda^+)}\circ s=\id$ and $s(\Lambda^+)\subset \A_{\inf,K}(\Lambda^+)$. In particular, $s$ is continuous with respect to $p$-adic topology.

Set 
\[
 v=\pi^{-1}([\pi^\flat]-\pi).
\]
Recall $\B_{\dR}^+(\Lambda^+)=\varprojlim_m \A_{\inf,K}(\Lambda^+)[p^{-1}]/(\Ker \theta_{\A_{\inf,K}(\Lambda^+)}[p^{-1}])^m$.
Note that $v$ generates $\Ker\theta_{\B_{\dR}^+(\Lambda^+)}$ by Lemmas~\ref{lem: ker theta is principal} and \ref{lem:horizontal de Rham period rings are isomorphic}. 
For $x\in \B_{\dR}^+(\Lambda^+)$, we define two sequences $\bigl(a_m(x)\bigr)_{m\in\N}$
and $\bigl(b_m(x)\bigr)_{m\in\N}$ with $a_m(x)\in \B_{\dR}^+(\Lambda^+)$ and $b_m(x)\in \Lambda$ by 
\begin{align*}
 a_0(x)&=x,\\
 b_m(x)&=\theta_{\B_{\dR}^+(\Lambda^+)}\bigl(a_m(x)\bigr),\quad a_{m+1}(x)=\frac{1}{v}\bigl(a_m(x)-s(b_m(x))\bigr).
\end{align*}
Define
\[
 \widetilde{\theta}_{s,v}\colon \B_{\dR}^+(\Lambda^+)\ra \Lambda[[X]]
\quad\text{by}\quad \widetilde{\theta}_{s,v}(x):=\sum_{m=0}^\infty b_m(x)X^m.
\]
\end{construction}

\begin{lem}\label{lem: properties of theta-tilde}
\hfill
\begin{enumerate}
 \item The map $\widetilde{\theta}_{s,v}$ is a $K$-linear isomorphism, and the inverse map is given by
$\sum_{m=0}^\infty c_mX^m\mapsto \sum_{m=0}^\infty s(c_m)v^m$.
 \item 
For $F(X)\in K[[X]]$ and $x\in \B_{\dR}^+(\Lambda^+)$, we have
\[
 \widetilde{\theta}_{s,v}\bigl(xF(v)\bigr)=\widetilde{\theta}_{s,v}(x)F(X).
\]
 \item We have the following inclusions of $\calO_K$-modules:
\[
 \Lambda^+[[\pi^2 X]]\subset
\widetilde{\theta}_{s,v}\bigl(\A_{\inf,K}(\Lambda^+)\bigr)\subset
 \Lambda^+[[\pi X]].
\]
 \item
The map $\widetilde{\theta}_{s,v}$ induces natural identifications (as $K$-vector spaces)
\[
\A_{\max,K}(\Lambda^+)\cong \Lambda^+\langle X\rangle, \quad\text{and}\quad
\B_{\max,K}^+(\Lambda^+)\cong \Lambda\langle X\rangle.
\]
Here $\Lambda^+\langle X\rangle$ denotes the $p$-adic completion of the polynomial algebra $\Lambda^+[X]$ and $\Lambda\langle X\rangle:=\Lambda^+\langle X\rangle[p^{-1}]$.
\end{enumerate}
\end{lem}

\begin{proof}
Part (i) follows from the fact that $\Ker \theta_{\B_{\dR}^+(\Lambda^+)}$ is generated by $v=\pi^{-1}([\pi^\flat]-\pi)$, and
Part (ii) follows from the definition of $\widetilde{\theta}_{s,v}$ and the fact that $\theta_{\B_{\dR}^+(\Lambda^+)}$ and $s$ are both $K$-linear.

We show Part (iii). 
For $x\in \A_{\inf}(\Lambda^+)$, we see $b_m(x)\in \pi^m\Lambda^+$ by induction on $m$.
Hence the second inclusion follows. For the first inclusion, choose any $f\in \Lambda^+[[\pi^2X]]$ and write $f=\sum_{m=0}^\infty c_m(\pi^2X)^m$ with $c_m\in\Lambda^+$.
Then $s(c_m)\in \A_{\inf,K}(\Lambda^+)$ and $y:=\sum_{m=0}^\infty s(c_m)\pi^m([\pi^\flat]-\pi)^m$ converges in $\A_{\inf,K}(\Lambda^+)$. It is now easy to see $\widetilde{\theta}_{s,v}(y)=f$.

For Part (iv), first note that Parts (ii) and (iii) imply
\[
 \Lambda^+[[\pi^2 X]][X]\subset
\widetilde{\theta}_{s,v}\biggl(\A_{\inf,K}(\Lambda^+)\biggl[\frac{\Ker\theta}{\pi}\biggr]\biggr)\subset
\Lambda^+[[\pi X]][X].
\]
Hence the identification $\A_{\max,K}(\Lambda^+)\cong \Lambda^+\langle X\rangle$  follows from
\[
 \varprojlim_{m}\Lambda^+[[\pi^2 X]][X]/p^m\Lambda^+[[\pi^2 X]][X]\cong
 \varprojlim_{m}\Lambda^+[[\pi X]][X]/p^m\Lambda^+[[\pi X]][X]\cong \Lambda^+\langle X\rangle.
\]
By inverting $p$, we obtain $\B_{\max,K}^+(\Lambda^+)\cong \Lambda\langle X\rangle$.
\end{proof}

We will show that there is an injective $K$-algebra homomorphism
\[
 \B_{\max,K}^+(\Lambda^+)\ra \B_{\dR}^+(\Lambda^+).
\]
Since $\B_{\dR}^+(\Lambda^+)$ is $\Ker \theta_{\B_{\dR}^+(\Lambda^+)}$-adically complete, the natural $K$-algebra homomorphism
$\A_{\inf,K}(\Lambda^+)\bigl[\frac{\Ker\theta}{\pi}\bigr]\ra \B_{\dR}^+(\Lambda^+)$
extends to $\A_{\max,K}(\Lambda^+)\ra \B_{\dR}^+(\Lambda^+)$ and thus to
$\B_{\max,K}^+(\Lambda^+)\ra \B_{\dR}^+(\Lambda^+)$.
We deduce the injectivity from the description of these period rings in terms of $\widetilde{\theta}_{s,v}$ as follows:

\begin{prop}
The map $\widetilde{\theta}_{s,v}$ induces identifications (as $K$-vector spaces)
\[
 \xymatrix{
\A_{\max,K}(\Lambda^+)\ar^\cong[d]\ar[r]
& \B_{\max,K}^+(\Lambda^+)\ar[d]^\cong\ar[r]
& \B_{\dR}^+(\Lambda^+)\ar[d]^\cong_{\widetilde{\theta}_{v,s}}\\
\Lambda^+\langle X\rangle \ar@{^{(}->}[r]
&\Lambda\langle X\rangle \ar@{^{(}->}[r]
&\Lambda[[X]], 
}
\]
making the diagram commutative.
In particular, the top horizontal maps are injective, and
$\B_{\max,K}^+(\Lambda)$ is $t$-torsion free.
\end{prop}

\begin{proof}
 We have already proved that the vertical maps are all isomorphisms.
Since the bottom horizontal maps are injective, so are the top horizontal maps.
Since $\B_{\dR}^+(\Lambda^+)$ is $t$-torsion free, the last assertion follows.
\end{proof}

\begin{defn}
 We set
\begin{align*}
 \B_{\max}(\Lambda^+)&=\B_{\max}(\Lambda,\Lambda^+):=\B_{\max}^+(\Lambda^+)[t^{-1}],
\quad\text{and}\\
 \B_{\max,K}(\Lambda^+)&=\B_{\max,K}(\Lambda,\Lambda^+):=\B_{\max,K}^+(\Lambda^+)[t^{-1}].
\end{align*}
Since $\varphi(t)=pt$, the Frobenius $\varphi$ on $\A_{\max}(\Lambda^+)$ extends to an endomorphism on $\B_{\max}(\Lambda^+)$, which we still call the Frobenius and denote by $\varphi$.
\end{defn}

Let us define the horizontal semistable period rings.
Consider the elements
\[
\log \frac{[p^\flat]}{p}\!:=\!\sum_{m=1}^\infty\!\frac{(-1)^{m-1}}{m}\biggl(\frac{[p^\flat]}{p}-1\!\biggr)^m\!\!,\;
 \log \frac{[\pi^\flat]}{\pi}\!:=\!\sum_{m=1}^\infty\!\frac{(-1)^{m-1}}{m}\biggl(\frac{[\pi^\flat]}{\pi}-1\!\biggr)^m\!\!\in \B_{\dR}^+(\Lambda^+).
\]

\begin{defn}
 Define $\B_{\st}^+(\Lambda^+)=\B_{\st}^+(\Lambda, \Lambda^+)$ to be 
the $\B_{\max}^+(\Lambda^+)$-subalgebra of $\B_{\dR}^+(\Lambda^+)$ generated by $\log \frac{[p^\flat]}{p}$.
Similarly, define $\B_{\st,K}^+(\Lambda^+)=\B_{\st,K}^+(\Lambda, \Lambda^+)$ to be 
the $\B_{\max,K}^+(\Lambda^+)$-subalgebra of $\B_{\dR}^+(\Lambda^+)$ generated by $\log \frac{[p^\flat]}{p}$
\footnote{Strictly speaking, our definitions slightly differ from the ones in \cite[\S 8.6]{Colmez}. However, our horizontal semistable period rings are isomorphic to Colmez's period rings by Corollary~\ref{cor:semistable period ring otimes K}.}.
\end{defn}

\begin{lem}\label{lem: well-definedness of Bst}
The definitions of  $\B_{\st}^+(\Lambda^+)$ and $\B_{\st}^+(\Lambda^+)$ are  independent of the choice of a compatible system of $p$-power roots $(p_m)$ of $p$. 
\end{lem}

\begin{proof}
 If $(p'_m)$ is a different choice, there exists $\alpha\in \Z_p$ such that $p'_m=\varepsilon_m^{\alpha}p_m$ for every $m\in\N$. It follows 
\[
 \log\frac{[p'^\flat]}{p}=\log\frac{[p^\flat]}{p}+\alpha t \quad\text{in $\B_{\dR}^+(\Lambda^+)$ with $\alpha t\in\B_{\max}^+(\Lambda^+)$}.
\]
\end{proof}

\begin{lem}\label{lem: Bst in BdR}
The ring $\B_{\st,K}^+(\Lambda^+)$ coincides with the $\B_{\max,K}^+(\Lambda^+)$-subalgebra of $\B_{\dR}^+(\Lambda^+)$ generated by $\log \frac{[\pi^\flat]}{\pi}$. 
\end{lem}

\begin{proof}
Set $e=[K:K_0]$.
 Write $p=a\pi^e$ for some $a\in\calO_K^\times$ and take
$a^\flat\in (\calO_{C}^\flat)^\times$ such that $p^\flat=a^\flat(\pi^\flat)^e$.
Then $[a^\flat]a^{-1}\in\A_{\inf,K}(\Lambda^+)$ satisfies $\theta_{\A_{\inf,K}(\Lambda^+)}([a^\flat]a^{-1}-1)=0$.
It follows that
the series
\[
 \sum_{m=1}^\infty\frac{(-1)^{m-1}}{m}\biggl(\frac{[a^\flat]}{a}-1\biggr)^m
\]
converges in $\A_{\max,K}(\Lambda^+)$. We denote the limit by $\log\frac{[a^\flat]}{a}$.
We can also check that the image of $\log\frac{[a^\flat]}{a}$ in $\B_{\dR}^+(\Lambda^+)$
is given by the same series and satisfies
\[
 \log \frac{[p^\flat]}{p}=\log\frac{[a^\flat]}{a}+e\log \frac{[\pi^\flat]}{\pi}.
\]
Hence the assertion follows.
\end{proof}

We will describe $\B_{\st,K}^+(\Lambda^+)$ via $\widetilde{\theta}_{s,v}$.
For this we need the following fact:

\begin{lem}\label{lem:transcendence of log}
 Let $A$ be a uniform $\Q_p$-Banach algebra.
Then the power series $\log(1+X):=\sum_{m=0}^\infty\frac{(-1)^{m-1}}{m}X^m\in A[[X]]$ is transcendental over $A\langle X\rangle$.
Hence the $A\langle X\rangle$-subalgebra $A\langle X\rangle[\log(1+X)]$ of $A[[X]]$ generated by $\log(1+X)$ is isomorphic to a polynomial algebra over $A\langle X\rangle$.
\end{lem}

\begin{proof}
Since $A$ is uniform,  
equip $A$ with its spectral norm. Then
the Gel'fand transform $A\ra \prod_{x\in\calM(A)}\calH(x)$ is isometric with each $\calH(x)$ a field (\cite[Corollary 1.3.2 (ii)]{Berkovich-book}).
We can thus reduce the general case to the case where $A$ is a field.
We may further assume that $A$ is a complete algebraically closed field.
In this case, the assertion is well-known.
In fact, suppose that there exist $f_0(X),\ldots,f_m(X)\in A\langle X\rangle$ such that
\[
 f_0(X)+f_1(X)\log (1+X)+\cdots+f_m(X)(\log (1+X))^m=0.
\]
If $\zeta$ is a $p$-power root of unity, then $X=\zeta-1$ is a root of $\log(1+X)$, and thus it is also a root of $f_0(X)$.
Since any nonzero element of $A\langle X\rangle$ has at most finitely many roots in $\calO_A$,
we have $f_0(X)=0$. By repeating the same argument, we conclude that $f_0(X),\ldots,f_m(X)$ are all zero.
\end{proof}

\begin{prop}\label{Prop: Colmez's description of semistable period ring}
The map $\widetilde{\theta}_{s,v}$ induces identifications (as $K$-vector spaces)
\[
 \xymatrix{
\B_{\max,K}^+(\Lambda^+)\ar[d]^\cong\ar[r]
& \B_{\st,K}^+(\Lambda^+) \ar[d]^\cong\ar[r]
& \B_{\dR}^+(\Lambda^+)\ar[d]^\cong\\
\Lambda\langle X\rangle \ar[r]
&\Lambda\langle X\rangle[\log(1+X)] \ar[r]
&\Lambda[[X]],
}
\]
sending $\log \frac{[\pi^\flat]}{\pi}$ to $\log(1+X)$.
\end{prop}

\begin{proof}
Recall $v=\frac{[\pi^\flat]}{\pi}-1$. Hence we have $\widetilde{\theta}_{s,v}((\log\frac{[\pi^\flat]}{\pi})^m)=(\log(1+X))^m$ by Lemma~\ref{lem: properties of theta-tilde}(ii).
By Lemma~\ref{lem: Bst in BdR}, 
\[
 \B_{\st,K}^{+}(\Lambda^+)=\sum_{m=0}^\infty \B_{\max,K}^+(\Lambda^+)\biggl(\log\frac{[\pi^\flat]}{\pi}\biggr)^m
\]
as $K$-vector subspaces of $\B_{\dR}^+(\Lambda^+)$.
It follows from Lemma~\ref{lem: properties of theta-tilde}(ii) and (iv) that
\[
 \widetilde{\theta}_{s,v}(\B_{\st,K}^{+}(\Lambda^+))=\sum_{m=0}^\infty \Lambda\langle X\rangle(\log (1+X))^m\subset \Lambda[[X]].
\]
\end{proof}

\begin{cor}\label{cor:semistable period ring otimes K}
The $\B_{\max}^+(\Lambda^+)$-algebra homomorphism
from the polynomial ring $\B_{\max}^+(\Lambda^+)[u]$ to $\B_{\st}^+(\Lambda^+)$ sending $u$ to $\log \frac{[p^\flat]}{p}$ is an isomorphism.
Similarly, the $\B_{\max,K}^+(\Lambda^+)$-algebra homomorphism $\B_{\max,K}^+(\Lambda^+)[u]\ra \B_{\st,K}^+(\Lambda^+); u\mapsto\log \frac{[p^\flat]}{p}$ is an isomorphism.
In particular, $K\otimes_{K_0}\B_{\st}^+(\Lambda^+)\cong \B_{\st,K}^+(\Lambda^+)$.
\end{cor}

\begin{proof}
For the second assertion, we only need to show that the homomorphism is bijective.
Hence it follows from Lemmas~\ref{lem: properties of theta-tilde}(ii), \ref{lem: Bst in BdR}, \ref{lem:transcendence of log}, and Proposition~\ref{Prop: Colmez's description of semistable period ring}.
The first assertion is reduced to the second assertion by considering the case $K=K_0$.
The last assertion follows from the first two since $K\otimes_{K_0}\B_{\max}^+(\Lambda^+)\cong \B_{\max,K}^+(\Lambda^+)$.
\end{proof}

\begin{defn}
We extend the Frobenius on $\B_{\max}(\Lambda^+)$ to the polynomial ring $\B_{\max}^+(\Lambda^+)[u]$
by setting $\varphi(u)=pu$.
Then the $\B_{\max}^+(\Lambda^+)$-linear derivation $N=-\frac{d}{du}$ on $\B_{\max}^+(\Lambda^+)[u]$ satisfies $N\varphi=p\varphi N$.
Via the $\B_{\max}^+(\Lambda^+)$-algebra isomorphism
\[
 \B_{\max}^+(\Lambda^+)[u]\stackrel{\cong}{\lra}\B_{\st}^+(\Lambda^+);\quad u\mapsto \log\frac{[p^\flat]}{p},
\]
we define the Frobenius $\varphi$ and the monodromy operator $N$ on $\B_{\st}^+(\Lambda^+)$. Note  $N\varphi=p\varphi N$ and 
\[
 \varphi\biggl(\log\frac{[p^\flat]}{p}\biggr)=p\log\frac{[p^\flat]}{p},\quad\text{and}\quad
N\biggl(\biggl(\log\frac{[p^\flat]}{p}\biggr)^m\biggr)=-m\biggl(\log\frac{[p^\flat]}{p}\biggr)^{m-1}.
\]
As in Lemma~\ref{lem: well-definedness of Bst}, we can check that the definitions of $\varphi$ and $N$ on $\B_{\st}^+(\Lambda^+)$ are independent of the choice of a compatible system of $p$-power roots of $p$.
\end{defn}

\begin{defn}
 Set
\begin{align*}
 \B_{\st}(\Lambda^+)&=\B_{\st}(\Lambda,\Lambda^+):=\B_{\st}^+(\Lambda^+)[t^{-1}],
\quad\text{and}\\
 \B_{\st,K}(\Lambda^+)&=\B_{\st,K}(\Lambda,\Lambda^+):=\B_{\st,K}^+(\Lambda^+)[t^{-1}].
\end{align*}
\end{defn}

Note that there are isomorphisms
\[
 \B_{\max}(\Lambda^+)[u]\stackrel{\cong}{\lra}\B_{\st}(\Lambda^+),
\quad\text{and}\quad
\B_{\max,K}(\Lambda^+)[u]\stackrel{\cong}{\lra}\B_{\st,K}(\Lambda^+)
\]
sending $u$ to $\log\frac{[p^\flat]}{p}$.
Moreover, the Frobenius $\varphi$ (resp.~the monodromy operator $N$) on $\B_{\st}^+(\Lambda^+)$ extends to a ring (resp.~$\B_{\max}^+(\Lambda^+)$-linear) endomorphism on $\B_{\st}(\Lambda^+)$, which we still denote by $\varphi$ (resp.~$N$).

\begin{lem}
 The sequence
\[
 0\lra \B_{\max}(\Lambda^+)\lra \B_{\st}(\Lambda^+)\stackrel{N}{\lra}\B_{\st}(\Lambda^+)\lra 0
\]
is exact.
\end{lem}

\begin{proof}
 This follows easily from the isomorphism $ \B_{\max}(\Lambda^+)[u]\stackrel{\cong}{\lra}\B_{\st}(\Lambda^+)$ commuting with $N$.
\end{proof}

We also introduce the horizontal crystalline period ring for completeness:
\begin{defn}
 Let $\A_{\cris}'(\Lambda^+)$ denote  the PD-envelope of $\A_{\inf}(\Lambda^+)$ relative to the ideal $\Ker \theta_{\A_{\inf}(\Lambda^+)}$ compatible with the canonical PD-structure on $p\A_{\inf}(\Lambda^+)$.
Define $\A_{\cris}(\Lambda^+)=\A_{\cris}(\Lambda, \Lambda^+)$ to be the $p$-adic completion of $\A_{\cris}'(\Lambda^+)$.
\end{defn}

Note that the natural inclusion 
$\A_{\cris}'(\Lambda^+)\hra \A_{\inf}(\Lambda^+)\bigr[\frac{\ker\theta}{p}\bigr]$
yields an inclusion
\[
 \A_{\cris}(\Lambda^+)\hra \A_{\max}(\Lambda^+).
\]

\begin{lem}
 The Frobenius $\varphi$ on $\A_{\inf}(\Lambda^+)$ extends to a ring endomorphism $\varphi$ on $\A_{\cris}(\Lambda^+)$.
\end{lem}

\begin{proof}
 Since $\varphi(x)\equiv x^p \bmod{p\A_{\inf}(\Lambda^+)}$ for $x\in \A_{\inf}(\Lambda^+)$,
the Frobenius $\varphi$ on $\A_{\inf}(\Lambda^+)$ extends to a ring endomorphism on $\A_{\cris}'(\Lambda^+)$.
Passing to the $p$-adic completion, it further extends to a ring endomorphism on $\A_{\cris}(\Lambda^+)$.
\end{proof}

By construction, the inclusion $\A_{\cris}(\Lambda^+)\hra \A_{\max}(\Lambda^+)$ is $\varphi$-equivariant.

\begin{lem}
 We have
\[
 \varphi(\A_{\max}(\Lambda^+))\subset \A_{\cris}(\Lambda^+)\subset \A_{\max}(\Lambda^+).
\]
\end{lem}

\begin{proof}
We need to prove the first inclusion.  For this, it is enough to show
\[
 \varphi\biggl(\A_{\inf}(\Lambda^+)\biggr[\frac{\ker\theta}{p}\biggr]\biggr)\subset \A_{\cris}'(\Lambda^+).
\]
Choose $x\in \Ker\theta_{\A_{\inf}(\Lambda^+)}$ and write $\varphi(x)=x^p+py$ for some $y\in\A_{\inf}(\Lambda^+)$. Then we have
\[
 \varphi\biggl(\frac{x}{p}\biggr)=(p-1)!x^{[p]}+y\in \A_{\cris}'(\Lambda^+).
\]
\end{proof}

\begin{defn}
 Set
\begin{align*}
  \B_{\cris}^+(\Lambda^+)&= \B_{\cris}^+(\Lambda, \Lambda^+):=\A_{\cris}(\Lambda)[p^{-1}],\quad\text{and}\\
 \B_{\cris}(\Lambda^+)&=\B_{\cris}(\Lambda, \Lambda^+):=\B_{\cris}(\Lambda)[t^{-1}].
\end{align*}
The Frobenius $\varphi$ on $\A_{\cris}(\Lambda^+)$ naturally extends over these rings, and the previous lemma implies
\[
 \varphi(\B_{\max}(\Lambda^+))\subset \B_{\cris}(\Lambda^+)\subset \B_{\max}(\Lambda^+).
\]
\end{defn}

\begin{rem}
 The construction of all the period rings is functorial on perfectoid $C$-Banach pairs $(\Lambda,\Lambda^+)$ with morphisms of $K$-Banach pairs.
\end{rem}

The following lemma will be heavily used in the proof of the purity theorem for horizontal semistable representations (Theorem~\ref{thm:purity for horiztaonlly semistable}).

\begin{lem}\label{lem:Cartesian diagram of period rings}
 Let $(\Lambda',\Lambda'^+)$ be another perfectoid $C$-Banach pair such that $\Lambda'^+$ is an $\calO_{C}$-algebra, and  let $f^+\colon \Lambda^+\hra \Lambda'^+$ be an injective $\calO_K$-algebra homomorphism with $p$-torsion free cokernel.
Then the natural commutative diagram
\[
 \xymatrix{
\B_{\st,K}^+(\Lambda^+)
\ar[r]\ar[d]
&\B_{\dR,K}^+(\Lambda^+)\ar[d]\\
\B_{\st,K}^+(\Lambda'^+)
\ar[r]
&\B_{\dR,K}^+(\Lambda'^+)
}
\]
is Cartesian with injective vertical maps.
\end{lem}

\begin{proof}
By construction of the period rings, we have a natural commutative diagram as in the statement. We will prove that it is Cartesian.

Let $f\colon \Lambda\ra \Lambda'$ denote the induced $K$-algebra homomorphism from $f^+$ by inverting $p$. We also change the norm on $\Lambda$ (resp.~$\Lambda'$) so that $\Lambda^+$ (resp.~$\Lambda'^+$) becomes the unit ball.
Then $f$ is a closed isometric embedding of $K$-Banach algebras;
by $p$-torsion freeness of $\Cok f^+$, the induced map $\Lambda^+/p^n\Lambda^+\ra \Lambda'^+/p^n\Lambda'^+$ is injective for all $n\in\N$, and thus $f$ is closed and isometric.

The commutative diagram in question only depends on $f^+$ as an $\calO_K$-algebra homomorphism and not on the entire $C$-Banach structures on $\Lambda$ and $\Lambda'$. Hence by replacing the $C$-Banach structure $C\ra \Lambda'$ by precomposing an isometric $K$-algebra automorphism $C\ra C$,
we may further assume $f(\pi_n)=\pi_n$ for every $n\in\N$.

We will explain that Construction~\ref{const:section of theta map} is made compatibly with $f^+\colon \Lambda^+\hra \Lambda'^+$.
Since the map
$\Lambda^+/\pi\Lambda^+\ra \Lambda'^+/\pi\Lambda'^+$ is injective, 
choose a family of elements $(e_i)_{i\in I}$ of $\Lambda^+$ and a family of elements $(e_j')_{j\in J}$ of $\Lambda'^{+}$ such that
the images of $e_i$'s in $\Lambda^+/\pi\Lambda^+$ form a $k$-basis of $\Lambda^+/\pi\Lambda^+$ and such that the images of $f^+(e_i)$'s and $e_j'$'s in $\Lambda'^+/\pi\Lambda'^+$ form a $k$-basis of $\Lambda'^+/\pi\Lambda'^+$;
this is possible by lifting $k$-bases of $\Lambda^+/\pi\Lambda^+\subset \Lambda'^{+}/\pi\Lambda'^{+}$.
We also choose a lift $\widetilde{e}_i\in\A_{\inf,K}(\Lambda^+)$ of $e_i$ with respect to $\theta_{\A_{\inf,K}(\Lambda^+)}$ and a lift $\widetilde{e}_j'\in\A_{\inf,K}(\Lambda'^+)$ of $e_j'$ with respect to $\theta_{\A_{\inf,K}(\Lambda'^+)}$.

Consider the $K$-linear maps
\begin{align*}
 &s\colon \Lambda\ra \A_{\inf,K}(\Lambda^+)[p^{-1}];\quad
\sum_{i\in I}a_ie_i\mapsto\sum_{i\in I}a_i\widetilde{e}_i,\quad\text{and}\\
 &s'\colon \Lambda'\ra \A_{\inf,K}(\Lambda'^+)[p^{-1}];\quad\!\!
\sum_{i\in I}a_if(e_i)+\sum_{j\in J}b_je_j'\mapsto\sum_{i\in I}a_i\A_{\inf,K}(f)(\widetilde{e}_i)+\sum_{j\in J}b_j\widetilde{e}_j'.
\end{align*}
Then we have $\A_{\inf,K}(f)\circ s=s'\circ f$.

Since $f(\pi_n)=\pi_n$ for every $n\in\N$, we have $\A_{\inf,K}(f)([\pi^\flat])=[\pi^\flat]$.
Hence we apply Construction~\ref{const:section of theta map} to the maps $s$ and $s'$ with $v=\pi^{-1}([\pi^\flat]-\pi)$ and obtain the commutative diagram
\[
\xymatrix{
\B_{\dR,K}^+(\Lambda^+)\ar[d]\ar[r]^{\widetilde{\theta}_{s,v}}_\cong
&\Lambda[[X]]\ar[d]\\
\B_{\dR,K}^+(\Lambda'^+)\ar[r]^{\widetilde{\theta}_{s',v}}_\cong 
& \Lambda'[[X]].
} 
\]

By these maps and Proposition~\ref{Prop: Colmez's description of semistable period ring},
it suffices to show that the commutative diagram
\[
\xymatrix{
 \Lambda\langle X\rangle[\log (1+X)]\ar[r]\ar[d]
& \Lambda[[X]]\ar[d]\\
 \Lambda'\langle X\rangle[\log (1+X)]\ar[r]
& \Lambda'[[X]]
}
\]
is Cartesian.
This follows from Lemma~\ref{lem: BC Lemma 6.3.1} below.
\end{proof}

 \begin{lem}\label{lem: BC Lemma 6.3.1}
  Let $A$ be a $\Q_p$-Banach algebra and $B$ a closed subalgebra of $A$.
If 
\[
 g(x)\in \bigoplus_{i=0}^h A\langle X\rangle\log (1+X)^i
\]
 satisfies $g(x)\in B[[X]]$,
then 
\[
 g(x)\in \bigoplus_{i=0}^h B\langle X\rangle\log (1+X)^i.
\]
 \end{lem}

\begin{proof}
 This is stated and proved in the proof of \cite[Lemme 6.3.1]{Berger-Colmez}.
\end{proof}

\begin{rem}
Assume that  $\Lambda$ is sympathetic (see \cite[p.356]{Colmez} for definition).
Under this assumption, Colmez shows that the natural sequence
\[
 0\ra \Q_p\ra \B_{\max}(\Lambda^\circ)^{\varphi=1}\ra \B_{\dR}(\Lambda^\circ)/\B_{\dR}^+(\Lambda^\circ)\ra 0
\]
is exact \cite[Proposition 9.25]{Colmez}. The sequence is called the fundamental exact sequence. The exactness does not hold for perfectoid $C$-Banach pairs $(\Lambda,\Lambda^+)$ in general. See also Proposition~\ref{prop:fundamental exact sequence}.
\end{rem}

\section{Preliminaries on base rings}\label{section:base rings}\hfill

In this section, we review the class of base rings which is considered in \cite{Brinon-crisdeRham}.
As in the previous section, 
let $k$ be a perfect field of characteristic $p$ and set $K_0:=W(k)[p^{-1}]$.
Let $K$ be a totally ramified finite extension of $K_0$ and fix a uniformizer $\pi$ of $K$.
We denote by $C$ the $p$-adic completion of an algebraic closure $\overline{K}$ of $K$.

Let $\calO_K\langle T_1^{\pm 1},\ldots,T_n^{\pm 1}\rangle$ denote
the $p$-adic completion of $\calO_K[T_1^{\pm 1},\ldots,T_n^{\pm 1}]$.

\begin{set-up}[{\cite[p.~7]{Brinon-crisdeRham}}]\label{set-up: Brinon's rings}
Let $\widetilde{R}$ be a ring obtained from $\calO_K\langle T_1^{\pm 1},\ldots,T_n^{\pm 1}\rangle$
by a finite number of iterations of the following operations:
\begin{itemize}
 \item $p$-adic completion of an \'etale extension;
 \item $p$-adic completion of a localization;
 \item completion with respect to an ideal containing $p$.
\end{itemize}
We further assume that either $\calO_K\langle T_1^{\pm 1},\ldots,T_n^{\pm 1}\rangle\ra \widetilde{R}$ has geometrically regular fibers or $\widetilde{R}$ has Krull dimension less than two, and that $k\ra \widetilde{R}/\pi\widetilde{R}$ is geometrically integral.
Let $R$ be an $\widetilde{R}$-algebra such that
\begin{enumerate}
 \item $R$ is an integrally closed domain;
 \item $R$ is finite and flat over $\widetilde{R}$;
 \item $R[p^{-1}]$ is \'etale over $\widetilde{R}[p^{-1}]$;
 \item $K$ is algebraically closed in $R[p^{-1}]$.
\end{enumerate}
\end{set-up}

\begin{defn}[{\cite[p.~9]{Brinon-crisdeRham}}]
 We say that $R$ satisfies (BR) if $R=\widetilde{R}$.
This is a condition on good reduction (bonne r\'eduction), and \cite{Brinon-crisdeRham} often assumes it in the study of (horizontal) crystalline representations.
\end{defn}

Set $R_K:=R[p^{-1}]$. Fix an algebraic closure $\overline{\Frac R}$ of $\Frac R$ containing $\overline{K}$.
Let $\overline{R}$ be the union of finite $R$-subalgebras $R'$ of $\overline{\Frac R}$ such that $R'[p^{-1}]$ is \'etale over $R_K$.
Note that $\overline{R}$ is a normal domain as it can be written as an increasing union of such
(by replacing each $R'$ with its normalization).
Let $\widehat{\overline{R}}$ denote the $p$-adic completion of $\overline{R}$. Then $\widehat{\overline{R}}$ is an $\calO_{C}$-algebra.
Note that $\overline{R}[p^{-1}]$ is the union of finite Galois extensions of $R_K$, and thus
we set\footnote{Our $\calG_{R_K}$ is denoted by $\calG_R$ in \cite{Brinon-crisdeRham}.}
\[
 \calG_{R_K}:=\Gal(\overline{R}[p^{-1}]/R_K).
\]
It is a profinite group acting on $\overline{R}$, and this action extends to $\widehat{\overline{R}}$ continuously.

Let us summarize the properties of $\widehat{\overline{R}}$.

\begin{lem}[{\cite[Proposition 2.0.3]{Brinon-crisdeRham}}]
 The ring $\widehat{\overline{R}}$
is $p$-torsion free.
Moreover, $\overline{R}\ra\widehat{\overline{R}}$ is injective, and
$\overline{R}\cap p\widehat{\overline{R}}=p\overline{R}$.
\end{lem}

We will explain that the ring $\widehat{\overline{R}}[p^{-1}]$ admits a $C$-Banach algebra structure and that $(\widehat{\overline{R}}[p^{-1}],\widehat{\overline{R}})$ is a perfectoid $C$-Banach pair (hence we can apply the results in the previous section).
For $x\in \widehat{\overline{R}}[p^{-1}]$, set
\[
 \lvert  x\rvert=\inf\bigl\{ \lvert a\rvert_{C}
\bigm| \, a\in C, \,x\in a\widehat{\overline{R}}\bigr\}.
\]
Note that we normalize the norm $\lvert\;\rvert_{C}$ on $C$ by $\lvert p\rvert_{C}=p^{-1}$.

\begin{lem}
 The function $\lvert\;\rvert$ is a $C$-algebra norm on $\widehat{\overline{R}}[p^{-1}]$, and 
$\widehat{\overline{R}}[p^{-1}]$ is complete with respect to this norm.
Namely, $\widehat{\overline{R}}[p^{-1}]$ becomes a $C$-Banach algebra with this norm.
Moreover, $\widehat{\overline{R}}$ is an open subring of $\widehat{\overline{R}}[p^{-1}]$, and the induced topology on $\widehat{\overline{R}}$ agrees with the $p$-adic topology (i.e., the topology defined by the ideal $p\widehat{\overline{R}}$).
\end{lem}

\begin{proof}
 By construction, we see $\widehat{\overline{R}}\cap C=\calO_{C}$ inside $\widehat{\overline{R}}[p^{-1}]$. Since $\widehat{\overline{R}}$ is  $p$-torsion free and $p$-adically separated, the function $\lvert\;\rvert$ is a $C$-algebra norm. The $p$-adic completeness of $\widehat{\overline{R}}$ implies that $\widehat{\overline{R}}[p^{-1}]$ is complete with respect to $\lvert\;\rvert$.
Finally, the second assertion follows from the inclusions
\[
p^{m+1} \widehat{\overline{R}}\subset\{x\in \widehat{\overline{R}}[p^{-1}]\mid \lvert x\rvert\leq p^{-(m+1)}\}\subset p^m \widehat{\overline{R}}.
\]
\end{proof}

\begin{lem}\label{lem:uniform}
The $C$-Banach algebra $\widehat{\overline{R}}[p^{-1}]$ is uniform.
\end{lem}

\begin{proof}
Since $p^{-1}\widehat{\overline{R}}$ is a bounded subset, it suffices to show $\bigl(\widehat{\overline{R}}[p^{-1}]\bigr)^\circ\subset p^{-1}\widehat{\overline{R}}$.

 First we prove $\bigl(\widehat{\overline{R}}[p^{-1}]\bigr)^\circ\cap \overline{R}[p^{-1}]\subset p^{-1}\overline{R}$.
Take $x\in \bigl(\widehat{\overline{R}}[p^{-1}]\bigr)^\circ\cap \overline{R}[p^{-1}]$.
Then the sequence $(p^nx^n)_n$ converges to zero as $n\to\infty$.
In particular, $(px)^n\in\overline{R}$ for some $n$. Since $\overline{R}$ is a normal domain, we have $px\in\overline{R}$ and thus $x\in p^{-1}\overline{R}$.

For $x\in \bigl(\widehat{\overline{R}}[p^{-1}]\bigr)^\circ$, write $x=u+v$ with $u\in \overline{R}[p^{-1}]$ and $v\in\widehat{\overline{R}}$.
Since $\bigl(\widehat{\overline{R}}[p^{-1}]\bigr)^\circ$ is a ring containing $\widehat{\overline{R}}$, we have $u\in \bigl(\widehat{\overline{R}}[p^{-1}]\bigr)^\circ$.
Hence $u\in p^{-1}\overline{R}$ by what we have proved, and we conclude $x\in p^{-1}\widehat{\overline{R}}$.
\end{proof}

\begin{lem}\label{lem:ring of integral elements}
 The subring $\widehat{\overline{R}}$ of $\widehat{\overline{R}}[p^{-1}]$ is a ring of integral elements.
\end{lem}

\begin{proof}
We need to show that $\widehat{\overline{R}}$ is integrally closed in $\widehat{\overline{R}}[p^{-1}]$.

Let $x\in\widehat{\overline{R}}[p^{-1}]$ be an element that is integral over $\widehat{\overline{R}}$ and suppose that $x$ satisfies
\[
 x^m+a_{m-1}x^{m-1}+\dots+a_0=0
\]
for some $a_0,\ldots,a_{m-1}\in \widehat{\overline{R}}$. For each $l\in \N$, there exist
$y_l\in\overline{R}[p^{-1}]$ and $z_l\in p^l\widehat{\overline{R}}$ such that
$x=y_l+z_l$. Similarly, for each $i=0,\ldots,m-1$. there exists $a_{i,l}\in \overline{R}$ such that $a_i-a_{i,l}\in p^l\widehat{\overline{R}}$. We compute
\begin{align*}
& \lvert y_l^m+a_{m-1,l}y_l^{m-1}+\dots+a_{0,l}\rvert\\
&=\lvert (y_l^m+a_{m-1,l}y_l^{m-1}+\dots+a_{0,l})-(x^m+a_{m-1}x^{m-1}+\dots+a_0)\rvert\\
&\leq \max\bigl\{\lvert (y_l^m+a_{m-1,l}y_l^{m-1}+\dots+a_{0,l})-(x^m+a_{m-1,l}x^{m-1}+\dots+a_{0,l})\rvert, \\
&\phantom{maxmax}\lvert (x^m+a_{m-1,l}x^{m-1}+\dots+a_{0,l})-(x^m+a_{m-1}x^{m-1}+\dots+a_0)\rvert\bigr\}\\
&\leq \lvert p^l\rvert\,\max\bigl\{1,\lvert x\rvert^{m-1}\bigr\}.
\end{align*}
Since $\{y\in \widehat{\overline{R}}[p^{-1}]\mid \lvert y\rvert<1\}\subset \widehat{\overline{R}}$ and $\overline{R}[p^{-1}]\cap \widehat{\overline{R}}=\overline{R}$, we see that 
$y_l^m+a_{m-1,l}y_l^{m-1}+\dots+a_{0,l}\in \overline{R}$ for $l\gg 0$.
For such $l$, we have $y_l\in \overline{R}$ by normality of $\overline{R}$ and thus $x=y_l+z_l\in\widehat{\overline{R}}$.
\end{proof}

\begin{lem}
 The $C$-Banach pair $(\widehat{\overline{R}}[p^{-1}],\widehat{\overline{R}})$ is perfectoid.
\end{lem}

\begin{proof}
It remains to prove that the Frobenius on $\widehat{\overline{R}}/p\widehat{\overline{R}}$ is surjective.
 Note that $\widehat{\overline{R}}/p\widehat{\overline{R}}=\overline{R}/p\overline{R}$.  Let $\overline{x}\in \overline{R}/p\overline{R}$ and choose a lift $x\in\overline{R}$ of $\overline{x}$.
Set $R':=R[x][T]/(T^{p^2}+pT-x)$. Then $R'$ is finite over $R$ and $R'[p^{-1}]$ is \'etale over $R_K$. Choose an embedding of $R'$ into $\overline{\Frac R}$ and let $y$ denote the image of $T$. Then $y\in \overline{R}$ and it satisfies $(y^p)^p\equiv x\bmod{p\overline{R}}$.
Hence the Frobenius on $\overline{R}/p\overline{R}$ is surjective. 
\end{proof}

\section{Horizontal crystalline, semistable, and de Rham representations}\label{section:horizontal semistable representations}

In this section, we define horizontal crystalline, semistable, and de Rham representations of $\calG_{R_K}$, and study their basic properties. We mainly follow \cite[\S 8]{Brinon-crisdeRham}, where horizontal crystalline and de Rham representations are studied.

\subsection{Admissible representations}
We review the notion of admissible representations.
Let $E$ be a finite extension of $\Q_p$ and $G$ a topological group.

\begin{defn}
 An \emph{$E$-representation} of $G$ is a finite-dimensional $E$-vector space equipped with a continuous action of $G$.
A morphism of $E$-representations of $G$ is an $E$-linear map commuting with actions of $G$.
We denote the category of $E$-representations of $G$ by $\Rep_E(G)$.
It is naturally a Tannakian category.
\end{defn}

\begin{rem}
Note that \cite[\S 1]{Fontaine-exposeIII} deals with an abstract field $E$ and an abstract group $G$ but that properties discussed there continue to hold in our setting.
\end{rem}

\begin{defn}[cf.~{\cite[1.3]{Fontaine-exposeIII}}]
\hfill
\begin{enumerate}
 \item  An \emph{$(E,G)$-ring} is an $E$-algebra $B$ equipped with an $E$-linear action of $G$.
 \item Let $B$ be an $(E,G)$-ring and let $V\in\Rep_E(G)$.
We set
\[
 D_B(V):=(B\otimes_{E}V)^{G},
\]
and we let 
\[
 \alpha_B(V)\colon B\otimes_{B^G}D_B(V)\ra B\otimes_EV
\]
denote the $B$-linear map induced from the inclusion 
$D_B(V)\subset B\otimes_EV$.
We say that $V$ is \emph{$B$-admissible}
if $\alpha_B(V)$ is an isomorphism of $B$-modules.\\
\end{enumerate}
\end{defn}

\begin{lem}\label{lem: admissiblity of two period rings}
 Let $G'$ be another topological group and let $\alpha\colon G'\ra G$ be a continuous group homomorphism.
Let $B$ be an $(E,G)$-ring and $B'$ an $(E,G')$-ring equipped with an $E$-algebra homomorphism $\beta\colon B\ra B'$ satisfying
\[
 \beta(\alpha(g')\cdot b)=g'\cdot \beta(b)\quad \text{for all $g'\in G'$ and $b\in B$}.
\]
Let $V\in \Rep_E(G)$ and let $V|_{G'}$ denote the corresponding $E$-representation of $G'$.
If $V$ is $B$-admissible and $\alpha_{B'}(V|_{G'})$ is injective, then $V|_{G'}$ is $B'$-admissible, and
the natural map
\[
 (B')^{G'}\otimes_{B^G}D_B(V)\ra D_{B'}(V|_{G'})
\]
is an isomorphism. 
\end{lem}

\begin{proof}
Consider the commutative diagram
\[
 \xymatrix{
B\otimes_{B^G}D_B(V)\ar[rrr]^{\alpha_B(V)}_{\cong}\ar[d]
&&& B\otimes_EV\ar[d]\\
B'\otimes_{(B')^{G'}}D_{B'}(V|_{G'})\ar@{^{(}->}[rrr]^{\alpha_{B'}(V|_{G'})}
&&& B'\otimes_EV.
}
\]
It follows that $V\subset B'\otimes_EV$ is contained in the image of $\alpha_{B'}(V|_{G'})$.
Since $\alpha_{B'}(V|_{G'})$ is injective and $B'$-linear, it is an isomorphism.
We then have a $G'$-equivariant isomorphism
\[
 B'\otimes_{B}(B\otimes_{B^G}D_B(V))\stackrel{\cong}{\lra}B'\otimes_{(B')^{G'}}D_{B'}(V|_{G'}).
\]
Taking $G'$-invariants yields an isomorphism
$(B')^{G'}\otimes_{B^G}D_B(V)\stackrel{\cong}{\lra} D_{B'}(V|_{G'})$.
 \end{proof}

\begin{prop}[{\cite[Proposition 3.5]{Brinon-cris}}]\label{prop:tensor and SES of admissible representations}
Let $B$ be an $(E,G)$-ring and assume that $B$ is faithfully flat over $F:=B^G$
and that $\alpha_B(V)$ is injective for every $V\in\Rep_E(G)$.
Let $V_1,V_2\in\Rep_E(G)$.
\begin{enumerate}
 \item If $V_1$ and $V_2$ are $B$-admissible, then
so is $V_1\otimes_EV_2$ and the natural map
$D_B(V_1)\otimes_FD_B(V_2)\ra D_B(V_1\otimes_EV_2)$ is an isomorphism.
 \item Let $V\in\Rep_E(G)$ be an extension of $V_2$ by $V_1$.
If $V$ is $B$-admissible, then so are $V_1$ and $V_2$.
Moreover, the sequence $0\ra D_B(V_1)\ra D_B(V)\ra D_B(V_2)\ra 0$ is exact.
 \item Let $h\in\N$ and let $V\in\Rep_E(G)$ be $B$-admissible.
Then $\bigwedge_E^hV$ is also $B$-admissible and the natural map
$\bigwedge_F^hD_B(V)\ra D_B(\bigwedge_E^hV)$ is an isomorphism.
\end{enumerate}
\end{prop}

\begin{prop}[{\cite[Proposition 3.7]{Brinon-cris}}]\label{prop:dual of admissible representations}
Let $B$ be an $(E,G)$-ring satisfying the following conditions:
\begin{enumerate}
 \item $F:=B^G$ is Noetherian and $B$ is faithfully flat over $F$;
 \item $\alpha_B(V)$ is injective for every $V\in\Rep_E(G)$;
 \item if $V\in\Rep_E(G)$ is $B$-admissible and $\dim_EV=1$,
then the dual representation $V^\vee$ is also $B$-admissible.
\end{enumerate}
Then for every $B$-admissible representation $V\in\Rep_E(G)$,
the dual $V^\vee$ is also $B$-admissible and the natural map
\[
 D_B(V^\vee)\ra \Hom_F(D_B(V),F)
\]
is an isomorphism.
\end{prop}

\subsection{Horizontal crystalline, semistable, and de Rham representations}

Let $R$ be an $\calO_K$-algebra satisfying the conditions in Set-up~\ref{set-up: Brinon's rings}. Recall that we set $\calG_{R_K}:=\Gal(\overline{R_K}/R_K)$ and that  $(\widehat{\overline{R}}[p^{-1}],\widehat{\overline{R}})$ is a perfectoid $C$-Banach pair.

\begin{defn}
 Set
\begin{align*}
\rmB_{\max}^\nabla(R)&:=\B_{\max}(\widehat{\overline{R}}),\quad
 \rmB_{\max,K}^\nabla(R):=\B_{\max,K}(\widehat{\overline{R}}),\\
\rmB_{\cris}^\nabla(R)&:=\B_{\cris}(\widehat{\overline{R}}),\quad\\
\rmB_{\st}^\nabla(R)&:=\B_{\st}(\widehat{\overline{R}}),\quad
\rmB_{\st,K}^\nabla(R):=\B_{\st,K}(\widehat{\overline{R}}),\\
\rmB_{\dR}^\nabla(R)&:=\B_{\dR}(\widehat{\overline{R}})=\B_{\dR,K}(\widehat{\overline{R}}), \quad\text{and}\quad \rmB_{\dR}^{\nabla+}(R):=\B_{\dR}^+(\widehat{\overline{R}}).
\end{align*}
By functoriality of the construction, all these period rings have an action of $\calG_{R_K}$.
In particular, they are $(\Q_p,\calG_{R_K})$-rings.
To simplify the notation, for $V\in\Rep_{\Q_p}(\calG_{R_K})$, we write $D_{\max}^\nabla(V)$ for $D_{\rmB_{\max}^\nabla(R)}(V)$ and $\alpha_{\max}^\nabla(V)$ for $\alpha_{\rmB_{\max}^\nabla(R)}(V)$. We use similar conventions for other period rings.
 \end{defn}

\begin{rem}\hfill
\begin{enumerate}
 \item
 By definition, our period rings $\rmB_{\max}^\nabla(R)$, $\rmB_{\cris}^\nabla(R)$, and $\rmB_{\dR}^\nabla(R)$ coincide with the horizontal period rings defined in \cite[D\'efinitions 5.1.2, 6.1.2]{Brinon-crisdeRham} and denoted by the same symbols.
 \item The period rings $\rmB_{\max}^\nabla(R)$ and $\rmB_{\cris}^\nabla(R)$ admit the Frobenius $\varphi$. Similarly, $\rmB_{\st}^\nabla(R)$ admits the Frobenius $\varphi$ and the monodromy operator $N$ satisfying $N\varphi=p\varphi N$.
 \item
Recall the following relations among period rings:
\begin{align*}
 &K\otimes_{K_0}\rmB_{\max}^\nabla(R)=\rmB_{\max,K}^\nabla(R),\quad
K\otimes_{K_0}\rmB_{\st}^\nabla(R)=\rmB_{\st,K}^\nabla(R),\\
& (\rmB_{\st}^\nabla(R))^{N=0}=\rmB_{\max}^\nabla(R), \quad\text{and}\quad
\varphi(\rmB_{\max}^\nabla(R))\subset \rmB_{\cris}^\nabla(R)\subset \rmB_{\max}^\nabla(R).
\end{align*}
  \item The action of $\calG_{R_K}$ on $\rmB_{\max}^\nabla(R)$ and $\rmB_{\cris}^\nabla(R)$ commutes with the Frobenius.
In particular, for $V\in\Rep_{\Q_p}(\calG_{R_K})$, the modules $D_{\max}^\nabla(V)$ and $D_{\cris}^\nabla(V)$ are equipped with the Frobenius $\varphi$ by restricting $\varphi\otimes\id$ on 
$\rmB_{\max}^\nabla(R)\otimes_{\Q_p}V$ and $\rmB_{\cris}^\nabla(R)\otimes_{\Q_p}V$.
The maps $\alpha_{\max}^\nabla(V)$ and $\alpha_{\cris}^\nabla(V)$ are $\varphi$-equivariant.
We also have
\[
 \varphi(D_{\max}^\nabla(V))\subset D_{\cris}^\nabla(V)\subset D_{\max}^\nabla(V).
\]
 \item The action of $\calG_{R_K}$ on $\rmB_{\st}^\nabla(R)$ commutes with the Frobenius and the monodromy operator.
In particular, for $V\in\Rep_{\Q_p}(\calG_{R_K})$, the module $D_{\st}^\nabla(V)$ is equipped with the Frobenius $\varphi$ and the monodromy operator $N$ by restricting $\varphi\otimes\id$ and $N\otimes\id$ on $\rmB_{\st}^\nabla(R)\otimes_{\Q_p}V$.
These endomorphisms satisfy
\[
 N\varphi=p\varphi N,\quad\text{and}\quad (D_{\st}^\nabla(V))^{N=0}=D_{\max}^\nabla(V).
\]
\end{enumerate}
\end{rem}

\begin{prop}[{\cite[Corollaire 5.3.7]{Brinon-crisdeRham}}]
 We have
\[
 \bigl(\rmB_{\dR}^\nabla(R)\bigr)^{\calG_{R_K}}=K.
\]
\end{prop}

\begin{cor}
 We have
\begin{align*}
\bigl(\rmB_{\max,K}^\nabla(R)\bigr)^{\calG_{R_K}}
&=\bigl(\rmB_{\st,K}^\nabla(R)\bigr)^{\calG_{R_K}}
=K;\\
\bigl(\rmB_{\cris}^\nabla(R)\bigr)^{\calG_{R_K}}
&=\bigl(\rmB_{\max}^\nabla(R)\bigr)^{\calG_{R_K}}
=\bigl(\rmB_{\st}^\nabla(R)\bigr)^{\calG_{R_K}}=K_0.
\end{align*}
\end{cor}

\begin{proof}
Since we have $\calG_{R_K}$-equivariant $K$-algebra homomorphisms
\[
 K\otimes_{K_0} \rmB_{\max}^\nabla(R)\stackrel{\cong}{\lra} \rmB_{\max,K}^\nabla(R)\hra\rmB_{\dR}^\nabla(R),
\]
we obtain
\[
 \bigl(\rmB_{\max,K}^\nabla(R)\bigr)^{\calG_{R_K}}=K,\quad\text{and}\quad
\bigl(\rmB_{\max}^\nabla(R)\bigr)^{\calG_{R_K}}=K_0.
\]
By the inclusions $\varphi(\rmB_{\max}^\nabla(R))\subset \rmB_{\cris}^\nabla(R)\subset \rmB_{\max}^\nabla(R)$, we have $\bigl(\rmB_{\cris}^\nabla(R)\bigr)^{\calG_{R_K}}=K_0$.
We can prove the remaining assertions in the same way.
\end{proof}

The following is clear from what we have discussed:

\begin{lem}\label{lem:basic properties of D}
Let $V\in\Rep_{\Q_p}(\calG_{R_K})$.
 \begin{enumerate}
 \item $D_{\max}^\nabla(V)$ and $D_{\st}^\nabla(V)$ are $K_0$-vector spaces, and
$D_{\max,K}^\nabla(V)$, $D_{\st,K}^\nabla(V)$, and $D_{\dR}^\nabla(V)$ are $K$-vector spaces.
  \item We have 
\[
 K\otimes_{K_0}D_{\max}^\nabla(V)\cong D_{\max,K}^\nabla(V),
\quad\text{and}\quad
 K\otimes_{K_0}D_{\st}^\nabla(V)\cong D_{\st,K}^\nabla(V).
\]
Moreover, $\alpha_{\max}^\nabla(V)$ is an isomorphism if and only if 
$\alpha_{\max,K}^\nabla(V)$ is an isomorphism.
Similarly, $\alpha_{\st}^\nabla(V)$ is an isomorphism if and only if 
$\alpha_{\st,K}^\nabla(V)$ is an isomorphism.
 \item There are natural injective maps of $K$-vector spaces:
\[
 D_{\max,K}^\nabla(V)\hra 
 D_{\st,K}^\nabla(V)\hra
 D_{\dR}^\nabla(V).
\]
 \end{enumerate}
\end{lem}

\begin{lem}
For every $V\in\Rep_{\Q_p}(\calG_{R_K})$, the maps 
\[
 \alpha_{\cris}^\nabla(V), \quad\alpha_{\max}^\nabla(V), \quad\alpha_{\max,K}^\nabla(V), \quad\alpha_{\st}^\nabla(V), \quad\alpha_{\st,K}^\nabla(V), \quad\alpha_{\dR}^\nabla(V)
\]
 are all injective.
Moreover, we have
\begin{align*}
  \dim_{K_0}D_{\max}^\nabla(V)=\dim_KD_{\max,K}^\nabla(V)
&\leq 
\dim_{K_0}D_{\st}^\nabla(V)=\dim_KD_{\st,K}^\nabla(V)\\
&\leq \dim_KD_{\dR}^\nabla(V)\\
&\leq \dim_{\Q_p}V.
\end{align*}
\end{lem}

\begin{proof}
The map $\alpha_{\dR}^\nabla(V)$ is injective by \cite[Proposition 8.2.9]{Brinon-crisdeRham}.
Consider the commutative diagram
\[
 \xymatrix{
 K\otimes_{K_0}( \rmB_{\max}^\nabla(R)\otimes_{K_0}D_{\max}^\nabla(V))
\ar@{=}[d]\ar[rrr]^{K\otimes_{K_0}\alpha_{\max}^\nabla(V)}
&&& K\otimes_{K_0}(\rmB_{\max}^\nabla(R)\otimes_{\Q_p}V)\ar@{=}[d]\\
  \rmB_{\max,K}^\nabla(R)\otimes_{K}D_{\max,K}^\nabla(V)
\ar@{^{(}->}[d]\ar[rrr]^{\alpha_{\max,K}^\nabla(V)}
&&& \rmB_{\max,K}^\nabla(R)\otimes_{\Q_p}V\ar@{^{(}->}[d]\\
  \rmB_{\dR}^\nabla(R)\otimes_{K}D_{\dR}^\nabla(V)
\ar[rrr]^{\alpha_{\dR}^\nabla(V)}
&&& \rmB_{\dR}^\nabla(R)\otimes_{\Q_p}V.
}
\]
The vertical equalities follow from Lemma~\ref{lem:basic properties of D}(ii).
Hence the injectivity of $\alpha_{\dR}^\nabla(V)$ implies that of
 $\alpha_{\max}^\nabla(V)$ and $\alpha_{\max,K}^\nabla(V)$.
The same argument works for the injectivity of $\alpha_{\st}^\nabla(V)$ and $\alpha_{\st,K}^\nabla(V)$. The injectivity of $\alpha_{\cris}^\nabla(V)$ follows from that of $\alpha_{\max}^\nabla(V)$ and the inclusion $\rmB_{\cris}^\nabla(R)\subset \rmB_{\max}^\nabla(R)$ by a similar argument.
The second assertion follows from Lemma~\ref{lem:basic properties of D}(iii).
\end{proof}

\begin{rem}
 The injectivity of $\alpha_{\cris}^\nabla(V)$ is also proved in \cite[Proposition 8.2.9]{Brinon-crisdeRham} under the additional assumption (BR) on $R$.
\end{rem}

\begin{lem}\label{lem:equiv of max and cris admissibility}
 For $V\in\Rep_{\Q_p}(\calG_{R_K})$, $\alpha_{\max}^\nabla(V)$ is an isomorphism if and only if $\alpha_{\cris}^\nabla(V)$ is an isomorphism.
Similarly, $\alpha_{\st}^\nabla(V)$ is an isomorphism if and only if $\alpha_{\rmB_{\cris}^\nabla(R)[\log \frac{[p^\flat]}{p}]}(V)$ is an isomorphism.
\end{lem}

\begin{proof}
The sufficiency follows from Lemma~\ref{lem: admissiblity of two period rings}.
Since the proof of the necessity of the two assertions is similar, we deal with the necessity for the first. 
Assume that $\alpha_{\max}^\nabla(V)$ is an isomorphism.
Then the Frobenius $\varphi\colon \rmB_{\max}^\nabla(R)\ra \rmB_{\max}^\nabla(V)$ induces the commutative diagram
\[
 \xymatrix{
\rmB_{\max}^\nabla(R)\otimes_{K_0}D_{\max}^\nabla(V)
\ar[d]_{\varphi\otimes_{\varphi}\varphi}\ar[rr]^{\alpha_{\max}^\nabla(V)}_{\cong}
&&
\rmB_{\max}^\nabla(R)\otimes_{\Q_p}V\ar[d]^{\varphi\otimes\id}\\
\rmB_{\max}^\nabla(R)\otimes_{K_0}D_{\max}^\nabla(V)
\ar[rr]^{\alpha_{\max}^\nabla(V)}_{\cong}
&&
\rmB_{\max}^\nabla(R)\otimes_{\Q_p}V.
}
\]
Since $\varphi\colon \rmB_{\max}^\nabla(R)\ra \rmB_{\max}^\nabla(R)$ factors through $\rmB_{\cris}^\nabla(R)\subset \rmB_{\max}^\nabla(R)$, the above diagram also factors as
\[
 \xymatrix{
\rmB_{\max}^\nabla(R)\otimes_{K_0}D_{\max}^\nabla(V)
\ar[d]\ar[rr]^{\alpha_{\max}^\nabla(V)}_{\cong}
&&
\rmB_{\max}^\nabla(R)\otimes_{\Q_p}V\ar[d]\\
\rmB_{\cris}^\nabla(R)\otimes_{K_0}D_{\cris}^\nabla(V)
\ar[d]\ar@{^{(}->}[rr]^{\alpha_{\cris}^\nabla(V)}
&&
\rmB_{\cris}^\nabla(R)\otimes_{\Q_p}V\ar[d]\\
\rmB_{\max}^\nabla(R)\otimes_{K_0}D_{\max}^\nabla(V)
\ar[rr]^{\alpha_{\max}^\nabla(V)}_{\cong}
&&
\rmB_{\max}^\nabla(R)\otimes_{\Q_p}V.
}
\]
In particular, the image of $\alpha_{\cris}^\nabla(V)$ contains $V$.
Since $\alpha_{\cris}^\nabla(V)$ is injective and $\rmB_{\cris}^\nabla(R)$-linear, it is an isomorphism.
\end{proof}

\begin{lem}\label{lem: semistable with zero monodromoy implies crystalline}
 For $V\in \Rep_{\Q_p}(\calG_{R_K})$, $\alpha_{\max}^\nabla(V)$ is an isomorphism if and only if $\alpha_{\st}^\nabla(V)$ is an isomorphism and the monodromy operator $N$ on $D_{\st}^\nabla(V)$ is zero.
\end{lem}

\begin{proof}
 The necessity follows from Lemma~\ref{lem: admissiblity of two period rings}
and $(D_{\st}^\nabla(V))^{N=0}=D_{\max}^\nabla(V)$.
We prove the sufficiency. Note that the multiplication map
\[
 \rmB_{\st}^\nabla(R)\otimes_{K_0}\rmB_{\st}^\nabla(R)\ra \rmB_{\st}^\nabla(R)
\]
is compatible with $N\otimes \id+\id\otimes N$ on the source and $N$ on the target.
Since $\alpha_{\st}^\nabla(V)$ is an isomorphism,  this yields the following diagram with exact rows
\[
 \xymatrix{
0 \ar[r]&
\rmB_{\max}^\nabla\otimes_{\Q_p}V\ar[r]&
\rmB_{\st}^\nabla\otimes_{\Q_p}V\ar[r]^{N\otimes \id}&
\rmB_{\st}^\nabla\otimes_{\Q_p}V\ar[r]& 0\\
0 \ar[r]&
\Ker (f)\ar[r]\ar[u]_{\cong}&
\rmB_{\st}^\nabla\otimes_{K_0}D_{\st}^\nabla(V)\ar[r]^{f}\ar[u]_{\cong}^{\alpha_{\st}^\nabla(V)}&
\rmB_{\st}^\nabla\otimes_{K_0}D_{\st}^\nabla(V)\ar[r]\ar[u]_{\cong}^{\alpha_{\st}^\nabla(V)}& 0,
}
\]
where $\rmB_{\max}^\nabla=\rmB_{\max}^\nabla(R)$, $\rmB_{\st}^\nabla=\rmB_{\st}^\nabla(R)$, and $f$ denotes $N\otimes \id+\id\otimes N$.
Since the monodromy operator $N$ on $D_{\st}^\nabla(V)$ is assumed to be zero, 
we conclude $\Ker f=\rmB_{\max}^\nabla(R)\otimes_{K_0}D_{\max}^\nabla(V)$ and the left vertical isomorphism is $\alpha_{\max}^\nabla(V)$.
\end{proof}

\begin{defn}\label{def:horizontal de Rham representations}
 Let $V\in\Rep_{\Q_p}(\calG_{R_K})$.
\begin{enumerate}
 \item We say that $V$ is \emph{horizontal crystalline} if $\alpha_{\max}^\nabla(V)$ is an isomorphism.
 \item We say that $V$ is \emph{horizontal semistable} if $\alpha_{\st}^\nabla(V)$ is an isomorphism.
 \item We say that $V$ is \emph{horizontal de Rham} if $\alpha_{\dR}^\nabla(V)$ is an isomorphism.
\end{enumerate}
See Lemma~\ref{lem:basic properties of D}(ii), \ref{lem:equiv of max and cris admissibility}, and \ref{lem: semistable with zero monodromoy implies crystalline} for equivalent definitions.
By Lemma~\ref{lem: admissiblity of two period rings}, a horizontal crystalline representation is horizontal semistable, and a horizontal semistable representation is horizontal de Rham. 
\end{defn}

\begin{rem}
 Our definition of horizontal de Rham representations coincides with that of \cite[p.~117]{Brinon-crisdeRham}. By Lemma~\ref{lem:equiv of max and cris admissibility}, our definition of horizontal crystalline representations coincides with the definition of horizontal $R_0$-crystalline representations in \cite[p.~117]{Brinon-crisdeRham}.
\end{rem}

\begin{rem}\label{rem:tensor of horizontal de Rham representations}
Observe that the period rings $\rmB_{\max}^\nabla(R), \rmB_{\st}^\nabla(R)$, and $\rmB_{\dR}^\nabla(R)$ all satisfy the assumptions of Proposition~\ref{prop:tensor and SES of admissible representations}.
Hence the full subcategory of horizontal crystalline (resp.~semistable, resp.~de Rham) representations of $\calG_{R_K}$ is stable under taking tensor products, subquotients, and exterior products.
\end{rem}

\begin{lem}\label{lem:dual of horizontal de Rham}
The full subcategory of horizontal de Rham representations of $\calG_{R_K}$ is stable under
taking duals.
If $R$ satisfies Condition (BR), then 
the full subcategory of horizontal crystalline (resp. semistable) representations of $\calG_{R_K}$ is stable under taking duals.
\end{lem}

\begin{proof}
 By Proposition~\ref{prop:dual of admissible representations}, it suffices to prove the assertion for one-dimensional representations.
For horizontal de Rham representations (resp.~horizontal crystalline representations under Condition (BR)), it is \cite[Th\'eor\`eme 8.4.2]{Brinon-crisdeRham}.
In general, if $D_{\st}^\nabla(V)$ is one-dimensional over $K_0$ for $V\in\Rep_{\Q_P}(\calG_{R_K})$,  then $N=0$ by the relation $N\varphi=p\varphi N$.
In particular, every one-dimensional horizontal semistable representation is horizontal crystalline, and thus so is its dual.
\end{proof}

\begin{prop}\label{prop:fundamental exact sequence}
The sequence
\[
 0\ra \Q_p\ra \rmB_{\max}^\nabla(R)^{\varphi=1}\ra 
 \rmB_{\dR}^\nabla(R)/ \rmB_{\dR}^{\nabla+}(R)\ra 0
\]
is exact.
\end{prop}

\begin{proof}
Since $\varphi(\rmB_{\max}^\nabla(R))\subset \rmB_{\cris}^\nabla(R)\subset \rmB_{\max}^\nabla(R)$, we get 
\[
 \rmB_{\cris}^\nabla(R)^{\varphi=1}=\rmB_{\max}^\nabla(R)^{\varphi=1}.
\]
 Hence the assertion is \cite[Proposition 6.2.24]{Brinon-crisdeRham}.
\end{proof}

The exact sequence in the proposition is called \emph{the fundamental exact sequence}.

\subsection{Functoriality}\label{subsection:functoriality}

Let $R$ be an $\calO_K$-algebra satisfying the conditions in Set-up~\ref{set-up: Brinon's rings}
and fix an algebraic closure $\overline{\Frac R}$ of $\Frac R$ as before.
Let us discuss functorial aspects of our formalism in two special cases.

Let $R'$ be another $\calO_R$-algebra satisfying the conditions in Set-up~\ref{set-up: Brinon's rings} and let $f\colon R\ra R'$ be an $\calO_K$-algebra homomorphism.
Recall 
\[
 \calG_{R_K}=\Gal(\overline{R}[p^{-1}]/R_K)= \pi_1^{\et}(\Spec R_K,\overline{\Frac R}).
\]
Note that $R_K\ra R'_K\ra \overline{\Frac R'}$ defines another geometric point of $\Spec R_K$.
Choose a path on $\Spec R_K$ from $\overline{\Frac R'}$ to $\overline{\Frac R}$ (cf.~\cite[Expos\'e V, 7]{SGA1}).
Concretely, this is equivalent to choosing an $\calO_K$-algebra homomorphism $\overline{f}\colon \overline{R}\ra \overline{R'}$ extending $f$. 
The path defines a group homomorphism
\begin{align*}
  \calG_{R'_K}&= \pi_1^{\et}(\Spec R'_K,\overline{\Frac R'})\stackrel{\Spec f_{\ast}}{\lra} \pi_1^{\et}(\Spec R_K,\overline{\Frac R'})\\
&\lra \pi_1^{\et}(\Spec R_K,\overline{\Frac R})= \calG_{R_K}.
\end{align*}
In what follows (e.g. Proposition~\ref{prop:horizontal semistable plus crystalline at one point}), we always choose a path so that $\overline{f}$ is an $\calO_{\overline{K}}$-algebra homomorphism.
Then $\overline{f}$ gives rise to a  homomorphism $(\widehat{\overline{R}}[p^{-1}],\widehat{\overline{R}})\ra (\widehat{\overline{R'}}[p^{-1}],\widehat{\overline{R'}})$ of $C$-Banach pairs and 
induces homomorphisms 
\[
 \rmB_{\max}^\nabla(R)\ra\rmB_{\max}^\nabla(R'),\quad
 \rmB_{\st}^\nabla(R)\ra\rmB_{\st}^\nabla(R'),\quad\text{and}\quad
 \rmB_{\dR}^\nabla(R)\ra\rmB_{\dR}^\nabla(R').
\]
Moreover they are compatible with $\calG_{R'_K}\ra \calG_{R_K}$.

\begin{lem}
If a $\Q_p$-representation $V$ of $\calG_{R_K}$ is horizontal crystalline (resp. semistable, resp. de Rham), then so is the representation $V|_{\calG_{R'_K}}$ of $\calG_{R'_K}$.
\end{lem}

\begin{proof}
 This follows from Lemma~\ref{lem: admissiblity of two period rings}.
\end{proof}

We also consider the case of changing the base field.
Let $L$ be a finite extension of $K$, and set
\[
 R_{\calO_L}:=R\otimes_{\calO_K}\calO_L \quad\text{and}\quad R_L:=R_{\calO_L}[p^{-1}]=R_K\otimes_KL.
\]
Observe that $R_{\calO_L}$ satisfies the conditions in Set-up~\ref{set-up: Brinon's rings} with respect to $L$.
Since $\overline{\Frac R}$ is an algebraic closure of $\Frac R_{\calO_L}$, we have $\overline{R_{\calO_L}}=\overline{R}$. It follows $\calG_{R_L}$ is an open subgroup of $\calG_{R_K}$ of finite index.

\begin{lem}\label{lem:potentially horizontal de Rham is horizontal de Rham}
 Let $V\in\Rep_{\Q_p}(\calG_{R_K})$.
\begin{enumerate}
 \item If $V$ is horizontal crystalline, then the representation $V|_{\calG_{R_L}}$ of $\calG_{R_L}$ is also horizontal crystalline.
Similarly, if $V$ is horizontal semistable, then so is $V|_{\calG_{R_L}}$.
 \item $V$ is horizontal de Rham if and only if $V|_{\calG_{R_L}}$ is horizontal de Rham.
\end{enumerate}
\end{lem}

\begin{proof}
 Part (i) and the necessity of Part (ii) follow from Lemma~\ref{lem: admissiblity of two period rings}. Assume that $V|_{\calG_{R_L}}$ is horizontal de Rham. Using the necessity, we may reduce to the case where $L$ is Galois over $K$. Then $\calG_{R_L}$ is a normal subgroup of $\calG_{R_K}$ with quotient $\Gal(L/K)$.
Since $\rmB_{\dR}^\nabla(R)=\rmB_{\dR}^\nabla(R_{\calO_L})$, the finite-dimensional $L$-vector space $D_{\dR}^\nabla(V|_{\calG_{R_{\calO_L}}})$ has a semilinear $\Gal(L/K)$-action and 
\[
 D_{\dR}^\nabla(V)=D_{\dR}^\nabla(V|_{\calG_{R_L}})^{\Gal(L/K)}.
\]
By Galois descent, we have
\[
 L\otimes_KD_{\dR}^\nabla(V)\stackrel{\cong}{\lra}D_{\dR}^\nabla(V|_{\calG_{R_L}}).
\]
It follows that $\alpha_{\dR}^\nabla(V)=\alpha_{\dR}^\nabla(V|_{\calG_{R_L}})$ and it is an isomorphism by assumption.
\end{proof}

By combining these two cases, we can consider the following setting:

\begin{prop}\label{prop:horizontal semistable plus crystalline at one point}
Let $L$ be a finite extension of $K$ and $R'$ an $\calO_L$-algebra satisfying the conditions in Set-up~\ref{set-up: Brinon's rings} with respect to $L$.
Let $f\colon R\ra R'$ be an $\calO_K$-algebra homomorphism. Fix a path on $\Spec R_K$ from $\overline{\Frac R'}$ to $\overline{\Frac R}$ and consider the associated homomorphism $\calG_{R'_L}\ra\calG_{R_K}$. 
 Let $V\in\Rep_{\Q_p}(\calG_{R_K})$ and set $f^\ast V:=V|_{\calG_{R'_L}}$
If $V$ is horizontal semistable and $f^\ast V$ is horizontal crystalline, then $V$ is horizontal crystalline.
\end{prop}

\begin{proof}
Denote the residue field of $L$ by $k_L$ and set $L_0=W(k_L)[p^{-1}]$.
By Lemma~\ref{lem: admissiblity of two period rings}, we have
$D_{\st}^\nabla(f^\ast V)\cong L_0\otimes_{K_0}D_{\st}^\nabla(V)$.
Moreover, it follows from the construction that this isomorphism is compatible with the monodromy operator $N$ on the left hand side and $\id\otimes N$ on the right hand side.
Since $f^\ast V$ is horizontal crystalline, we have $N=0$ on $D_{\st}^\nabla(f^\ast V)$. By faithful flatness of $K_0\ra L_0$, we also have $N=0$ on $D_{\st}^\nabla(V)$. Hence $V$ is horizontal crystalline by Lemma~\ref{lem: semistable with zero monodromoy implies crystalline}.
\end{proof}

\begin{rem}
 We will use the proposition when $R'=\calO_L$.
In this case, it is informally phrased as follows: if a horizontal semistable representation is crystalline at one classical point, then it is horizontal crystalline.
\end{rem}

\subsection{Horizontal de Rham representations and de Rham representations}\label{subsection:de Rham representations}

In this subsection, we review the de Rham period ring studied in \cite{Brinon-crisdeRham}.
This ring will be used in the next section to define the $V_{\st}$-functor.
Let $R$ be an $\calO_K$-algebra satisfying the conditions in Set-up~\ref{set-up: Brinon's rings}.
Set
\[
 \widehat{\Omega}_R:=\varprojlim_m\Omega^1_{R/\Z}/p^m\Omega^1_{R/\Z}.
\]
\begin{defn}
 Let $\rmB_{\dR}(R)$ denote the de Rham period ring of $R$ defined in \cite[D\'efinition 5.1.5]{Brinon-crisdeRham}.
We briefly recall the definition below. 
\end{defn}

The map $\theta_{\A_{\inf}(\widehat{\overline{R}})}\colon \A_{\inf}(\widehat{\overline{R}})\ra \widehat{\overline{R}}$ extends $R$-linearly to the map
\[
 \theta_R\colon R\otimes_{\Z}\A_{\inf}(\widehat{\overline{R}})\ra \widehat{\overline{R}}.
\]
Let $\rmA_{\inf}(\widehat{\overline{R}}/R)$ denote the completion of the ring $R\otimes_{\Z}\A_{\inf}(\widehat{\overline{R}})$ with respect to the topology defined by the ideal $\theta_R^{-1}(p\widehat{\overline{R}})$ (\cite[D\'efinition 5.1.3]{Brinon-crisdeRham}). The map $\theta_R$ naturally extends to a map
$\rmA_{\inf}(\widehat{\overline{R}}/R)[p^{-1}]\ra \widehat{\overline{R}}[p^{-1}]$, 
which we still denote by $\theta_R$.
We set
\[
 \rmB_{\dR}^+(R):=\varprojlim_m\rmA_{\inf}(\widehat{\overline{R}}/R)[p^{-1}]/(\Ker\theta_R)^m,
\]
and denote the natural map $\rmB_{\dR}^+(R)\ra\widehat{\overline{R}}[p^{-1}]$ by $\theta_R$.
For $m\in\N$, we set $\Fil^m\rmB_{\dR}^+(R):=(\Ker\theta_R)^m$. This is a decreasing separated and exhaustive filtration on $\rmB_{\dR}^+(R)$.
We define
\[
 \rmB_{\dR}(R):=\rmB_{\dR}^+(R)[t^{-1}].
\]
The ring has an action of $\calG_{R_K}$, a decreasing separated and exhaustive filtration $\Fil^\bullet$ given by
\[
 \Fil^0\rmB_{\dR}(R)=\sum_{i=0}^\infty t^{-i}\Fil^i\rmB_{\dR}^+(R),\;\text{and}\;
\Fil^m\rmB_{\dR}(R)=t^m\Fil^0\rmB_{\dR}(R)\quad (m\in\Z),
\]
 and a $\rmB_{\dR}^\nabla(R)$-linear integrable connection
\[
 \nabla\colon \rmB_{\dR}(R)\ra \rmB_{\dR}(R)\otimes_R\widehat{\Omega}_R
\]
satisfying the Griffiths transversality
\[
 \nabla(\Fil^m\rmB_{\dR}(R))\subset \Fil^{m-1}\rmB_{\dR}(R)\otimes_{R}\widehat{\Omega}_R.
\]
We also have $(\rmB_{\dR}(R))^{\calG_{R_K}}=R_K$ and  $\rmB_{\dR}^\nabla(R)=\rmB_{\dR}(R)^{\nabla=0}$ (\cite[Propositions 5.2.12,  5.3.3]{Brinon-crisdeRham}).
The ring $\rmB_{\dR}(R)$ is a $(\Q_p,\calG_{R_K})$-ring.

\begin{defn}
For $V\in \Rep_{\Q_p}(\calG_{R_K})$, we write $D_{\dR}(V)$ for $D_{\rmB_{\dR}(R)}(V)$ and $\alpha_{\dR}(V)$ for $\alpha_{\rmB_{\dR}(R)}(V)$.
We say that $V$ is \emph{de Rham} if it is $\rmB_{\dR}(R)$-admissible, i.e., $\alpha_{\dR}(V)$ is an isomorphism of $\rmB_{\dR}(R)$-modules.
Note that $\alpha_{\dR}(V)$ is always injective by \cite[Proposition 8.2.4]{Brinon-crisdeRham}.
\end{defn}

The $R_K$-module $D_{\dR}(V)$ is projective of rank at most $\dim_{\Q_p}V$ by \cite[Proposition 8.3.1]{Brinon-crisdeRham}, and it is equipped with induced connection
\[
\nabla\colon D_{\dR}(V)\ra D_{\dR}(V)\otimes_R\widehat{\Omega}_R, 
\]
and filtration $\Fil^\bullet D_{\dR}(V)$.
We also have a natural $R_K$-linear map
\[
 \beta_{\dR}(V)\colon R_K\otimes_KD_{\dR}^\nabla(V)\ra D_{\dR}(V).
\]

\begin{lem}[{\cite[Proposition 8.2.1]{Brinon-crisdeRham}}]
For $V\in \Rep_{\Q_p}(\calG_{R_K})$, it is horizontal de Rham if and only if it is de Rham
and $\beta_{\dR}(V)$ is an isomorphism.
\end{lem}

\begin{lem}[{\cite[Proposition 8.4.3]{Brinon-crisdeRham}}]
Let $V\in\Rep_{\Q_p}(\calG_{R_K})$. If $V$ is de Rham, then the isomorphism
\[
 \alpha_{\dR}(V)\colon \rmB_{\dR}(R)\otimes_{R_K}D_{\dR}(V)\stackrel{\cong}{\lra}
\rmB_{\dR}(R)\otimes_{\Q_p}V
\]
is compatible with $\calG_{R_K}$-actions and connections, and strictly compatible with filtrations.
\end{lem}

\begin{lem}[{\cite[Corollaire 5.2.11]{Brinon-crisdeRham}}]\label{lem: restriction of filtrations to horizontal de Rham}
For $m\in\Z$, we have
\[
 \rmB_{\dR}^\nabla(R)\cap \Fil^m\rmB_{\dR}(R)=t^m\rmB_{\dR}^{\nabla+}(R).
\]
 \end{lem}

\begin{rem}\label{rem:horizontal comparison does not preserve filtrations}
Let $V\in \Rep_{\Q_p}(\calG_{R_K})$ be a horizontal de Rham representation.
The filtration $\Fil^m\rmB_{\dR}^\nabla(R):=t^m\rmB_{\dR}^{\nabla+}(R)$
on $\rmB_{\dR}^\nabla(R)$ induces a filtration on $D_{\dR}^\nabla(V)$.
The maps $\alpha_{\dR}^\nabla(V)$ and $\beta_{\dR}(V)$ preserve filtrations, i.e., 
\begin{align*}
 \alpha_{\dR}^\nabla(V)&(\Fil^m( \rmB_{\dR}^\nabla(R)\otimes_K D_{\dR}^\nabla(V)))\subset
\Fil^m \rmB_{\dR}^\nabla(R)\otimes_{\Q_p}V, 
\\
\beta_{\dR}(V)&(\Fil^m( R_K\otimes_K D_{\dR}^\nabla(V)))\subset\Fil^m D_{\dR}(V).
\end{align*}
 However, neither $\alpha_{\dR}^\nabla(V)$ nor $\beta_{\dR}(V)$ is strictly compatible with filtrations in general\footnote{\cite[Proposition 8.4.3]{Brinon-crisdeRham} claims that $\alpha_{\dR}^\nabla(V)$ is strictly compatible with filtrations, but a proof is not given.}.
See the example below.
\end{rem}

\begin{example}\label{example: non-compatibility of filtrations}
We will give an example of a horizontal crystalline representation (hence horizontal de Rham) such that $\alpha_{\dR}^\nabla(V)$ is not strictly compatible with filtrations.

Assume $K=K_0$ for simplicity.
Let $R=\calO_K\langle T^{\pm}\rangle$ and let $h(T)=1+pT\in R$.
Fix a nontrivial compatible system $(1,\varepsilon_1,\ldots)$ of $p$-power roots of $1$ ($\varepsilon_1\neq 1$) and a compatible system $(h,h^{1/p},\ldots)$ of $p$-power roots of $h$.
They define elements of $\widehat{\overline{R}}^\flat$:
\[
 \varepsilon:=(1\bmod p,\varepsilon_1\bmod p,\ldots), \qquad
 h^\flat:=(h\bmod p,h^{1/p}\bmod p,\ldots).
\]
Consider the two-dimensional $\Q_p$-representation $V$ of $\calG_{R_K}$
corresponding to $(\varepsilon, h^\flat)$.
Namely, 
set $V=\Q_pe_1+\Q_pe_2$ and define the action of $g\in \calG_{R_K}$ on $V$
by 
\[
 ge_1=\chi(g)e_1,\quad ge_2=\eta(g)e_1+e_2, 
\]
where $\chi\colon \calG_{R_K}\ra \Gal(\overline{K}/K)\ra \Z_p^\times$ is the $p$-adic cyclotomic character defined by $g(\varepsilon_m)=\varepsilon_m^{\chi(g)}$, 
 and $\eta\colon \calG_{R_K}\ra \Z_p$ is a map defined by
$g(h^{1/p^m})=\varepsilon_m^{\eta(g)}h^{1/p^m}$.

We will show that $V$ is a horizontal crystalline representation.

\begin{claim*}
 The series
\[
 \log [h^\flat]:=\sum_{m=1}^\infty  (-1)^{m-1}\frac{([h^\flat]-1)^m}{m}
\]
converges in $\rmA_{\max}^\nabla(R):=\A_{\max}(\widehat{\overline{R}})$ and satisfies
\[
 g\bigl(\log [h^\flat]\bigr)=\eta(g)t+\log[h^\flat],\quad \varphi\bigl(\log [h^\flat]\bigr)=p\log [h^\flat].
\]
Moreover, we have
$\log[h^\flat]\in \Fil^0\rmB_{\dR}^\nabla(R)\smallsetminus \Fil^1\rmB_{\dR}^\nabla(R)$.
\end{claim*}

\begin{proof}
Observe $\theta_{\A_{\inf}(\widehat{\overline{R}})}([h^\flat]-1)=h-1=pT\in  p\widehat{\overline{R}}$.
Hence we can write $[h^\flat]-1=x+py$ for some $x\in \Ker\theta_{\A_{\inf}(\widehat{\overline{R}})}$ and $y\in \A_{\inf}(\widehat{\overline{R}})$.
It follows
\[
 (-1)^{m-1}\frac{([h^\flat]-1)^m}{m}
=(-1)^{m-1}\frac{p^m}{m}\biggl(\frac{[h^\flat]-1}{p}\biggr)^m
=(-1)^{m-1}\frac{p^m}{m}\biggl(\frac{x}{p}+y\biggr)^m.
\]
Since $\frac{p^m}{m}\in \Z_p$ and it converges to zero as $m\to \infty$,
we conclude that the series $\sum_{m=1}^\infty  (-1)^{m-1}\frac{([h^\flat]-1)^m}{m}$
converges in $\rmA_{\max}^\nabla(R)$.

We can easily check the remaining statements (cf.~\cite[3.1.3]{Fontaine-exposeII} and \cite[Lemma 9.2.2]{Brinon-Conrad}).
For example, since $\theta_{\rmA_{\max}^\nabla(R)}\colon \rmA_{\max}^\nabla(R)\ra \widehat{\overline{R}}$ is continuous with respect to the $p$-adic topology, we have 
\[
 \theta(\log [h^\flat])=\sum_{m=1}^\infty  (-1)^{m-1}\frac{(h-1)^m}{m}=\sum_{m=1}^\infty  (-1)^{m-1}\frac{(pT)^m}{m}\neq 0.
\]
Hence $\log[h^\flat]\in \Fil^0\rmB_{\dR}^\nabla(R)\smallsetminus \Fil^1\rmB_{\dR}^\nabla(R)$.
\end{proof}

Set 
\[
 f_1=\frac{1}{t}e_1,\quad f_2=e_2-\frac{\log [h^\flat]}{t}e_1\in \rmB_{\max}^\nabla(R)\otimes_{\Q_p} V.
\]
Then we can check that 
$f_1$ and $f_2$ are fixed by $\calG_{R_K}$ and generate $\rmB_{\max}^\nabla(R)\otimes_{\Q_p} V$ as a $\rmB_{\max}^\nabla(R)$-module.
This implies that $V$ is horizontal crystalline with $D_{\max}^\nabla(V)=Kf_1+Kf_2$.
We also see that $\varphi(f_1)=\frac{1}{p}f_1$ and $\varphi(f_2)=f_2$.
Note that $D_{\dR}^\nabla(V)=D_{\max}^\nabla(V)$ in this case and 
if we equip $D_{\dR}^\nabla(V)$ with a decreasing filtration
induced from $ \Fil^m(\rmB_{\dR}^\nabla(R)\otimes_{\Q_p}V):=t^m\rmB_{\dR}^{\nabla+}(R)\otimes_{\Q_p}V$, 
we have
\[
 \Fil^{-1}D_{\dR}^\nabla(V)=D_{\dR}^\nabla(V),\quad \Fil^0D_{\dR}^\nabla(V)=0.
\]
In particular, the $\rmB_{\dR}^\nabla(R)$-linear isomorphism
\[
 \alpha_{\dR}^\nabla(V)\colon
\rmB_{\dR}^\nabla(R)\otimes_KD_{\dR}^\nabla(V)\stackrel{\cong}{\lra}\rmB_{\dR}^\nabla(R)\otimes_{\Q_p}V
\]
\emph{is not} strictly compatible with filtrations.
Indeed, we have $e_2=1\otimes f_2+\log[h^\flat]\otimes f_1\not\in \Fil^0(\rmB_{\dR}^\nabla(R)\otimes_KD_{\dR}^\nabla(V))$ while $e_2\in \Fil^0(\rmB_{\dR}^\nabla(R)\otimes_{\Q_p}V)$.

Let now turn to $D_{\dR}(V)=R_Kf_1+R_Kf_2$. We will give another $R_K$-basis.
Note that $h$ is invertible in $R$ and 
$[h^\flat]/h-1\in \Ker \theta_R$.
It follows that the series
\[
 \log \frac{[h^\flat]}{h}:=\sum_{m=1}^\infty  (-1)^{m-1}\frac{\bigl([h^\flat]/h-1\bigr)^m}{m}
\]
converges in $\rmB_{\dR}^+(R)$.
Let
\[
  \log h:=\sum_{m=1}^\infty  (-1)^{m-1}\frac{\bigl(h-1\bigr)^m}{m}
=\sum_{m=1}^\infty  (-1)^{m-1}\frac{p^m}{m}T^m\in R.
\]
We can check
\[
 \log[h^\flat]=\log \frac{[h^\flat]}{h}+\log h\in \rmB_{\dR}^+(R).
\]
Set
\[
 f_3=e_2-\frac{\log [h^\flat]/h}{t}e_1=f_2+\frac{\log h}{t}e_1\in \rmB_{\dR}(R)\otimes_{\Q_p}V.
\]
Then $f_3$ is fixed by $\calG_{R_K}$, and $D_{\dR}(V)=R_Kf_1+R_Kf_3$.
Consider the filtration on $D_{\dR}(V)$ induced from
$ \Fil^m(\rmB_{\dR}(R)\otimes_{\Q_p}V):=(\Fil^m\rmB_{\dR}(R))\otimes_{\Q_p}V.$
Since $\log\frac{[h^\flat]}{h}\in \Fil^1\rmB_{\dR}(R)$, we see
\[
 \Fil^{-1}D_{\dR}(V)=D_{\dR}(V),\quad \Fil^{0}D_{\dR}(V)=R_Kf_3,\quad \Fil^{1}D_{\dR}(V)=0.
\]
It follows that
the natural $R_K$-linear isomorphism
\[
\beta_{\dR}(V)\colon R_K\otimes_{K} D_{\dR}^\nabla(V)\stackrel{\cong}{\lra}D_{\dR}(V)
\]
\emph{is not} strictly compatible with filtrations. Note that the $\rmB_{\dR}(R)$-linear isomorphism
\[
 \alpha_{\dR}(V)\colon
\rmB_{\dR}(R)\otimes_{R_K}D_{\dR}(V)\stackrel{\cong}{\lra}\rmB_{\dR}(R)\otimes_{\Q_p}V
\]
is strictly compatible with filtrations.
\end{example}

\section{Filtered $(\varphi,N,\Gal(L/K),R_K)$-modules}\label{section:filtered phi N modules}

In this section, we follow \cite{Fontaine-exposeIII} and define filtered $(\varphi,N,\Gal(L/K),R_K)$-modules and the functors $D_{\pst}$ and $V_{\pst}$.

As before, let $k$ be a perfect field of characteristic $p$ and set $K_0:=W(k)[p^{-1}]$.
Let $K$ be a totally ramified finite extension of $K_0$.
We denote the maximal unramified extension of $K_0$ inside $\overline{K}$ by $K_0^\ur$.
Let $L$ be a Galois extension of $K$ inside $\overline{K}$. Let $L_0$ be the maximal unramified extension of $K_0$ in $L$ and $\sigma$ the absolute Frobenius on $L_0$.
Note $(\overline{K})_0=K_0^\ur$.

Let us first recall the definition of $(\varphi,N,\Gal(L/K))$-modules:
\begin{defn}[{\cite[4.2.1]{Fontaine-exposeIII}}]
 A \emph{$(\varphi,\!N,\!\Gal(L/K))$-module} is a finite-dimensional $L_0$-vector space $D$ equipped with
\begin{enumerate}
 \item an injective $\sigma$-semilinear map $\varphi\colon D\ra D$ (\emph{Frobenius}),
 \item an $L_0$-linear endomorphism $N\colon D\ra D$ (\emph{monodromy operator}), and
 \item a semilinear action of $\Gal(L/K)$
\end{enumerate}
such that they satisfy the following compatibilities:
\begin{itemize}
 \item $N\varphi=p\varphi N$;
 \item for every $g\in\Gal(L/K)$, $g\varphi=\varphi g$, and $gN=Ng$.
\end{itemize}
When $L=K$, we simply call it a $(\varphi,N)$-module.

A morphism of $(\varphi,N,\Gal(L/K))$-modules is an $L_0$-linear map that commutes with $\varphi, N$, and $\Gal(L/K)$-action.

We say that a $(\varphi,N,\Gal(L/K))$-module $D$ is \emph{discrete} if the action of $\Gal(L/K)$ on $D$ is discrete, i.e., for every $d\in D$, the stabilizer subgroup of $d$ is open in $\Gal(L/K)$.
\end{defn}

Let $R$ be an $\calO_K$-algebra satisfying the conditions in Set-up~\ref{set-up: Brinon's rings}.
We set
\[
 R_{\calO_L}:=\calO_L\otimes_{\calO_K}R,\quad\text{and}\quad
 R_L:=L\otimes_KR_K.
\]
Note that $R_{\calO_L}$ is $p$-adically complete if $L$ is finite over $K$ but that it is not the case in general.
For an $L_0$-vector space $D$, we write $D_{R_L}$
for $R_L\otimes_{L_0}D$.

\begin{defn}[cf.~{\cite[4.3.2]{Fontaine-exposeIII}}]
 A \emph{filtered $(\varphi,N,\Gal(L/K), R_K)$-module}
is a pair consisting of
\begin{enumerate}
 \item a $(\varphi,N,\Gal(L/K))$-module $D$, and 
 \item a separated and exhaustive decreasing filtration
$(\Fil^i D_{R_L})_{i\in\Z}$ of $D_{R_L}$
\end{enumerate} 
such that
each $\Fil^i D_{R_L}$ is a $\Gal(L/K)$-stable projective
$R_L$-submodule of $D_{R_L}$.
When $L=K$, we simply call it a $(\varphi,N,R_K)$-module.

A morphism of filtered $(\varphi,N,\Gal(L/K),R_K)$-modules is defined to be a morphism of $(\varphi,N,\Gal(L/K))$-modules whose scalar extension from $L_0$ to $R_L$ preserves filtrations.
A filtered $(\varphi,N,\Gal(L/K), R_K)$-module is called \emph{discrete} if it is discrete as a $(\varphi,N,\Gal(L/K))$-module.

We denote the full subcategory of discrete filtered $(\varphi,N,\Gal(L/K),R_K)$-modules by $\MF_{L/K}(\varphi,N,R_K)$ .
\end{defn}

\begin{prop}\label{prop:Dst functor}
Let $L$ be a finite Galois extension of $K$ and let $V\in\Rep_{\Q_p}(\calG_{R_K})$.
If $V|_{\calG_{R_L}}$ is horizontal semistable, then
$D_{\st}^\nabla(V|_{\calG_{R_L}})$ is a $(\varphi,N,\Gal(L/K))$-module.
Moreover, if we equip $D_{\st}^\nabla(V|_{\calG_{RL}})$ with the filtration
induced from $D_{\dR}(V|_{\calG_{R_L}})$ via the isomorphism
\[
R_L\otimes_{L_0}D_{\st}^\nabla(V|_{\calG_{R_L}})\stackrel{\cong}{\lra}
D_{\dR}(V|_{\calG_{R_L}}),
\]
then 
$D_{\st}^\nabla(V|_{\calG_{R_L}})$ is a filtered $(\varphi,N,\Gal(L/K),R_K)$-module.
\end{prop}

\begin{proof}
Recall that $\rmB_{\st}^\nabla(R_{\calO_L})\otimes_{\Q_p}V$ admits an injective $\sigma$-semilinear map $\varphi\otimes\id$ and an $L_0$-linear endomorphism $N\otimes \id$
satisfying $(N\otimes \id)(\varphi\otimes\id)=p(\varphi\otimes\id)(N\otimes\id)$.
Moreover, by the identification $\rmB_{\st}^\nabla(R_{\calO_L})\otimes_{\Q_p}V=\rmB_{\st}^\nabla(R)\otimes_{\Q_p}V$, it is equipped with diagonal semilinear $\calG_{R_K}$-action  commuting with $\varphi\otimes\id$ and $N\otimes \id$.
The first assertion follows from these remarks and $\calG_{R_K}/\calG_{R_L}\cong \Gal(L/K)$. The second assertion follows from the first and \cite[Proposition 8.3.4(a)]{Brinon-crisdeRham}.
\end{proof}

\begin{defn}
Assume that $L$ is finite over $K$.  Recall  the natural identifications $\rmB_{\st}^\nabla(R_{\calO_L})=\rmB_{\st}^\nabla(R)$ and $\rmB_{\dR}^\nabla(R_{\calO_L})=\rmB_{\dR}^\nabla(R)$.
For a filtered $(\varphi,N,\Gal(L/K), R_K)$-module $D$, set
\[
 V_{\st,L}(D):=(\rmB_{\st}^\nabla(R)\otimes_{L_0}D)^{\varphi=1, N=0}\cap
\Fil^0(\rmB_{\dR}(R)\otimes_{R_L}D_{R_L}),
\]
where $\varphi$ (resp.~$N$) on $\rmB_{\st}^\nabla(R)\otimes_{L_0}D$ denotes $\varphi\otimes\varphi$ (resp.~$N\otimes \id+\id\otimes N$) and $\Fil^\bullet$ on $\rmB_{\dR}(R)\otimes_{R_L}D_{R_L}$ denotes $\sum_{m\in\Z}\Fil^m\rmB_{\dR}(R)\otimes_{R_L} \Fil^{\bullet-m}D_{R_L}$.
 It is a $\Q_p$-vector space equipped with a $\Q_p$-linear action of $\calG_{R_K}$.
\end{defn}

\begin{prop}\label{prop:Vst functor}
Let $L$ be a finite Galois extension of $K$ and let $V\in\Rep_{\Q_p}(\calG_{R_K})$.
If $V|_{\calG_{R_L}}$ is horizontal semistable, then
the natural $\Q_p$-linear map
\[
 V\ra V_{\st,L}(D_{\st}^\nabla(V|_{\calG_{R_L}}))
\]
is a $\calG_{R_K}$-equivariant isomorphism.
\end{prop}

\begin{proof}
We use the identification $\rmB_{\st}^\nabla(R_{\calO_L})=\rmB_{\st}^\nabla(R)$.
For simplicity, we write $V_L$ for $V|_{\calG_{R_L}}$.
 Since $V_L$ is horizontal semistable, we have the commutative diagram
\[
 \xymatrix{
\rmB_{\st}^\nabla(R)\otimes_{L_0}D_{\st}^\nabla(V_L)
\ar[rrr]^{\alpha_{\st}^\nabla(V_L)}_{\cong}\ar[d]
&&& \rmB_{\st}^\nabla(R)\otimes_{\Q_p}V\ar[d]\\
\rmB_{\dR}(R)\otimes_{R_L}D_{\dR}(V_L)
\ar[rrr]^{\alpha_{\dR}(V_L)}_{\cong}
&&& \rmB_{\dR}(R)\otimes_{\Q_p}V.
}
\]
Furthermore, the isomorphism $\alpha_{\st}^\nabla(V_L)$ is compatible with $\calG_{R_K}$-actions, Frobenii, and monodromy operators, and the isomorphism $\alpha_{\dR}(V_L)$ is compatible with $\calG_{R_K}$-actions and strictly compatible with filtrations. Note also $D_{\st}^\nabla(V_L)_{R_L}\cong D_{\dR}(V_L)$.
Hence we have $\calG_{R_K}$-equivariant isomorphisms
\begin{align*}
 V_{\st}(D_{\st}^\nabla(V_L))&=(\rmB_{\st}^\nabla(R)\otimes_{L_0}D_{\st}^\nabla(V_L))^{\varphi=1, N=0}\cap \Fil^0(\rmB_{\dR}(R)\otimes_{R_L}D_{\st}^\nabla(V_L)_{R_L})\\
&\cong(\rmB_{\st}^\nabla(R)\otimes_{L_0}D_{\st}^\nabla(V_L))^{\varphi=1, N=0}\cap \Fil^0(\rmB_{\dR}(R)\otimes_{R_L}D_{\dR}(V_L))\\
&\cong (\rmB_{\st}^\nabla(R)\otimes_{\Q_p}V)^{\varphi=1, N=0}\cap \Fil^0(\rmB_{\dR}(R)\otimes_{\Q_p}V)\\
&\cong \bigl((\rmB_{\st}^\nabla(R))^{\varphi=1, N=0}\cap \Fil^0(\rmB_{\dR}(R))\bigr)\otimes_{\Q_p}V.
\end{align*}
By Lemma~\ref{lem: restriction of filtrations to horizontal de Rham} and Proposition~\ref{prop:fundamental exact sequence}, we have
\begin{align*}
 (\rmB_{\st}^\nabla(R))^{\varphi=1, N=0}\cap \Fil^0(\rmB_{\dR}(R))
&=(\rmB_{\st}^\nabla(R))^{\varphi=1, N=0}\cap (\rmB_{\dR}^\nabla(R)\cap\Fil^0(\rmB_{\dR}(R)))\\
&=(\rmB_{\max}^\nabla(R))^{\varphi=1}\cap \rmB_{\dR}^{\nabla+}(R)=\Q_p.
\end{align*}
Hence $\bigl((\rmB_{\st}^\nabla(R))^{\varphi=1, N=0}\cap \Fil^0(\rmB_{\dR}(R))\bigr)\otimes_{\Q_p}V$ is $\calG_{R_K}$-equivariantly isomorphic to $V$.
\end{proof}

\begin{defn}
For $V\in\Rep_{\Q_p}(\calG_{R_K})$, set
\[
 D_{\pst}^\nabla(V):=\varinjlim_LD_{\st}^\nabla(V|_{\calG_{R_L}}),
\]
where $L$ ranges over all the finite Galois extensions of $K$ inside $\overline{K}$.
It is a vector space over $K_0^\ur$ equipped with a Frobenius $\varphi$, a monodromy operator $N$, and a semilinear action of $\Gal(\overline{K}/K)$.
By unwinding the definitions, we see $D_{\pst}^\nabla(V)^{\Gal(\overline{K}/L)}=D_{\st}^\nabla(V|_{\calG_{R_L}})$ for every finite Galois extension $L$ of $K$ inside $\overline{K}$.

Let $\Rep_{\pst}^\nabla(\calG_{R_K})$ be the full subcategory of $\Rep_{\Q_p}(\calG_{R_K})$
consisting of objects $V$ with the following property: there exists a finite Galois extension $L$ of $K$ such that $V|_{\calG_{R_L}}$ is horizontal semistable.
\end{defn}

\begin{rem}
 By Lemma~\ref{lem:potentially horizontal de Rham is horizontal de Rham}(ii), every object of $\Rep_{\pst}^\nabla(\calG_{R_K})$ is horizontal de Rham.
We will prove that $\Rep_{\pst}^\nabla(\calG_{R_K})$ is exactly the full subcategory of horizontal de Rham representations if $R$ satisfies Condition (BR) (Theorem~\ref{thm:p-adic monodromy for horizontal de Rham representations}).
In this case, $\Rep_{\pst}^\nabla(\calG_{R_K})$ is a Tannakian subcategory of $\Rep_{\Q_p}(\calG_{R_K})$ by Remark~\ref{rem:tensor of horizontal de Rham representations},  Lemma~\ref{lem:dual of horizontal de Rham} and by Theorem~\ref{thm:p-adic monodromy for horizontal de Rham representations} or the fact that $R_{\calO_L}$ satisfies Condition (BR) with respect to $L$ for every finite extension $L$ of $K$.
\end{rem}

\begin{lem}
Let $V\in \Rep_{\pst}^\nabla(\calG_{R_K})$.
Let $L$ be a finite Galois extension of $K$ inside $\overline{K}$ such that $V|_{\calG_{R_L}}$ is horizontal semistable. 
Then there is a natural identification $D_{\pst}^\nabla(V)_{R_{\overline{K}}}\cong R_{\overline{K}}\otimes_LD_{\dR}(V|_{\calG_{R_L}})$.
Moreover, if we set
\[
 \Fil^\bullet D_{\pst}^\nabla(V)_{R_{\overline{K}}}:=R_{\overline{K}}\otimes_{R_L}\Fil^\bullet D_{\dR}(V|_{\calG_{R_L}}),
\]
then $D_{\pst}^\nabla(V)$ is a discrete filtered $(\varphi,N,\Gal(\overline{K}/K), R_K)$-module and the filtration is independent of the choice of $L$.
\end{lem}

\begin{proof}
For every finite Galois extension $L'$ of $K$ containing $L$, $V|_{\calG_{R_{L'}}}$ is horizontal semistable and we have
\begin{align*}
 D_{\st}^\nabla(V|_{\calG_{R_{L'}}})&\cong L_0'\otimes_{L_0}D_{\st}^\nabla(V|_{\calG_{R_L}}),\quad\text{and}\\
D_{\dR}(V|_{\calG_{R_{L'}}})&\cong R_{L'}\otimes_{R_L}D_{\dR}(V|_{\calG_{R_L}})\cong R_{L'}\otimes_{L_0}D_{\st}^\nabla(V|_{\calG_{R_L}}).
\end{align*}
Moreover, the first isomorphism is compatible with Frobenii, monodromy operators and $\Gal(L'/K)$-actions,  and the second isomorphism is compatible with $\Gal(L'/K)$-actions and strictly compatible with filtrations.
 The assertions follow from these remarks and Proposition~\ref{prop:Dst functor}.
\end{proof}

\begin{defn}
Let $D$ be a filtered $(\varphi,N,\Gal(\overline{K}/K), R_K)$-module.
Set
\[
 V_{\pst}(D):=(\rmB_{\st}^\nabla(R)\otimes_{K_0^\ur}D)^{\varphi=1, N=0}\cap
\Fil^0(\rmB_{\dR}(R)\otimes_{R_{\overline{K}}}D_{R_{\overline{K}}}),
\]
where $\varphi$ (resp. $N$) on $\rmB_{\st}^\nabla(R)\otimes_{K_0^\ur}D$ denotes $\varphi\otimes\varphi$ (resp. $N\otimes \id+\id\otimes N$) and $\Fil^\bullet$ on $\rmB_{\dR}(R)\otimes_{R_{\overline{K}}}D_{R_{\overline{K}}}$ denotes $\sum_{m\in\Z}\Fil^m\rmB_{\dR}(R)\otimes_{R_{\overline{K}}} \Fil^{\bullet-m}D_{R_{\overline{K}}}$.
 It is a $\Q_p$-vector space equipped with a $\Q_p$-linear action of $\calG_{R_K}$.

We say that $D$ is \emph{admissible} if it is isomorphic to $D_{\pst}^\nabla(V)$ for
some $V\in\Rep_{\pst}^\nabla(\calG_{R_K})$.
Let $\MF_{\overline{K}/K}^{\ad}(\varphi,N,R_K)$ be the full subcategory of $\MF_{\overline{K}/K}(\varphi,N,R_K)$ consisting of admissible objects.
\end{defn}

\begin{thm}\label{thm:equivalence by Dpst and Vpst}
 Assume that $R$ satisfies Condition (BR). Then the functor
\[
 D_{\pst}^\nabla\colon \Rep_{\pst}^\nabla(\calG_{R_K})\ra 
\MF_{\overline{K}/K}^{\ad}(\varphi,N,R_K)
\]
is an equivalence of categories with quasi-inverse given by $V_{\pst}$.
\end{thm}

\begin{proof}
By Proposition~\ref{prop:Vst functor}, we see that 
the $\Q_p$-linear map $V\ra V_{\pst}(D_{\pst}(V))$ is a $\calG_{R_K}$-equivariant isomorphism.
For $V_1,V_2\in\Rep_{\pst}^\nabla(\calG_{R_K})$, this result for $V=\Hom(V_1,V_2)\in \Rep_{\pst}^\nabla(\calG_{R_K})$ implies that $D_{\pst}$ is fully faithful.
Hence the assertion follows.
\end{proof}

\begin{rem}
 As mentioned in the introduction, it is desirable to characterize admissible filtered $(\varphi,N,\Gal(\overline{K}/K), R_K)$-module intrinsically.
When $R=K$, admissibility is equivalent to weak admissibility by \cite{Colmez-Fontaine}.
Results on $p$-adic period domains suggest that admissibility will not be characterized by weak admissibility after every specialization of $R$ to a finite extension of $K$ (see \cite[Example 6.7]{Hartl-Rapoport-Zink}).

In \cite{Brinon-crisdeRham}, Brinon defines crystalline representations of $\calG_{R_K}$ and associates to them filtered $(\varphi,\nabla)$-modules over $R_K$ under Condition (BR).
He also defines the notion of pointwise weak admissibility for filtered $(\varphi,\nabla)$-modules over $R_K$ \cite[D\'efinition 7.1.11]{Brinon-crisdeRham} by considering the specialization of $R$ to the $p$-adic completion of the localization of $R$ at the prime ideal $(\pi)$ (when $R$ satisfies Condition (BR)).
Brinon proves that filtered $(\varphi,\nabla)$-modules over $R_K$ coming from crystalline representations are pointwisely weakly admissible (\cite[Proposition 8.3.4]{Brinon-crisdeRham}).
However, Moon \cite{Moon} proves that there are pointwisely weakly admissible filtered $(\varphi,\nabla)$-modules that do not arise from crystalline representations.
\end{rem}

\section{Purity for horizontal semistable representations}\label{section:purity}

\subsection{Statement}\label{subsection:statement of purity}
We first recall our set-up.
Let $k$ be a perfect field of characteristic $p$ and set $K_0:=W(k)[p^{-1}]$.
Let $K$ be a totally ramified finite extension of $K_0$ and fix a uniformizer $\pi$ of $K$.
We denote by $C$ the $p$-adic completion of an algebraic closure $\overline{K}$ of $K$.
 Let $R$ be an $\calO_K$-algebra satisfying the conditions in Set-up~\ref{set-up: Brinon's rings}. We further assume that $R$ satisfies Condition (BR). In particular, $(\pi)\subset R$ is a prime ideal. 

Consider the localization $R_{(\pi)}$ of $R$ at the prime ideal $(\pi)$. Let $\calO_{\calK}$ denote the $p$-adic completion of $R_{(\pi)}$. It is a complete discrete valuation ring with residue field admitting a finite $p$-basis. Set $\calK:=\calO_{\calK}[p^{-1}]$ and
fix an algebraic closure $\overline{\calK}$ of $\calK$. Let $\calO_{\overline{\calK}}$ denote the ring of integers of $\overline{\calK}$.
We denote the $p$-adic completion of $\calO_{\overline{\calK}}$ by $\widehat{\calO_{\overline{\calK}}}$.
Representations of $\Gal(\overline{\calK}/\calK)$ are studied in \cite{Brinon-cris, Morita, Ohkubo}.
In particular, we can discuss whether such a representation is horizontal crystalline, semistable, or de Rham (cf. \cite{Ohkubo}). Note that these classes of representations are defined by admissibility of the period rings $\B_{\cris}(\Lambda^+)$ (or equivalently $\B_{\max}(\Lambda^+)$), $\B_{\st}(\Lambda^+)$, and $\B_{\dR}(\Lambda^+)$ for $\Lambda^+=\widehat{\calO_{\overline{\calK}}}$, respectively (see Lemma~\ref{lem:equiv of max and cris admissibility} and \cite[\S 4]{Ohkubo} for horizontal semistable representations).

Let $T$ denote the set of prime ideals of $\overline{R}$ above $(\pi)\subset R$.
Note that $\calG_{R_K}$ acts transitively on $T$.
For $\fkp\in T$, consider the decomposition subgroup
\[
 \calG_{R_K}(\fkp):=\{g\in\calG_{R_K}\mid\,g(\fkp)=\fkp\}.
\]

We now fix an $R$-algebra embedding $\overline{R}\hra \calO_{\overline{\calK}}$.
This determines a prime ideal of $\overline{R}$ above $(\pi)\subset R$, which we denote by $\fkp_0$, and it gives rise to
 a homomorphism
\[
 \widehat{\calG}_{R_K}(\fkp_0):=\Gal(\overline{\calK}/\calK)\ra \calG_{R_K}.
\]
 Note that the map factors as $\widehat{\calG}_{R_K}(\fkp_0)\twoheadrightarrow \calG_{R_K}(\fkp_0)\hra\calG_{R_K} $ by \cite[Lemme 3.3.1]{Brinon-crisdeRham}.
We also remark that a different choice of the embedding $\overline{R}\hra \calO_{\overline{\calK}}$ defines another ideal $\fkp_0'\in T$ and it changes the homomorphism $\Gal(\overline{\calK}/\calK)\ra \calG_{R_K}$ with its conjugate by an element $g$ of $\calG_{R_K}$ such that $g(\fkp_0)=\fkp_0'$.

\begin{thm}
\label{thm:purity for horiztaonlly semistable}
Assume that $R$ satisfies Condition (BR).
Let $V\in \Rep_{\Q_p}(\calG_{R_K})$.
If $V$ is horizontal de Rham and $V|_{\widehat{\calG}_{R_K}(\fkp_0)}$ is horizontal semistable, then $V$ is horizontal semistable.
\end{thm}

\begin{rem}
The converse of Theorem~\ref{thm:purity for horiztaonlly semistable} also holds by Lemma~\ref{lem: admissiblity of two period rings}. 
\end{rem}

The rest of this section is devoted to the proof of Theorem~\ref{thm:purity for horiztaonlly semistable}. Our proof is modeled after Tsuji's work on a purity theorem on crystalline local systems \cite[Theorem 5.4.8]{Tsuji-cris}.

\subsection{Preliminaries on relevant period rings}

\begin{lem}\label{lem:canonical subfield}
Consider $d\colon \calK\ra \calK\otimes_{\calO_{\calK}}\widehat{\Omega}_{\calO_{\calK}}$,
where $\widehat{\Omega}_{\calO_{\calK}}:=\varprojlim_m\Omega^1_{\calO_{\calK}/\Z}/p^m\Omega^1_{\calO_{\calK}/\Z}$.
 We have
\[
\Ker(d\colon \calK\ra \calK\otimes_{\calO_{\calK}}\widehat{\Omega}_{\calO_{\calK}} )=K.
\]
\end{lem}

\begin{proof}
Since $K$ is algebraically closed in $R$, it is also algebraically closed in $\calK$. Hence the lemma follows from \cite[Proposition 2.28]{Brinon-cris}.
\end{proof}

\begin{rem}
 In \cite{Ohkubo}, this kernel is denoted by $\calK_{\can}$ (cf.~\cite[Definition 1.3 (i), Remark 1.4 (i)]{Ohkubo}).
\end{rem}

Note that $\widehat{\calO_{\overline{\calK}}}[p^{-1}]$ is a complete algebraically closed field containing $C$.
Hence the pair $\bigl(\widehat{\calO_{\overline{\calK}}}[p^{-1}],\widehat{\calO_{\overline{\calK}}}\bigr)$ is a perfectoid $C$-Banach pair.

For $\fkp\in T$, let $\widehat{\overline{R}_{\fkp}}$ denote
the $p$-adic completion of the localization $\overline{R}_{\fkp}$ of $\overline{R}$ at $\fkp$.
It is naturally an $\calO_{C}$-algebra.
Since $\overline{R}_{\fkp}$ can be written as the union of one-dimensional Noetherian normal local domains with $p$ non-unit,  $\widehat{\overline{R}_{\fkp}}[p^{-1}]$
is a complete valuation field with respect to $p$-adic valuation; the associated norm is given by
\[
 \lvert  x\rvert=\inf\bigl\{ \lvert a\rvert_{C}
\bigm| \, a\in C, \,x\in a\widehat{\overline{R}_{\fkp}}\bigr\}.
\]
In particular, $\widehat{\overline{R}_{\fkp}}[p^{-1}]$ is a uniform $C$-Banach algebra
and $\widehat{\overline{R}_{\fkp}}$ is a ring of integral elements.

\begin{lem}
The $C$-Banach pair $(\widehat{\overline{R}_{\fkp}}[p^{-1}],\widehat{\overline{R}_{\fkp}})$ is perfectoid.
\end{lem}

\begin{proof}
Since $\widehat{\overline{R}_{\fkp}}/p\widehat{\overline{R}_{\fkp}}=\overline{R}_{\fkp}/p\overline{R}_{\fkp}$ is the localization of $\overline{R}/p\overline{R}$ at the ideal $\fkp\bmod p\overline{R}$, the Frobenius is surjective on $\widehat{\overline{R}_{\fkp}}/p\widehat{\overline{R}_{\fkp}}$ as it is so on $\overline{R}/p\overline{R}$.
\end{proof}

To simplify our notation, set
\[
 \Lambda_R^+:=\widehat{\overline{R}},\quad
 \Lambda_{\fkp}^+:=\widehat{\overline{R}_{\fkp}} \quad(\fkp\in T), \quad\text{and}\quad
 \Lambda_{\calK}^+:=\widehat{\calO_{\overline{\calK}}}.
\]
We also set $\Lambda_{\ast}:=\Lambda_{\ast}^+[p^{-1}]$ for $\ast\in\{R,\fkp,\calK\}$.

 Consider the $\widehat{\calG}_{R_K}(\fkp_0)$-equivariant inclusion
$\overline{R}_{\fkp_0}\hra \calO_{\overline{\calK}}$, 
 where $\widehat{\calG}_{R_K}(\fkp_0)$ acts on 
$\overline{R}_{\fkp_0}$ via $\widehat{\calG}_{R_K}(\fkp_0)\twoheadrightarrow \calG_{R_K}(\fkp_0)$.
This extends to  $\widehat{\calG}_{R_K}(\fkp_0)$-equivariant injective maps
$\Lambda_{\fkp_0}^+\hra \Lambda_{\calK}^+$ and $ \Lambda_{\fkp_0}\hra\Lambda_{\calK}$.
Note that the former is a map of complete valuation rings of rank one and thus it has $p$-torsion free cokernel.

\begin{lem}\label{lem:Cartesian diagram for Lambda fkp and Lambda pi}
The natural
$\widehat{\calG}_{R_K}(\fkp_0)$-equivariant commutative diagram
\[
 \xymatrix{
\B_{\st,K}^+(\Lambda_{\fkp_0}^+)
\ar@{^{(}->}[d]\ar[r]
&
\B_{\dR}^+(\Lambda_{\fkp_0}^+)
\ar@{^{(}->}[d]\\
\B_{\st,K}^+(\Lambda_{\calK}^+)
\ar[r]
&
\B_{\dR}^+(\Lambda_{\calK}^+)
}
\]
is Cartesian.
\end{lem}

\begin{proof}
This follows from Lemma~\ref{lem:Cartesian diagram of period rings}.
\end{proof}

Let us recall the notion of $G$-regularity.

\begin{defn}[cf.~{\cite[1.4.1]{Fontaine-exposeIII}}]
Let $E$ be a finite extension of $\Q_p$ and $G$ a topological group.
 An $(E,G)$-ring $B$ is called \emph{$G$-regular} if 
the following three conditions hold:
\begin{itemize}
 \item[($G\cdot R_1$)] The ring $B$ is reduced.
 \item[($G\cdot R_2$)] For all $V\in\Rep_E(G)$, the map $\alpha_B(V)$ is injective.
 \item[($G\cdot R_3$)] Every one-dimensional $E$-vector subspace of $B$ that is stable under $G$-action  is generated by an invertible element of $B$ over $E$.
\end{itemize}
\end{defn}

\begin{prop}\label{prop:admissibility criterion}
Let $E$ be a finite extension of $\Q_p$ and $G$ a topological group.
Let $B$ be an $(E, G)$ ring that is $G$-regular.
Then $F:=B^G$ is a field, and the following holds:
\begin{enumerate}
 \item For every $V\in\Rep_E(G)$,
\[
 \dim_FD_B(V)\leq \dim_EV.
\]
 \item Moreover, if $\dim_FD_B(V)= \dim_EV$ holds, then the map $\alpha_B(V)$ is an isomorphism.
\end{enumerate}
\end{prop}

\begin{proof}
 The proof of \cite[1.4.2]{Fontaine-exposeIII} works in this setting.
\end{proof}

\begin{lem}\label{lem:regularity of period rings for Lambda calK}
We have
\[
 \bigl(\B_{\st}(\Lambda_{\calK}^+)\bigr)^{\widehat{\calG}_{R_K}(\fkp_0)}=K_0,\quad
\text{and}\quad
\bigl(\B_{\dR}(\Lambda_{\calK}^+)\bigr)^{\widehat{\calG}_{R_K}(\fkp_0)}=K.
\]
Moreover, the $(\Q_p, \widehat{\calG}_R(\fkp_0))$-rings $\B_{\st,K}(\Lambda_{\calK}^+)$ and $\B_{\dR}(\Lambda_{\calK}^+)$ are $\widehat{\calG}_{R_K}(\fkp_0)$-regular.
\end{lem}

\begin{proof}
The first assertion follows from \cite[Corollary 4.3, third]{Ohkubo}, Lemma~\ref{lem:canonical subfield}, and the inclusion $K\otimes_{K_0}\B_{\st}^+(\Lambda_{\calK}^+)\subset \B_{\dR}^+(\Lambda_{\calK}^+)$.
Since $\Lambda_{\calK}$ is a field, so is $\B_{\dR}(\Lambda_{\calK}^+)$.
Hence $\B_{\dR}(\Lambda_{\calK}^+)$ is $\widehat{\calG}_{R_K}(\fkp_0)$-regular by \cite[Proposition 1.6.3]{Fontaine-exposeIII}.

By \cite[Proposition 1.6.5]{Fontaine-exposeIII}, it remains to verify that $\B_{\st,K}(\Lambda_{\calK}^+)$ satisfies Condition ($G\cdot R_3$). We follow \cite[Lemma~4.5]{Ohkubo} and \cite[Lemme 5.1.3 (ii)]{Fontaine-exposeIII}.

There exists a complete valuation field $\calK^{\pf}\subset \overline{\calK}$ containing $\calK$ such that the relative ramification index of $\calK^\pf$ over $\calK$ is one and such that the residue field of $\calK^{\pf}$ is the perfect closure of $\calK$ inside the residue field of $\overline{\calK}$ (cf.~\cite[1G]{Ohkubo}).
Let $\calL$ denote the $p$-adic completion of the maximal unramified extension of $\calK^{\pf}$ inside $\Lambda_{\calK}=\widehat{\calO_{\overline{\calK}}}[p^{-1}]$ and let $P_0$ denote the fraction field of the ring of Witt vectors with coefficients in the residue field of $\calL$.
 Then $\calL$ is naturally a subfield of $\Lambda_{\calK}$, and $\B_{\st}(\Lambda_{\calK}^+)$ is a $P_0$-algebra. Note 
\[
 \calL\cong K\otimes_{K_0}P_0 \quad\text{and}\quad \B_{\st,K}(\Lambda_{\calK}^+)=\calL\otimes_{P_0}\B_{\st}(\Lambda_{\calK}^+).
\]
 Let $\overline{\calL}$ denote the algebraic closure of $\calL$ in $\Lambda_{\calK}$.
So $\Lambda_{\calK}$ is the $p$-adic completion of $\overline{\calL}$.

We will show that every one-dimensional $\calL$-vector subspace $\Delta$ of $\B_{\st,K}(\Lambda_{\calK}^+)$ that is stable by $\Gal(\overline{\calL}/\calL)$ is written as $\calL t^i$ for some $i\in\Z$. 
Since $t$ is invertible in $\B_{\st,K}(\Lambda_{\calK}^+)$ and $\Gal(\overline{\calL}/\calL)$ can be identified with a subgroup of $\widehat{\calG}_{R_K}(\fkp_0)$, this will complete the proof.

Let $b\in\Delta$ be a generator. By replacing $b$ by $bt^{-i}$ for some $i\in\Z$, we may assume $b\in \B_{\dR}^+(\Lambda_{\calK}^+)\setminus t\B_{\dR}^+(\Lambda_{\calK}^+)$.
It is enough to show $b\in \calL$.
Consider $\theta=\theta_{\B_{\dR}^+(\Lambda_{\calK}^+)}\colon \B_{\dR}^+(\Lambda_{\calK}^+)\ra \Lambda_{\calK}$.
The image $\theta(\Delta)$ is a one-dimensional $\calL$-vector subspace generated by $\theta(b)$ and stable by $\Gal(\overline{\calL}/\calL)$. By a result of Sen (cf.~ \cite[Remarque 3.3 (ii)]{Fontaine-exposeIII}), we have $\theta(b)\in \overline{\calL}\subset \Lambda_{\calK}$.
In particular, the action of $\Gal(\overline{\calL}/\calL)$ on $\theta(\Delta)$ factors through $\Gal(\calL'/\calL)$ for some finite Galois extension $\calL'$ of $\calL$ inside $\overline{\calL}$ and so does its action on $\Delta$. 

Let $\calL(b)$ denote the $\calL$-subalgebra of $\B_{\st,K}(\Lambda_{\calK}^+)$ generated by $b$. Since $b\in \B_{\dR}^+(\Lambda_{\calK}^+)\setminus t\B_{\dR}^+(\Lambda_{\calK}^+)$, we see that $\calL(b)$ is $\Gal(\overline{\calL}/\calL)$-equivariantly isomorphic to its image $\theta(\calL(b))$. It follows that $\calL(b)$ is isomorphic to a subextension of $\calL$ inside $\calL'$. This implies that
$\calL'\otimes_{\calL}\calL(b)$ is not an integral domain unless $\calL(b)=\calL$.
On the other hand, we know that $\calL'\otimes_{\calL}\B_{\st,K}(\Lambda_{\calK}^+)=\calL'\otimes_{P_0}\B_{\st}(\Lambda_{\calK}^+)$ embeds into $\B_{\dR}^+(\Lambda_{\calK}^+)$
by Corollary~\ref{cor:semistable period ring otimes K} (with respect to $\calL'$ instead of $K$) and thus it is an integral domain. Hence we conclude $b\in \calL$.
\end{proof}

\begin{lem}\label{lem:regularity of period rings for Lambda fkp}
We have
\[
 \bigl(\B_{\st}(\Lambda_{\fkp}^+)\bigr)^{\calG_{R_K}(\fkp)}=K_0,\quad
\text{and}\quad
\bigl(\B_{\dR}(\Lambda_{\fkp}^+)\bigr)^{\calG_{R_K}(\fkp)}=K.
\] 
Moreover, the $(\Q_p, \calG_{R_K}(\fkp))$-rings $\B_{\st,K}(\Lambda_{\fkp}^+)$ and $\B_{\dR}(\Lambda_{\fkp}^+)$ are $\calG_{R_K}(\fkp)$-regular.
\end{lem}

\begin{proof}
Without loss of generality we may assume $\fkp=\fkp_0$.
The first assertion follows from Lemmas~\ref{lem:Cartesian diagram for Lambda fkp and Lambda pi} and \ref{lem:regularity of period rings for Lambda calK}. 
Since $\Lambda_{\fkp_0}$ is a field, so is
$\B_{\dR}(\Lambda_{\fkp_0}^+)$.
Hence $\B_{\dR}(\Lambda_{\fkp_0})$ is $\calG_{R_K}(\fkp_0)$-regular.
It remains to verify that $\B_{\st,K}(\Lambda_{\fkp_0}^+)$ satisfies Condition ($G\cdot R_3$) by \cite[Proposition 1.6.5]{Fontaine-exposeIII}.
This follows from the $\widehat{\calG}_{R_K}(\fkp_0)$-equivariant Cartesian diagram in Lemma~\ref{lem:Cartesian diagram for Lambda fkp and Lambda pi} 
and the fact that $\B_{\dR}(\Lambda_{\fkp_0})$, $\B_{\st,K}(\Lambda_{\calK}^+)$, and $\B_{\dR}(\Lambda_{\calK}^+)$ satisfy Condition ($G\cdot R_3$).
\end{proof}

\subsection{Proof of Theorem~\ref{thm:purity for horiztaonlly semistable}}
Let us go back to the setting of Theorem~\ref{thm:purity for horiztaonlly semistable}.
By assumption, $V|_{\widehat{\calG}_R(\fkp_0)}$ is $\B_{\st}(\Lambda_{\calK}^+)$-admissible and thus $\B_{\st,K}(\Lambda_{\calK}^+)$-admissible. By replacing $V$ by its Tate twist $V(-n)$ for $n\gg 0$, we may assume 
\[
 \bigl(\B_{\dR}^+(\Lambda_R^+)\otimes_{\Q_p}V\bigr)^{\calG_{R_K}}\cong \bigl(\B_{\dR}(\Lambda_R^+)\otimes_{\Q_p}V\bigr)^{\calG_{R_K}}
\]
and 
\[
 \bigl(\B_{\st,K}^+(\Lambda_{\calK}^+)\otimes_{\Q_p}V\bigr)^{\widehat{\calG}_{R_K}(\fkp_0)}\!\!\stackrel{\cong}{\lra}\!
\bigl(\B_{\dR}^+(\Lambda_{\calK}^+)\otimes_{\Q_p}V\bigr)^{\widehat{\calG}_{R_K}(\fkp_0)}
\!\!\stackrel{\cong}{\lra}\!
\bigl(\B_{\dR}(\Lambda_{\calK}^+)\otimes_{\Q_p}V\bigr)^{\widehat{\calG}_{R_K}(\fkp_0)}.
\]
In particular, the natural map 
$\bigl(\B_{\dR}^+(\Lambda_R^+)\otimes_{\Q_p}V\bigr)^{\calG_{R_K}}\ra\bigl(\B_{\dR}^+(\Lambda_{\calK}^+)\otimes_{\Q_p}V\bigr)^{\widehat{\calG}_{R_K}(\fkp_0)}$
is an isomorphism.

\begin{lem}
The representation $V|_{\calG_{R_K}(\fkp_0)}$ is $\B_{\st,K}(\Lambda_{\fkp_0}^+)$-admissible.
Moreover, $\bigl(\B_{\st,K}^+(\Lambda_{\fkp_0}^+)\otimes_{\Q_p}V\bigr)^{\calG_{R_K}(\fkp_0)}=\bigl(\B_{\dR}^+(\Lambda_{\fkp_0}^+)\otimes_{\Q_p}V\bigr)^{\calG_{R_K}(\fkp_0)}$.
\end{lem}

\begin{proof}
Consider the commutative diagram
\[
\xymatrix{
\bigl(\B_{\st,K}^+(\Lambda_{\fkp_0}^+)\otimes_{\Q_p}V\bigr)^{\calG_{R_K}(\fkp_0)}
\ar@{^{(}->}[d]\ar@{^{(}->}[r]
&
\bigl(\B_{\dR}^+(\Lambda_{\fkp_0}^+)\otimes_{\Q_p}V\bigr)^{\calG_{R_K}(\fkp_0)}
\ar@{^{(}->}[d]\\
\bigl(\B_{\st,K}^+(\Lambda_{\calK}^+)\otimes_{\Q_p}V\bigr)^{\widehat{\calG}_{R_K}(\fkp_0)}
\ar[r]^{\cong}
&
\bigl(\B_{\dR}^+(\Lambda_{\calK}^+)\otimes_{\Q_p}V\bigr)^{\widehat{\calG}_{R_K}(\fkp_0)}.
} 
\]
It follows from the remark before the lemma that the right vertical map is an isomorphism of $K$-vector spaces of dimension $\dim_{\Q_p}V$. 
By Lemma~\ref{lem:Cartesian diagram for Lambda fkp and Lambda pi}, we conclude
that all the four $K$-vector spaces in the diagram are equal, and we obtain the second assertion.
We also have
\[
\dim_{\Q_p}V= \dim_K\bigl(\B_{\st,K}^+(\Lambda_{\fkp_0}^+)\otimes_{\Q_p}V\bigr)^{\calG_{R_K}(\fkp_0)}\leq\dim_K\bigl(\B_{\st,K}(\Lambda_{\fkp_0}^+)\otimes_{\Q_p}V\bigr)^{\calG_{R_K}(\fkp_0)}.
\]
By Proposition~\ref{prop:admissibility criterion} and Lemma~\ref{lem:regularity of period rings for Lambda fkp}, we have $\dim_K\bigl(\B_{\st,K}(\Lambda_{\fkp_0}^+)\otimes_{\Q_p}V\bigr)^{\calG_{R_K}(\fkp_0)}=\dim_{\Q_p}V$
and thus
$V|_{\calG_{R_K}(\fkp_0)}$ is $\B_{\st,K}(\Lambda_{\fkp_0}^+)$-admissible.
\end{proof}

Observe that each $\sigma\in\calG_{R_K}$ induces an $R$-algebra isomorphism $\Lambda_{\fkp}^+\ra\Lambda_{\sigma(\fkp)}^+$ for $\fkp\in T$, and $\calG_{R_K}(\sigma(\fkp))=\sigma\, \calG_{R_K}(\fkp)\,\sigma^{-1}$ in $\calG_{R_K}$.

\begin{lem}\label{lem:admissibility for Lambda fkp}
 For $\fkp\in T$, the representation $V|_{\calG_{R_K}(\fkp)}$ is $\B_{\st,K}(\Lambda_{\fkp}^+)$-admissible.
Moreover, $\bigl(\B_{\st,K}^+(\Lambda_{\fkp}^+)\otimes_{\Q_p}V\bigr)^{\calG_{R_K}(\fkp)}=\bigl(\B_{\dR}^+(\Lambda_{\fkp}^+)\otimes_{\Q_p}V\bigr)^{\calG_{R_K}(\fkp)}$.
\end{lem}

\begin{proof}
Choose $\sigma\in\calG_{R_K}$ with $\sigma(\fkp_0)=\fkp$.
Hence $\calG_{R_K}(\fkp)=\sigma\, \calG_{R_K}(\fkp_0)\,\sigma^{-1}$ in $\calG_{R_K}$, 
and $\sigma$ induces isomorphisms $\sigma\colon \Lambda_{\fkp_0}^+\stackrel{\cong}{\lra}\Lambda_{\fkp}^+$ and $\sigma\colon \B_{\st,K}(\Lambda_{\fkp_0}^+)\stackrel{\cong}{\lra}\B_{\st,K}(\Lambda_{\fkp}^+)$.
These maps are compatible with the group homomorphism $\calG_{R_K}(\fkp_0)\ra \calG_{R_K}(\fkp)$ given by $\sigma$-conjugation.

Hence for every $\tau\in\calG_{R}(\fkp_0)$, we have  a commutative diagram
\[
 \xymatrix{
\B_{\st,K}(\Lambda_{\fkp_0}^+)\otimes_{\Q_p}V
\ar[r]^{\sigma\otimes\sigma}_\cong\ar[d]_{\tau\otimes \tau}
&
\B_{\st,K}(\Lambda_{\fkp}^+)\otimes_{\Q_p}V
\ar[d]^{\sigma\tau\sigma^{-1}\otimes \sigma\tau\sigma^{-1}}\\
\B_{\st,K}(\Lambda_{\fkp_0}^+)\otimes_{\Q_p}V
\ar[r]^{\sigma\otimes\sigma}_\cong
&
\B_{\st,K}(\Lambda_{\fkp}^+)\otimes_{\Q_p}V.
}
\]
In particular, if $x\in \B_{\st,K}(\Lambda_{\fkp_0}^+)\otimes_{\Q_p}V$
is $\calG_{R_K}(\fkp_0)$-invariant, then
$(\sigma\otimes\sigma)(x)\in \B_{\st,K}(\Lambda_{\fkp}^+)\otimes_{\Q_p}V$
is $\calG_{R_K}(\fkp)$-invariant.
It follows that $\sigma\otimes\sigma$ restricts to a $K$-linear isomorphism
$D_{\B_{\st,K}(\Lambda_{\fkp_0})}(V|_{\calG_{R_K}(\fkp_0)})\stackrel{\cong}{\lra} D_{\B_{\st,K}(\Lambda_{\fkp})}(V|_{\calG_{R_K}(\fkp)})$.
Hence we have
\[
 \dim_KD_{\B_{\st,K}(\Lambda_{\fkp})}(V|_{\calG_{R_K}(\fkp)})=\dim_{\Q_p}V.
\]
By Proposition~\ref{prop:admissibility criterion} and Lemma~\ref{lem:regularity of period rings for Lambda fkp}, we conclude that 
$V|_{\calG_{R_K}(\fkp)}$ is $\B_{\st,K}(\Lambda_{\fkp}^+)$-admissible.

By a similar argument, we obtain the commutative diagram with isomorphic vertical maps
\[
\xymatrix{
\bigl(\B_{\st,K}^+(\Lambda_{\fkp_0}^+)\otimes_{\Q_p}V\bigr)^{\calG_{R_K}(\fkp_0)}
\ar[d]^{\cong}_{\sigma\otimes\sigma}\ar@{^{(}->}[r]
&
\bigl(\B_{\dR}^+(\Lambda_{\fkp_0}^+)\otimes_{\Q_p}V\bigr)^{\calG_{R_K}(\fkp_0)}
\ar[d]^{\cong}_{\sigma\otimes\sigma}\\
\bigl(\B_{\st,K}^+(\Lambda_{\fkp}^+)\otimes_{\Q_p}V\bigr)^{\calG_{R_K}(\fkp)}
\ar@{^{(}->}[r]
&
\bigl(\B_{\dR}^+(\Lambda_{\fkp}^+)\otimes_{\Q_p}V\bigr)^{\calG_{R_K}(\fkp)}.
} 
\]
Hence the second assertion follows from the one for $\fkp_0$.
 \end{proof}

Set
\[
 \Lambda_T^+:=\prod_{\fkp\in T}\Lambda_{\fkp}^+=\prod_{\fkp\in T}\widehat{\overline{R}_{\fkp}},\quad\text{and}\quad \Lambda_T:=\Lambda_T[p^{-1}].
\]
Hence $\Lambda_T$ is the product of $\Lambda_{\fkp}$ over $\fkp\in T$ as a $C$-Banach algebra. 

\begin{lem}\label{lem:localization map is injective with p-torsion free cokernel}
The map
\[
 \Lambda_R^+\ra \Lambda_T^+
\]
is injective with $p$-torsion free cokernel.
\end{lem}

\begin{proof}
Recall that $\overline{R}$ can be written as the union of Noetherian normal domains finite over $R$.
Since $R/\pi R$ is an integral domain, we deduce that the natural map
$\overline{R}/\pi\overline{R}\ra \prod_{\fkp\in T}\overline{R}_{\fkp}/\pi\overline{R}_{\fkp}$ is injective.
It follows that $\overline{R}/\pi^m\overline{R}\ra \prod_{\fkp\in T}\overline{R}_{\fkp}/\pi^m\overline{R}_{\fkp}$ is injective for every $m\in\N$. 
By taking the inverse limit over $m$, we see that 
$\Lambda_R^+\ra \prod_{\fkp\in T}\Lambda_{\fkp}^+=\Lambda_T^+$ is injective with $\pi$-torsion free cokernel. Hence the cokernel is also $p$-torsion free.
\end{proof}

Each $\sigma\in\calG_{R_K}$ induces an $R$-algebra isomorphism $\Lambda_{\fkp}^+\ra\Lambda_{\sigma(\fkp)}^+$ for $\fkp\in T$.
This gives rise to a natural $\calG_{R_K}$-action on $\Lambda_T^+$
making the injection $\Lambda_R^+\ra \Lambda_T^+$ $\calG_{R_K}$-equivariant (cf.~\cite[Remarque 3.3.2]{Brinon-crisdeRham}).

Since products commute with inverse limits and the Witt vector functor, we have
$\A_{\inf}(\Lambda_T^+)=\prod_{\fkp\in T}\A_{\inf}(\Lambda_{\fkp}^+)$.
Since $\calO_K$ is finite free over $W(k)$, we conclude
\[
 \A_{\inf,K}(\Lambda_T^+)=\prod_{\fkp\in T}\A_{\inf,K}(\Lambda_{\fkp}^+),
\]
and $\theta_{\A_{\inf,K}(\Lambda_T^+)}$ is the product of $\theta_{\A_{\inf,K}(\Lambda_{\fkp}^+)}$ over $\fkp\in T$.

\begin{lem}\label{lem:comparing periods of Lambda R and Lambda T}
 The natural commutative diagram
\[
\xymatrix{
 \B_{\st,K}^+(\Lambda_R^+)
\ar[r]\ar[d]
&
\B_{\dR}^+(\Lambda_R^+)
\ar[d]\\
\B_{\st,K}^+(\Lambda_T^+)
\ar[r]
&
\B_{\dR}^+(\Lambda_T^+)
}
\]
is Cartesian.
\end{lem}

\begin{proof}
 This follows from Lemmas~\ref{lem:Cartesian diagram of period rings} and \ref{lem:localization map is injective with p-torsion free cokernel}.
\end{proof}

For each $\fkp\in T$, choose $\sigma_\fkp\in\calG_{R_K}$ such that $\sigma_{\fkp}(\fkp_0)=\fkp$ and such that $\sigma_{\fkp}|_{\overline{K}}=\id$.
This is possible since $R/\pi R$ is geometrically integral over $k$.
We also assume $\sigma_{\fkp_0}=\id$.
We have an isomorphism
\[
 \sigma_{\fkp}\otimes\sigma_{\fkp}\colon
\B_{\st,K}^+(\Lambda_{\fkp_0}^+)\otimes_{\Q_p}V
\stackrel{\cong}{\lra}
\B_{\st,K}^+(\Lambda_{\fkp}^+)\otimes_{\Q_p}V
\]
compatible with the homomorphism $\calG_{R_K}(\fkp_0)\ra \calG_{R_K}(\fkp)$ given by $\sigma_{\fkp}$-conjugation.

\begin{lem}\label{lem:comparing semistable period of Lambda T and prod Lamda fkp}\hfill
\begin{enumerate}
 \item 
 The map
\[
\B_{\st,K}^+(\Lambda_T^+)
\ra \prod_{\fkp\in T}\B_{\st,K}^+(\Lambda_{\fkp}^+)
\]
is injective.
 \item 
If $x=(x_{\fkp})_{\fkp}\in \prod_{\fkp\in T}\bigl(\B_{\st,K}^+(\Lambda_{\fkp}^+)\otimes_{\Q_p}V\bigr)$
satisfies 
\[
 (\sigma_{\fkp}\otimes\sigma_{\fkp})(x_{\fkp_0})=x_{\fkp}
\]
 for every $\fkp\in T$,
then
\[
 x\in \B_{\st,K}^+(\Lambda_T^+)\otimes_{\Q_p}V.
\]
\end{enumerate}
\end{lem}

\begin{proof}
Consider $\theta_{\A_{\inf,K}(\Lambda_{\fkp_0}^+)}\colon \A_{\inf,K}(\Lambda_{\fkp_0}^+)\ra \Lambda_{\fkp_0}^+$ and choose an $\calO_K$-linear section $s_{\fkp_0}\colon \Lambda_{\fkp_0}^+\ra \A_{\inf,K}(\Lambda_{\fkp_0}^+)$ of $\theta_{\A_{\inf,K}(\Lambda_{\fkp_0}^+)}$  as in Construction~\ref{const:section of theta map}.
By the argument there, it defines a $K$-linear isomorphism
\[
\widetilde{\theta}_{v,\fkp_0}\colon \B_{\st,K}^+(\Lambda_{\fkp_0}^+)
\stackrel{\cong}{\lra}
\Lambda_{\fkp_0}\langle X\rangle[\log (1+X)].
\]

For each $\fkp\in T$, we set
\[
 s_{\fkp}:=\sigma_{\fkp}\circ s_{\fkp_0}\circ \sigma_{\fkp}^{-1}\colon 
\Lambda_{\fkp}^+\ra \Lambda_{\fkp_0}^+\ra \A_{\inf,K}(\Lambda_{\fkp_0}^+)\ra \A_{\inf,K}(\Lambda_{\fkp}^+).
\]
Then $s_{\fkp}$ is an $\calO_K$-linear section of $\theta_{\A_{\inf,K}(\Lambda_{\fkp}^+)}$ satisfying
$s_{\fkp}\circ \sigma_{\fkp}=\sigma_{\fkp}\circ s_{\fkp_0}$.
It also yields a $K$-linear isomorphism
\[
\widetilde{\theta}_{v,\fkp}\colon \B_{\st,K}^+(\Lambda_{\fkp}^+)
\stackrel{\cong}{\lra}
\Lambda_{\fkp}\langle X\rangle[\log (1+X)].
\]
Let $\sigma_{\fkp}\colon \Lambda_{\fkp_0}\langle X\rangle[\log (1+X)]\ra \Lambda_{\fkp}\langle X\rangle[\log (1+X)]$ denote the continuous $K$-linear map extending $\sigma_{\fkp}\colon \Lambda_{\fkp_0}\ra \Lambda_{\fkp}$ by $\sigma_{\fkp}(X^l(\log (1+X))^m)=X^l(\log (1+X))^m$.
We see $\widetilde{\theta}_{v,\fkp}\circ \sigma_{\fkp}=\sigma_{\fkp}\circ \widetilde{\theta}_{v,\fkp_0}$ 
 by Lemma~\ref{lem: properties of theta-tilde}(ii), Proposition~\ref{Prop: Colmez's description of semistable period ring}, and the assumption $\sigma_{\fkp}|_{\overline{K}}=\id$.

Similarly, 
\[
 \prod_{\fkp\in T}s_{\fkp}\colon \Lambda_T^+=\prod_{\fkp\in T}\Lambda_{\fkp}^+
\ra \prod_{\fkp\in T}\A_{\inf,K}(\Lambda_{\fkp}^+)=\A_{\inf,K}(\Lambda_T^+)
\]
is an $\calO_K$-linear section of $\theta_{\A_{\inf,K}(\Lambda_T^+)}$, and it yields
a $K$-linear isomorphism
\[
\widetilde{\theta}_{v,T}\colon \B_{\st,K}^+(\Lambda_T^+)
\stackrel{\cong}{\lra}
\Lambda_T\langle X\rangle[\log(1+X)].
\]
Then we have a commutative diagram
\[
 \xymatrix{
\B_{\st,K}^+(\Lambda_T^+)\ar[r]\ar[d]_{\widetilde{\theta}_{v,T}}^{\cong}
&\prod_{\fkp\in T}\B_{\st,K}^+(\Lambda_{\fkp}^+)\ar[d]_{\prod\widetilde{\theta}_{v,\fkp}}^{\cong}\\
\Lambda_T\langle X\rangle[\log(1+X)]\ar[r]
&\prod_{\fkp\in T}\Lambda_{\fkp}\langle X\rangle[\log(1+X)].
}
\]

 For $h\in\N$ and $\ast\in\{T,\fkp\}$, let $\A_{\st,K}^{\leq h}(\Lambda_\ast^+)$ denote the $\A_{\max,K}(\Lambda_\ast^+)$-submodule
$\bigoplus_{i=0}^h\A_{\max,K}(\Lambda_\ast^+)(\log[\pi^\flat]/\pi)^i$ of $\B_{\st,K}^+(\Lambda_\ast^+)$. Note $\B_{\st,K}^+(\Lambda_\ast^+)=\bigcup_h\A_{\st,K}^{\leq h}(\Lambda_\ast^+)[p^{-1}]$.
By Lemma~\ref{lem: properties of theta-tilde}(ii) and (iv), we have the functorial identification
\[
 \widetilde{\theta}_{v,\ast}\colon 
\A_{\st,K}^{\leq h}(\Lambda_\ast^+)\stackrel{\cong}{\lra}
\bigoplus_{i=0}^h\Lambda_\ast^+\langle X\rangle(\log(1+X))^i=:\Lambda_\ast^+\langle X\rangle[\log(1+X)]^{\leq h}
\]
as $\calO_K$-modules.

Recall
\[
 \Lambda_\ast^+\langle X\rangle
=\biggl\{\sum_{i=0}^\infty \lambda_iX^i\in \Lambda_{\ast}^+[[X]]\mid
\forall l\in\N, \lambda_i\in p^l\Lambda_{\ast}^+\, \text{for all but finitely many $i$}\biggr\}.
\]
In particular, $ \Lambda_T^+\langle X\rangle\ra \prod_{\fkp\in T}\bigl(\Lambda_{\fkp}^+\langle X\rangle\bigr)$ is injective.

We now prove Part (i). 
By the above remark and the transcendence of $\log (1+X)$ (Lemma~\ref{lem:transcendence of log}), we see that the map
\[
 \Lambda_T^+\langle X\rangle[\log(1+X)]^{\leq h}
\ra \prod_{\fkp\in T}\bigl(\Lambda_{\fkp}^+\langle X\rangle[\log(1+X)]^{\leq h}\bigr)
\]
is injective.
Hence the map $\A_{\st,K}^{\leq h}(\Lambda_T^+)\ra \prod_{\fkp\in T}\A_{\st,K}^{\leq h}(\Lambda_{\fkp}^+)$ is injective and Part (i) follows.

For Part (ii), consider the commutative diagram
\[
 \xymatrix{
\B_{\st,K}^+(\Lambda_{\fkp_0}^+)\otimes_{\Q_p}V
\ar[rr]^-{\widetilde{\theta}_{v,\fkp_0}\otimes\id}_-{\cong}\ar[d]^{\sigma_{\fkp}\otimes \sigma_{\fkp}}
&&
\Lambda_{\fkp_0}\langle X\rangle[\log(1+X)]\otimes_{\Q_p}V
\ar[d]^{\sigma_{\fkp}\otimes \sigma_{\fkp}}\\
\B_{\st,K}^+(\Lambda_{\fkp}^+)\otimes_{\Q_p}V
\ar[rr]^-{\widetilde{\theta}_{v,\fkp}\otimes\id}_-{\cong}
&&
\Lambda_{\fkp}\langle X\rangle[\log(1+X)]\otimes_{\Q_p}V.
}
\]
Then the assertion is equivalent to the following:
\begin{quote}
If $x=(x_{\fkp})_{\fkp}\in \prod_{\fkp\in T}\bigl(\Lambda_{\fkp}\langle X\rangle[\log(1+X)]\otimes_{\Q_p}V\bigr)$
satisfies $(\sigma_{\fkp}\otimes\sigma_{\fkp})(x_{\fkp_0})=x_{\fkp}$
 for every $\fkp\in T$, then
$x\in \Lambda_T\langle X\rangle[\log(1+X)]\otimes_{\Q_p}V$.
\end{quote}

Fix a $\calG_{R_K}$-stable $\Z_p$-lattice $V_{\Z_p}$ of $V$ and choose such $x=(x_{\fkp})_{\fkp}$. There exist $m, h\in\N$ such that $p^mx_{\fkp_0}\in \Lambda_{\fkp_0}^+\langle X\rangle[\log(1+X)]^{\leq h}\otimes_{\Z_p}V_{\Z_p}$.
Then for every $\fkp\in T$, we have
\[
 p^mx_{\fkp}= (\sigma_{\fkp}\otimes\sigma_{\fkp})(p^mx_{\fkp_0})
\in \Lambda_{\fkp}^+\langle X\rangle[\log(1+X)]^{\leq h}\otimes_{\Z_p}V_{\Z_p}.
\]
Moreover, we see that if there exists $l\in\N$ and $f(X)\in \Lambda_{\fkp_0}^+[X][\log(1+X)]^{\leq h}\otimes_{\Z_p}V_{\Z_p}$ (i.e., a polynomial in $X$ and $\log(1+X)$) such that
\[
 p^mx_{\fkp_0}-f(X)\in p^l\bigl(\Lambda_{\fkp_0}^+\langle X\rangle[\log(1+X)]^{\leq h}\otimes_{\Z_p}V_{\Z_p}\bigr), 
\]
then
\[
 p^mx_{\fkp}-(\sigma_{\fkp}\otimes\sigma_{\fkp})(f(X))\in p^l\bigl(\Lambda_{\fkp}^+\langle X\rangle[\log(1+X)]^{\leq h}\otimes_{\Z_p}V_{\Z_p}\bigr).
\]
This implies $p^mx\in \Lambda_T^+\langle X\rangle[\log(1+X)]^{\leq h}\otimes_{\Z_p}V_{\Z_p}$, and thus
$ x\in \Lambda_T\langle X\rangle[\log(1+X)]^{\leq h}\otimes_{\Q_p}V$.
\end{proof}

\begin{rem}\label{rem:comparing the integral structure of Lambda T and Lambda fkp}
With the notation as in the proof, the same argument as in the last part of the proof shows the following:
Let $D_{\calO_K}$ be a finite free $\calO_K$-module and let $W\subset \A_{\st,K}^{\leq h}(\Lambda_{\fkp_0}^+)\otimes_{\calO_K}D_{\calO_K}$ be a finite $\Z_p$-submodule.
Then the submodule $\prod_{\fkp\in T}\bigl((\sigma_{\fkp}\otimes\id)(W)\bigr)$
of $\prod_{\fkp\in T}(\A_{\st,K}^{\leq h}(\Lambda_{\fkp}^+)\otimes_{\calO_K}D_{\calO_K})$
lands in $\A_{\st,K}^{\leq h}(\Lambda_T^+)\otimes_{\calO_K}D_{\calO_K}$.
In particular, if $x=(x_\fkp)\in \prod_{\fkp\in T}\bigl(\A_{\st,K}^{\leq h}(\Lambda)\otimes_{\calO_K}D_{\calO_K}\bigr)$ satisfies $x_{\fkp}\in (\sigma_{\fkp}\otimes\id)(W)\subset \A_{\st,K}^{\leq h}(\Lambda_{\fkp}^+)\otimes_{\calO_K}D_{\calO_K}$ for every $\fkp\in T$, then $x\in \A_{\st,T}^{\leq h}(\Lambda_T^+)\otimes_{\calO_K}D_{\calO_K}$.
\end{rem}

\begin{lem}
 Let $x\in \bigl(\B_{\dR}^+(\Lambda_R^+)\otimes_{\Q_p}V)^{\calG_{R_K}}$ and let
$x_{\fkp}$ denote the image of $x$ in $\bigl(\B_{\dR}^+(\Lambda_{\fkp}^+)\otimes_{\Q_p}V\bigr)^{\calG_{R_K}(\fkp)}$.
Then for each $\fkp\in T$, we have
\[
 (\sigma_{\fkp}\otimes \sigma_{\fkp})(x_{\fkp_0})=x_{\fkp}.
\]
\end{lem}

\begin{proof}
 This follows from the commutative diagram
\[
 \xymatrix{
\B_{\dR}^+(\Lambda_R^+)\otimes_{\Q_p}V
\ar[r]^{\sigma_{\fkp}\otimes\sigma_{\fkp}}\ar[d]
&
\B_{\dR}^+(\Lambda_R^+)\otimes_{\Q_p}V
\ar[d]\\
\B_{\dR}^+(\Lambda_{\fkp_0}^+)\otimes_{\Q_p}V
\ar[r]^{\sigma_{\fkp}\otimes\sigma_{\fkp}}
&
\B_{\dR}^+(\Lambda_{\fkp}^+)\otimes_{\Q_p}V.
} 
\]
Namely, we have
$(\sigma_{\fkp}\otimes \sigma_{\fkp})(x_{\fkp_0})=\bigl((\sigma_{\fkp}\otimes\sigma_{\fkp})(x)\bigr)_{\fkp}=x_{\fkp}$.
\end{proof}

\begin{lem}
The inclusion
\[
 (\B_{\st,K}^+(\Lambda_R)\otimes_{\Q_p}V)^{\calG_{R_K}}
\subset\bigl(\B_{\dR}^+(\Lambda_R)\otimes_{\Q_p}V)^{\calG_{R_K}}.
\] 
is an equality.
\end{lem}

\begin{proof}
Let $x\in \bigl(\B_{\dR}^+(\Lambda_R^+)\otimes_{\Q_p}V)^{\calG_{R_K}}$.
For each $\fkp\in T$, let $x_{\fkp}$ denote the image of $x$ in $\bigl(\B_{\dR}^+(\Lambda_{\fkp}^+)\otimes_{\Q_p}V\bigr)^{\calG_{R_K}(\fkp)}$. Then we have
$ (\sigma_{\fkp}\otimes \sigma_{\fkp})(x_{\fkp_0})=x_{\fkp}$ by the previous lemma.
Since $\bigl(\B_{\st,K}^+(\Lambda_{\fkp}^+)\otimes_{\Q_p}V\bigr)^{\calG_{R_K}(\fkp)}=\bigl(\B_{\dR}^+(\Lambda_{\fkp}^+)\otimes_{\Q_p}V\bigr)^{\calG_{R_K}(\fkp)}$
by Lemma~\ref{lem:admissibility for Lambda fkp}, 
we have $(x_{\fkp})_{\fkp}\in \prod_{\fkp\in T}\bigl(\B_{\st,K}^+(\Lambda_{\fkp}^+)\otimes_{\Q_p}V\bigr)$.
Hence we get 
 $x\in \B_{\st,K}^+(\Lambda_T^+)\otimes_{\Q_p}V$ by Lemma~\ref{lem:comparing semistable period of Lambda T and prod Lamda fkp} (ii).
Since the commutative diagram
\[
\xymatrix{
 \B_{\st,K}^+(\Lambda_R^+)\otimes_{\Q_p}V
\ar[r]\ar[d]
&
\B_{\dR}^+(\Lambda_R^+)\otimes_{\Q_p}V
\ar[d]\\
 \B_{\st,K}^+(\Lambda_T^+)\otimes_{\Q_p}V
\ar[r]
&
\B_{\dR}^+(\Lambda_T^+)\otimes_{\Q_p}V
}
\]
is Cartesian by Lemma~\ref{lem:comparing periods of Lambda R and Lambda T},  we conclude
$x\in  \B_{\st,K}^+(\Lambda_R^+)\otimes_{\Q_p}V$.
\end{proof}

Let us complete the proof of Theorem~\ref{thm:purity for horiztaonlly semistable}.
Set 
\[
D:= (\B_{\st,K}^+(\Lambda_R^+)\otimes_{\Q_p}V)^{\calG_{R_K}}
=\bigl(\B_{\dR}^+(\Lambda_R^+)\otimes_{\Q_p}V)^{\calG_{R_K}}
=\bigl(\B_{\dR}(\Lambda_R^+)\otimes_{\Q_p}V)^{\calG_{R_K}}.
\] 
This is a $K$-vector space and $\dim_KD=\dim_{\Q_p}V$.
Consider the commutative diagram
\[
\xymatrix{
 \B_{\st,K}(\Lambda_R^+)\otimes_{K}D
\ar@{^{(}->}[rr]^{\alpha_{\st,K}^{\nabla}(V)}\ar@{^{(}->}[d]
&&
 \B_{\st,K}(\Lambda_R^+)\otimes_{\Q_p}V
\ar@{^{(}->}[d]\\
\B_{\dR}(\Lambda_R^+)\otimes_{K}D
\ar[rr]^{\alpha_{\dR}^{\nabla}(V)}_{\cong}
&&
\B_{\dR}(\Lambda_R^+)\otimes_{\Q_p}V.
}
\]
We need to show that $\alpha_{\st,K}^{\nabla}(V)$ is surjective.
Choose $m\in\N$ such that the $\Q_p$-vector subspace $\Q_pt^m\otimes_{\Q_p}V$ of
$\B_{\dR}(\Lambda_R^+)\otimes_{\Q_p}V$ is contained in
the image of $\B_{\dR}^+(\Lambda_R^+)\otimes_{K}D$ under $\alpha_{\dR}^{\nabla}(V)$.
Fix a $\calG_{R_K}$-stable $\Z_p$-lattice $V_{\Z_p}$ of $V$ and set
$V_{\Z_p}':=\Z_pt^m\otimes_{\Z_p}V_{\Z_p}$.
It suffices to prove that $V_{\Z_p}'$, as a submodule of $\B_{\st,K}(\Lambda_R^+)\otimes_{\Q_p}V$, is contained in the image of $\alpha_{\st,K}^{\nabla}(V)$.

For $h\in \N$ and $\ast\in\{R,T,\fkp\}$, consider the $\A_{\max,K}(\Lambda_\ast^+)$-submodule
\[
 \A_{\st,K}^{\leq h}(\Lambda_\ast^+):=\bigoplus_{i=0}^h\A_{\max,K}(\Lambda_\ast^+)\biggl(\log\frac{[\pi^\flat]}{\pi}\biggr)^i
\subset \B_{\st,K}^+(\Lambda_\ast^+).
\]
Note $\B_{\st,K}^+(\Lambda_\ast^+)=\bigcup_h\A_{\st,K}^{\leq h}(\Lambda_\ast^+)[p^{-1}]$.
Observe that each $\A_{\st,K}^{\leq h}(\Lambda_R^+)[p^{-1}]$ is stable under the action of $\calG_{R_K}$
and  
\[
\bigl(\A_{\st,K}^{\leq h}(\Lambda_R^+)[p^{-1}]\bigr)^{\calG_{R_K}}
=\bigl(\A_{\st,K}^{\leq h}(\Lambda_{\fkp}^+)[p^{-1}]\bigr)^{\calG_{R_K}(\fkp)}
=K.
\]

Since $D$ is finite-dimensional over $K$, there exists $h\in\N$ such that
\[
 D=\bigl(\A_{\st,K}^{\leq h}(\Lambda_R^+)[p^{-1}]\otimes_{\Q_p}V\bigr)^{\calG_{R_K}}.
\]
It follows that for each $\fkp\in T$, we have the isomorphism
\[
 D\stackrel{\cong}{\lra}
(\A_{\st,K}^{\leq h}(\Lambda_{\fkp}^+)[p^{-1}]\otimes_{\Q_p}V)^{\calG_{R_K}(\fkp)}
=\bigl(\B_{\dR}^+(\Lambda_{\fkp}^+)\otimes_{\Q_p}V)^{\calG_{R_K}(\fkp)}.
\]

By increasing $h$ if necessary, we may further assume that
$V_{\Z_p}'$, as a submodule of $\B_{\st,K}(\Lambda_{\fkp_0}^+)\otimes_{\Q_p}V$, 
is contained in the image of $\A_{\st,K}^{\leq h}(\Lambda_{\fkp_0}^+)[p^{-1}]\otimes_KD$
under the isomorphism 
\[
 \alpha_{\B_{\st,K}(\Lambda_{\fkp_0}^+)}(V|_{\calG_{R_K}(\fkp_0)})
\colon \B_{\st,K}(\Lambda_{\fkp_0}^+)\otimes_KD
\stackrel{\cong}{\lra}\B_{\st,K}(\Lambda_{\fkp_0}^+)\otimes_{\Q_p}V.
\]
Choose an $\calO_K$-lattice $D_{\calO_K}$ of $D$ such that
$V_{\Z_p}'$, as a submodule of $\B_{\st,K}(\Lambda_{\fkp_0}^+)\otimes_{\Q_p}V$, 
is contained in the image of $\A_{\st,K}^{\leq h}(\Lambda_{\fkp_0}^+)\otimes_{\calO_K}D_{\calO_k}$
under $\alpha_{\B_{\st,K}(\Lambda_{\fkp_0}^+)}(V|_{\calG_{R_K}(\fkp_0)})$.

For each $\fkp\in T$, the functoriality of the construction gives a commutative diagram
\[
 \xymatrix{
\A_{\st,K}^{\leq h}(\Lambda_{\fkp_0}^+)\otimes_{\calO_K}D_{\calO_K}
\ar[rrr]^{\alpha_{\B_{\st,K}(\Lambda_{\fkp_0}^+)}(V|_{\calG_{R_K}(\fkp_0)})}\ar[d]^{\sigma_{\fkp}\otimes \id}
&&&
\B_{\st,K}(\Lambda_{\fkp_0}^+)\otimes_{\Q_p}V
\ar[d]^{\sigma_{\fkp}\otimes \sigma_{\fkp}}\\
\A_{\st,K}^{\leq h}(\Lambda_{\fkp}^+)\otimes_{\calO_K}D_{\calO_K}
\ar[rrr]^{\alpha_{\B_{\st,K}(\Lambda_{\fkp}^+)}(V|_{\calG_{R_K}(\fkp)})}
&&&
\B_{\st,K}(\Lambda_{\fkp}^+)\otimes_{\Q_p}V.
}
\]
Since $V_{\Z_p}'$ is $\calG_{R_K}$-stable, we conclude that 
$V_{\Z_p}'$, as a submodule of $\B_{\st,K}^+(\Lambda_{\fkp}^+)\otimes_{\Q_p}V$, 
is contained in the image of $\A_{\st,K}^{\leq h}(\Lambda_{\fkp}^+)\otimes_{\calO_K}D_{\calO_k}$
under $\alpha_{\B_{\st,K}(\Lambda_{\fkp}^+)}(V|_{\calG_{R_K}(\fkp)})$.

Consider the isomorphism
\[
 \alpha_{\dR}^{\nabla}(V)\colon
\B_{\dR}(\Lambda_R^+)\otimes_KD\stackrel{\cong}{\lra}
\B_{\dR}(\Lambda_R^+)\otimes_{\Q_p}V
\]
and let $W\subset \B_{\dR}(\Lambda_R^+)\otimes_KD$ denote
 the inverse image of $V_{\Z_p}'$, as a submodule of $\B_{\dR}(\Lambda_R^+)\otimes_{\Q_p}V$, under $\alpha_{\dR}^{\nabla}(V)$.
Observe $W\subset \B_{\dR}^+(\Lambda_R^+)\otimes_KD$.
By the above discussion, the image of $W$ under the map
\[
 \B_{\dR}^+(\Lambda_R^+)\otimes_KD
\ra \prod_{\fkp\in T}(\B_{\dR}^+(\Lambda_{\fkp}^+)\otimes_KD)
\]
lands in the subspace $\prod_{\fkp\in T}\bigl(\A_{\st,K}^{\leq h}(\Lambda_{\fkp}^+)\otimes_{\calO_K}D_{\calO_K}\bigr)$.
It follows from Remark~\ref{rem:comparing the integral structure of Lambda T and Lambda fkp} that the image of $W$ under the map
\[
 \B_{\dR}^+(\Lambda_R^+)\otimes_KD
\ra  \B_{\dR}^+(\Lambda_T^+)\otimes_KD
\]
lands in the subspace $\A_{\st,K}^{\leq h}(\Lambda_T^+)\otimes_{\calO_K}D_{\calO_K}$ and thus in $\B_{\st,K}^+(\Lambda_T^+)\otimes_KD$.
By Lemma~\ref{lem:comparing periods of Lambda R and Lambda T}, we conclude that
$W$ lands in the subspace
\[
 \B_{\st,K}^+(\Lambda_R^+)\otimes_KD \subset \B_{\dR}^+(\Lambda_R^+)\otimes_KD.
\]
This implies that
 $V_{\Z_p}'$, as a submodule of $\B_{\st,K}(\Lambda_R^+)\otimes_{\Q_p}V$,
is contained in the image of $\alpha_{\st,K}^{\nabla}(V)$.
Hence $V$ is a horizontal semistable representation of $\calG_{R_K}$.
This completes the proof of Theorem~\ref{thm:purity for horiztaonlly semistable}.

\section{$p$-adic monodromy theorem for horizontal de Rham representations}\label{section:p-adic monodromy theorem for horizontal de Rham representations}

\subsection{Horizontal version of $p$-adic monodromy theorem for Galois representations}
Let us recall a result of Ohkubo.
We use a slightly different notation from \cite{Ohkubo}.
Let $\calK$ be a complete discrete valuation field of mixed characteristic $(0,p)$
and let $k_{\calK}$ denote the residue field of $\calK$\footnote{\cite{Ohkubo} uses $K$ for our $\calK$. In our paper, $K$ is a finite extension of $\Q_p$.}.
Fix an algebraic closure $\overline{\calK}$ of $\calK$.
Define the subfield $\calK_{\can}$ of $\calK$ as the algebraic closure of
$W(k_{\calK}^{p^\infty})[p^{-1}]$ in $\calK$, 
where $k_{\calK}^{p^\infty}:=\bigcap_{m\in\N}k_{\calK}^{p^m}$
Note that $\calK_{\can}$ coincides with the kernel of the canonical derivation
$d\colon \calK\ra \calK\otimes_{\calO_{\calK}}\widehat{\Omega}_{\calO_{\calK}}$ (cf.~\cite[Remark 1.4 (i), Proposition 1.13]{Ohkubo}).
We use a similar notation for finite extensions of $\calK$.

Ohkubo proves the $p$-adic monodromy theorem for horizontal de Rham representations of $\calK$:

\begin{thm}[{\cite[Theorem 7.4]{Ohkubo}}]
\label{thm:horizontal p-adic monodromy}
 Let $V_{\calK}$ be a horizontal de Rham representation of $\Gal(\overline{\calK}/\calK)$.
Then there exists a finite extension $L$ of $\calK_{\can}$ such that
$V_{\calK}|_{\Gal(\overline{\calK}/\calK L)}$ is horizontal semistable, where
$\calK L$ is the composite field of $\calK$ and $L$ in $\overline{\calK}$.
\end{thm}

\begin{rem}
 When $k_{\calK}$ is perfect, the above theorem is the usual $p$-adic monodromy theorem, and it was first proved by Berger \cite[Th\'eor\`eme 0.7] {Berger}.
When $k_{\calK}$ is imperfect, we can also consider the statement that a de Rham representation is potentially semistable.
This statement is proved by Morita when $k_{\calK}$ admits a finite $p$-basis \cite[Corollary 1.2]{Morita} and by Ohkubo in general \cite[Main Theorem]{Ohkubo}.
Ohkubo deduces Theorem~\ref{thm:horizontal p-adic monodromy} from this version of $p$-adic monodromy theorem.
\end{rem}

\begin{rem}
\cite[Theorem 7.6]{Ohkubo} claims that every horizontal de Rham representation of $\Gal(\overline{\calK}/\calK)$ comes from a de Rham representation of $\Gal(\overline{\calK_{\can}}/\calK_{\can})$ via the pullback along $\Gal(\overline{\calK}/\calK)\ra \Gal(\overline{\calK_{\can}}/\calK_{\can})$, which is not correct. The error comes from an incorrect argument in the proof of \cite[Proposition 7.3]{Ohkubo} (cf.~Remark~\ref{rem:horizontal comparison does not preserve filtrations} and Example~\ref{example: non-compatibility of filtrations}).
However, these errors are irrelevant to the proofs of the $p$-adic monodromy theorem and the horizontal $p$-adic monodromy theorem \cite[Main Theorem, Theorem 7.4]{Ohkubo}, and the latter theorems are correct.
\end{rem}

\subsection{$p$-adic monodromy theorem for horizontal de Rham representations}

Let $K$ be a complete discrete valuation field of characteristic zero with perfect residue field $k$ of characteristic $p$. Fix a uniformizer $\pi$ of $K$.
 Let $R$ be an $\calO_K$-algebra satisfying the conditions in Set-up~\ref{set-up: Brinon's rings}. We further assume that $R$ satisfies Condition (BR).
For each finite extension $L$ of $K$,
we denote $\calO_L\otimes_{\calO_K}R$ by $R_{\calO_L}$.
Then $R_{\calO_L}$ also satisfies the conditions in Set-up~\ref{set-up: Brinon's rings} and Condition (BR) with respect to $L$.
As before, we write $R_K$ (resp. $R_L$) for $R[p^{-1}]$ (resp. $R_{\calO_L}[p^{-1}]$).

Theorem~\ref{thm:p-adic monodromy for horizontal de Rham representations in intro} in the introduction
is a special case of the following $p$-adic monodromy theorem for horizontal de Rham representations:
\begin{thm}\label{thm:p-adic monodromy for horizontal de Rham representations}
Assume that $R$ satisfies Condition (BR). Let $V\in \Rep_{\Q_p}(\calG_{R_K})$.
If $V$ is horizontal de Rham, then there exists a finite extension $L$ of $K$
such that 
$V|_{\calG_{R_L}}$ is horizontal semistable.
\end{thm}

\begin{proof}
We use the notation in Subsection~\ref{subsection:statement of purity}.
Consider the complete discrete valuation field $\calK:=\widehat{R_{(\pi)}}[p^{-1}]$.
Fix an algebraic closure $\overline{\calK}$ of $\calK$ and an $R$-algebra embedding
$\overline{R}\hra \overline{\calK}$.
Note $\calK_{\can}=K$ by Lemma~\ref{lem:canonical subfield}.

Since $V$ is a horizontal de Rham representation of $\calG_{R_K}$,
$V_{\calK}:=V|_{\Gal(\overline{\calK}/\calK)}$ is a horizontal de Rham representation of $\Gal(\overline{\calK}/\calK)$.
Hence by Theorem~\ref{thm:horizontal p-adic monodromy},
there exists a finite extension $L$ of $K$ such that
$V_{\calK}|_{\Gal(\overline{\calK}/L\calK)}$ is horizontal semistable.

We prove that $V|_{\calG_{R_L}}$ is horizontal semistable.
Let $\pi_L$ denote a uniformizer of $L$.
Since $R/\pi R$ is geometrically integral over $k$, 
$(\pi_L)\subset R_{\calO_L}$ is a prime ideal and thus
the ring of integers of $\calK L$ coincides with the $p$-adic completion $\widehat{(R_{\calO_L})_{(\pi_L)}}$ of the localization of $R_{\calO_L}$ at $(\pi_L)$. Hence
$\calK L=\widehat{(R_{\calO_L})_{(\pi_L)}}[p^{-1}]$.

We know that $V|_{\calG_{R_L}}|_{\Gal(\overline{\calK}/L\calK)}=V_{\calK}|_{\Gal(\overline{\calK}/L\calK)}$
is a horizontal semistable representation of $\Gal(\overline{\calK}/L\calK)$.
The embedding 
\[
\overline{R_{\calO_L}}= \overline{R}\hra \overline{\calK}
\]
determines a prime ideal $\fkp_0$ of $\overline{R_{\calO_L}}$ above $(\pi_L)\subset R_{\calO_L}$,
and $\Gal(\overline{\calK}/L\calK)$ coincides with $\widehat{\calG}_{R_L}(\fkp_0)$.
Therefore $V|_{\calG_{R_L}}$ is horizontal semistable by Theorem~\ref{thm:purity for horiztaonlly semistable}.
\end{proof}

\begin{thm}\label{thm:horizontal de Rham and one point crystalline}
Assume $R$ satisfies Condition (BR).
 Let $V$ be a horizontal de Rham representation of $\calG_{R_K}$.
Assume that there exists a finite extension $K'$ of $K$ and a map $f\colon R\ra K'$
such that $f^\ast V$ is a potentially crystalline representation of $\Gal(\overline{K}/K')$.
Then there exists a finite extension $L$ of $K$ such that
$V|_{\calG_{R_L}}$ is horizontal crystalline.
\end{thm}

\begin{proof}
By extending $K'$ if necessary, we may assume that the representation $f^\ast V$ of $\Gal(\overline{K}/K')$ is crystalline (equivalently, horizontal crystalline). 
By Theorem~\ref{thm:p-adic monodromy for horizontal de Rham representations},
there exists a finite extension $L$ of $K'$ such that $V|_{\calG_{R_L}}$ is horizontal semistable. Then the assertion follows from Proposition~\ref{prop:horizontal semistable plus crystalline at one point}.
\end{proof}

\section{$p$-adic Hodge theory for rigid analytic varieties}\label{section:p-adic Hodge theory for rigid analytic varieties}

In this section, we review $p$-adic local systems and period sheaves.
Let $K$ be a complete discrete valuation field of characteristic zero with perfect residue field of characteristic $p$.

We denote the $n$-dimensional unit polydisk $\Spa\bigl(K\langle T_1,\ldots,T_n\rangle,\calO_K\langle T_1,\ldots,T_n\rangle\bigr)$ by $\bB_K^n$.
We also write $\bT_K^n$ for $\Spa(K\langle T_1^{\pm 1},\ldots,T_n^{\pm 1}\rangle,\calO_K\langle T_1^{\pm 1},\ldots,T_n^{\pm 1}\rangle)$, the $n$-dimensional rigid torus.

Let $X$ be a smooth adic space over $\Spa(K,\calO_K)$. 
We call a point $x$ of $X$ \emph{classical} if the residue class field $k(x)$ of $x$ is finite over $K$.

We work on the \'etale site $X_{\et}$ (cf.~\cite[Definition 2.3.1]{Huber-etale}) and the pro-\'etale site $X_{\proet}$. For the pro-\'etale site, we use the one defined in \cite{Scholze-p-adicHodge, Scholze-p-adicHodgeerrata}, and thus we have the projection 
\[
 \nu\colon X_{\proet}\ra X_{\et}.
\]
Recall that $U\in X_{\proet}$ is called \emph{affinoid perfectoid} if $U$ has a pro-\'etale presentation $U=\varprojlim U_i\ra X$ by affinoid $U_i=\Spa(A_i,A_i^+)$ such that, denoting by $A^+$ the $p$-adic completion of $\varinjlim A_i^+$, and $A=A^+[p^{-1}]$, the pair $(A,A^+)$ is a perfectoid affinoid $(L,L^+)$-algebra for a perfectoid field $L$ over $K$ and an open and bounded valuation subring $L^+\subset L$ (cf.~\cite[Definition 4.3 (i)]{Scholze-p-adicHodge}). In this case, we write $\widehat{U}$ for $\Spa(A,A^+)$, which is independent of the pro-\'etale presentation $U=\varprojlim U_i$.

\subsection{$p$-adic local systems on $X_{\et}$}\label{subsection:local systems}

Let us clarify our terminology of local systems.

\begin{defn}[cf.~{\cite[Definition 8.1]{Scholze-p-adicHodge}, \cite[\S 1.4, \S 8.4]{KL-I}}]
A \emph{lisse $\Z_p$-sheaf} $\bL=(\bL_n)$ on $X_{\et}$ is an inverse system of sheaves of $\Z/p^n$-modules $\bL_n$ on $X_{\et}$ such that each $\bL_n$ is locally on $X_{\et}$ a constant sheaf associated with a finitely generated flat $\Z/p^n$-module and such that this inverse system is isomorphic in the procategory to an inverse system for which $\bL_{n+1}/p^n\cong \bL_n$.
We also call a lisse $\Z_p$-sheaf on $X_{\et}$ an \emph{\'etale $\Z_p$-local system} on $X$ or a $\Z_p$-local system on $X_{\et}$.
\end{defn}

\begin{defn}[cf.~{\cite[\S 1.4, \S 8.4]{KL-I}}]
 An \emph{isogeny $\Z_p$-local system} on $X_{\et}$ is an object of the isogeny category of $\Z_p$-local systems on $X_{\et}$.
A \emph{$\Q_p$-local system} on $X_{\et}$ is an object of the stack associated to the fibered category of isogeny $\Z_p$-local systems on $X_{\et}$.
We also call a $\Q_p$-local system on $X_{\et}$ an \emph{\'etale $\Q_p$-local system} on $X$.
\end{defn}

\begin{rem}
Assume that $X=\Spa (A,A^+)$ is an affinoid.
In this case, the category of adic spaces which are finite and \'etale over $\Spa(A,A^+)$ is canonically equivalent to the category of schemes which are finite and \'etale over $\Spec A$
by \cite[Example 1.6.6 ii)]{Huber-etale}.
The natural functor from the category of lisse $\Z_p$-sheaves on the \'etale site $(\Spec A)_{\et}$ to the category of $\Z_p$-local systems on $\Spa(A,A^+)_{\et}$ is an equivalence of categories (cf.~\cite[Remark 8.4.5]{KL-I}).

Moreover, if $\Spec A$ is connected and $\overline{x}$ is a geometric point of $\Spec A$, we have the \'etale fundamental group $\pi_1^{\et}(\Spec A,\overline{x})$.
Then the category of continuous representations of $\pi_1^{\et}(\Spec A,\overline{x})$ in finite free $\Z_p$-modules is equivalent to the category of $\Z_p$-local systems on $\Spa(A,A^+)_{\et}$. Similarly, the category of continuous representations of $\pi_1^{\et}(\Spec A,\overline{x})$ in finite-dimensional $\Q_p$-vector spaces is equivalent to the category of isogeny $\Z_p$-local systems on $\Spa(A,A^+)_{\et}$ (cf.~\cite[Remark 1.4.4]{KL-I}).

In particular, if $A^+$ satisfies conditions in Set-up~\ref{set-up: Brinon's rings} relative to $K$, then the category $\Rep_{\Q_p}(\calG_A)$ is equivalent to the category of isogeny $\Z_p$-local systems on $\Spa(A,A^+)_{\et}$.
\end{rem}

\begin{defn}[{\cite[Definition 8.1]{Scholze-p-adicHodge}}]
Let $\widehat{\Z}_p=\varprojlim \Z/p^n$ as sheaves on $X_{\proet}$. Then a lisse $\widehat{\Z}_p$-sheaf on $X_{\proet}$ is a sheaf of $\widehat{\Z}_p$-modules on $X_{\proet}$ such that locally in $X_{\proet}$, it is isomorphic to $\widehat{\Z}_p\otimes_{\Z_p}M$, where $M$ is a finitely generated $\Z_p$-module.

Similarly, set $\widehat{\Q}_p=\widehat{\Z}_p[p^{-1}]$. 
A lisse $\widehat{\Q}_p$-sheaf on $X_{\proet}$ is a sheaf of $\widehat{\Q}_p$-modules on $X_{\proet}$ such that locally in $X_{\proet}$, it is isomorphic to $\widehat{\Q}_p\otimes_{\Q_p}M$, where $M$ is a finite-dimensional $\Q_p$-vector space.
\end{defn}

For a $\Z_p$-local system $\bL=(\bL_n)$ on $X_{\et}$, we set 
\[
 \widehat{\bL}=\varprojlim \nu^\ast\bL_n.
\]
Then $\widehat{\bL}$ is a lisse $\widehat{\Z}_p$-sheaf on $X_{\proet}$.
This functor is an equivalence of categories. Moreover, $R^j\varprojlim\nu^\ast\bL_n=0$ for $j>0$ by \cite[Proposition 8.2]{Scholze-p-adicHodge}.
Similarly, we can naturally associate to every $\Q_p$-local system $\bL$ on $X_{\et}$
a lisse $\widehat{\Q}_p$-sheaf on $X_{\proet}$, which we denote by $\widehat{\bL}$.

\subsection{Period sheaves and the Riemann--Hilbert functor $D_{\dR}$}

We will use the following sheaves on $X_{\proet}$:
\begin{itemize}
 \item The de Rham sheaf $\B_{\dR}$ (cf.~\cite[Definition 6.1(iii)]{Scholze-p-adicHodge}).
By \cite[Theorem 6.5(i)]{Scholze-p-adicHodge}, for an affinoid perfectoid $U\in X_{\proet}$ with $\widehat{U}=\Spa (\Lambda,\Lambda^+)$, we have
\[
 \B_{\dR}(U)=\B_{\dR}(\Lambda,\Lambda^+),
\]
where the latter is the period ring defined in Section~\ref{section:period rings}, which agrees with the one defined in \cite[page 49]{Scholze-p-adicHodge}.
 \item The structural de Rham sheaf $\OB_{\dR}$ (cf.~\cite{Scholze-p-adicHodge, Scholze-p-adicHodgeerrata}). 
This sheaf is a $\B_{\dR}$-algebra equipped with a filtration and a $\B_{\dR}$-linear connection
\[
 \nabla\colon \OB_{\dR}\ra \OB_{\dR}\otimes_{\calO_X}\Omega^1_X
\]
satisfying the Griffiths transversality and $(\OB_{\dR})^{\nabla=0}=\B_{\dR}$.
\end{itemize}

We have the following description of $\OB_{\dR}$:
Let $\Spa(A,A^+)$ be an affinoid admitting standard \'etale maps to both $X$ and $\bT^n_L$ for some finite extension $L$ over $K$.
If we replace $L$ with its algebraic closure in $A$, then $R=A^+$ satisfies conditions in Set-up~\ref{set-up: Brinon's rings} relative to $L$.
Write $\overline{R}=\varinjlim_{i\in I}A_i^+$ where $A_i^+$ is a finite normal $R$-subalgebra of $\overline{\Frac R}$ such that $A_i=A_i^+[p^{-1}]$ is \'etale over $A=R_K$.
Set $U:=\varprojlim_i\Spa(A_i,A_i^+)\in X_{\proet}$; this is affinoid perfectoid with $\widehat{U}=\Spa (\widehat{\overline{R}}[p^{-1}],\widehat{\overline{R}})$. 
In this setting, $\B_{\dR}(U)=\B_{\dR}(\widehat{\overline{R}})=\rmB_{\dR}^\nabla(R)$, and
there is a natural $\rmB_{\dR}^\nabla(R)$-linear isomorphism
\[
 \OB_{\dR}(U)\cong \rmB_{\dR}(R),
\]
compatible with filtrations and connections, 
where the right hand side is the de Rham period ring recalled in Subsection~\ref{subsection:de Rham representations}. To see this, note that $\OB_{\dR}$ is defined to be the sheafification of the presheaf defined in \cite[(3)]{Scholze-p-adicHodgeerrata}. From the definitions, we have a $\rmB_{\dR}^\nabla(R)$-algebra homomorphism $\rmB_{\dR}(R)\ra \OB_{\dR}(U)$. By comparing \cite[Proposition 6.10]{Scholze-p-adicHodge}, \cite[(3)]{Scholze-p-adicHodgeerrata}, and \cite[Proposition 5.2.2]{Brinon-crisdeRham}, we conclude that this map is an isomorphism compatible with filtrations and connections.

\begin{lem}\label{lem:affinoid perfectoid vanishing of structural de Rham sheaf}
Let $U\in X_{\proet}$ be affinoid perfectoid.
Then $H^q(U,\OB_{\dR})=0$ for every $q>0$.
\end{lem}

\begin{proof}
Consider the positive structural de Rham sheaf $\OB_{\dR}^+$ (\cite[Definition 6.8 (iii)]{Scholze-p-adicHodge}).
It is equipped with a decreasing filtration $\Fil^\bullet \OB_{\dR}^+$ of $\OB_{\dR}^+$ defined by $\Fil^i\OB_{\dR}^+=(\Ker\theta)^i$.
Then $\OB_{\dR}^+=\varprojlim_i \OB_{\dR}^+/\Fil^i\OB_{\dR}^+$
and $\gr^\bullet\OB_{\dR}^+$ is naturally a polynomial algebra over $\widehat{\calO}_X$ by \cite[Proposition 6.10]{Scholze-p-adicHodge}.

For every affinoid perfectoid $U$, we have
$H^q(U,\widehat{\calO}_X)=0$ for $q>0$ by \cite[Lemma 4.10 (v)]{Scholze-p-adicHodge}.
Hence $H^q(U,\Fil^i\OB_{\dR}^+/\Fil^{i+1}\OB_{\dR}^+)=0$ for $q>0$ by coherence of $U$ (\cite[Proposition 3.12 (iii)]{Scholze-p-adicHodge}). It follows that $H^0(U,\OB_{\dR}^+/\Fil^{i+1}\OB_{\dR}^+)\ra H^0(U,\OB_{\dR}^+/\Fil^i\OB_{\dR}^+)$ is surjective.
By \cite[Lemma 3.18, Proposition 4.8]{Scholze-p-adicHodge}, we have
$H^q(U,\OB_{\dR}^+)=0$ for every $q>0$.
Since $\OB_{\dR}=\OB_{\dR}^+[t^{-1}]$ is written as $\varinjlim\OB_{\dR}^+$ with multiplication-by-$t$ transition maps,  
we conclude $H^q(U,\OB_{\dR})=0$ for every $q>0$ by coherence of $U$.
\end{proof}

\begin{defn}[cf.~{\cite[Definition 8.3]{Scholze-p-adicHodge}}]
 Let $\bL$ be a $\Q_p$-local system on $X_{\et}$.
We say that $\bL$ is \emph{de Rham} if there exits a filtered $\calO_X$-module $\calE$ with an integrable connection such that there is an isomorphism of sheaves on $X_{\proet}$
\[
 \widehat{\bL}\otimes_{\widehat{\Q}_p}\OB_{\dR}\cong \calE\otimes_{\calO_X}\OB_{\dR}
\]
compatible with filtrations and connections.
Similarly, we say that a(n isogeny) $\Z_p$-local system on $X_{\et}$ is de Rham if the associated $\Q_p$-local system is de Rham.
\end{defn}

Based on the works of Kedlaya--Liu and Scholze, Liu and Zhu proved the following theorem:

\begin{thm}[{\cite[Theorem 3.9]{Liu-Zhu}}]
For a $\Q_p$-local system $\bL$ on $X_{\et}$, set
\[
 D_{\dR}(\bL):=\nu_{\ast}(\widehat{\bL}\otimes_{\widehat{\Q}_p}\OB_{\dR}).
\]
\begin{enumerate}
 \item $D_{\dR}(\bL)$ is a vector bundle on $X_{\et}$ equipped with an integrable connection
\[
 \nabla_{\bL}:=\nu_\ast(\id\otimes \nabla)\colon
D_{\dR}(\bL)\ra D_{\dR}(\bL)\otimes_{\calO_{X_{\et}}}\Omega_{X_{\et}}
\]
and a decreasing filtration $\Fil^\bullet$ satisfying the Griffiths transversality.
 \item  The functor $\bL\mapsto (D_{\dR}(\bL),\nabla_{\bL},\Fil^\bullet)$
commutes with pullbacks.
 \item The following conditions are equivalent for $\bL$:
\begin{enumerate}
 \item $\bL$ is de Rham;
 \item $\rank D_{\dR}(\bL)=\rank \bL$;
 \item for each connected component, there exists a classical point $x$ of $X$ and a geometric point $\overline{x}$ above $x$ such that the stalk $\bL_{\overline{x}}$ is a de Rham representation of the absolute Galois group of $k(x)$.
\end{enumerate}
In this case, the natural map
$\nu^\ast D_{\dR}(\bL)\otimes_{\calO_X}\OB_{\dR}\ra \widehat{\bL}\otimes_{\widehat{\Q}_p}\OB_{\dR}$ is an isomorphism of sheaves on $X_{\proet}$ compatible with filtrations and connections.
\end{enumerate}
\end{thm}

\begin{rem}
The category of vector bundles on $X$ (with the analytic topology) is equivalent to the category of vector bundles on $X_{\et}$ by \cite[Lemma 7.3]{Scholze-p-adicHodge}.
Hence we also regard $D_{\dR}(\bL)$ as a vector bundle on $X$.
\end{rem}

\begin{lem}\label{lem:de Rham local system vs representation}
 Assume that $X=\Spa(A,A^+)$ is affinoid  such that $R=A^+$ satisfies conditions in Set-up~\ref{set-up: Brinon's rings} relative to $K$ and it is equipped with the $p$-adic topology.
Let $\bL$ be an  isogeny $\Z_p$-local system on $X_{\et}$ and let 
$V\in \Rep_{\Q_p}(\calG_A)$ be the corresponding representation.
Then the vector bundle $D_{\dR}(\bL)$ is associated to the $A$-module $D_{\dR}(V)$
and this identification is compatible with connections and filtrations.
Moreover, $\bL$ is de Rham if and only if $V$ is a de Rham representation.
\end{lem}

\begin{proof}
 Since $D_{\dR}(\bL)$ is a coherent sheaf on $X$, it suffices to show that
$\Gamma(X,D_{\dR}(\bL))$ is naturally identified with $D_{\dR}(V)$.
By definition, we have
\[
 \Gamma(X,D_{\dR}(\bL))=\Gamma(X_{\proet}, \widehat{\bL}\otimes_{\widehat{\Q}_p}\OB_{\dR}).
\]

Write $\overline{R}=\varinjlim_{i\in I}A_i^+$ where $A_i^+$ is a finite normal $R$-subalgebra of $\overline{\Frac R}$ such that $A_i=A_i^+[p^{-1}]$ is \'etale over $A=R_K$.
Set $U:=\varprojlim_i\Spa(A_i,A_i^+)\in X_{\proet}$; this is affinoid perfectoid with $\widehat{U}=\Spa (\widehat{\overline{R}}[p^{-1}],\widehat{\overline{R}})$. 
Since every finite \'etale $A$-algebra splits over $\overline{R}[p^{-1}]$, the lisse $\widehat{\Q}_p$-sheaf $\widehat{\bL}$ is trivial over $U$ and naturally identified with $\widehat{\Q}_p\otimes_{\Q_p}V$.

For each $i\in\N$, the $i$-fold fiber product $U^{i/X}:=U\times_XU\times\cdots\times_XU$ of $U$ over $X$ exists and is affinoid perfectoid. 
By Lemma~\ref{lem:affinoid perfectoid vanishing of structural de Rham sheaf}, we have $H^q(U^{i/X},V\otimes_{\Q_p}\OB_{\dR})=0$ for every $q>0$ and $i\in\N$.
Applying the Cartan--Leray spectral sequence to the cover $U\ra X$ in $X_{\proet}$, we compute
\begin{align*}
 \Gamma(X,D_{\dR}(\bL))
&=H^0_{\cont}(\calG_A, \Gamma(U, \widehat{\bL}\otimes_{\widehat{\Q}_p}\OB_{\dR}))
=H^0_{\cont}(\calG_A, \Gamma(U, V\otimes_{\Q_p}\rmB_{\dR}(R)))\\
&=D_{\dR}(V).
\end{align*}
It is easy to see that this identification is compatible with connections and filtrations.

If $V$ is a de Rham representation, then $D_{\dR}(V)$ is a finite projective $A$-module of rank equal to $\dim_{\Q_p}V=\rank \bL$. Hence $\bL$ is de Rham.
Conversely, if $\bL$ is de Rham, we have an isomorphism 
$\nu^\ast D_{\dR}(\bL)\otimes_{\calO_X}\OB_{\dR}\ra \widehat{\bL}\otimes_{\widehat{\Q}_p}\OB_{\dR}$ of sheaves on $X_{\proet}$.
Evaluating both sides at $U\in X_{\proet}$ yields
$D_{\dR}(V)\otimes_A\rmB_{\dR}(R)\stackrel{\cong}{\lra}V\otimes_{\Q_p}\rmB_{\dR}(R)$
and this is identified with $\alpha_{\dR}(V)$. Hence $V$ is a de Rham representation.
\end{proof}

\section{A $p$-adic monodromy theorem for de Rham local systems}\label{section:p-adic monodromy theorem for de Rham local systems}

Let $K$ be a complete discrete valuation field of characteristic zero with perfect residue field of characteristic $p$.
Recall that a rigid analytic variety over $K$ refers to a quasi-separated adic space locally of finite type over $\Spa(K,\calO_K)$ (cf.~ \cite[Proposition 4.5(iv)]{Huber-gen}).
For a rigid analytic variety $X$ over $K$ and a complete extension $L$ of $K$, we write $X_L$ for $X\times_{\Spa(K,\calO_K)}\Spa(L,\calO_L)$.

\subsection{Statement of the theorem}
\begin{defn}
 Let $X$ be a smooth rigid analytic variety over $K$, and let $\bL$ be an isogeny $\Z_p$-local system on $X_{\et}$.
\begin{enumerate}
 \item We say that $\bL$ is \emph{horizontal de Rham} if $\bL$ is de Rham and 
$(D_{\dR}(\bL),\nabla)$ has a full set of solutions, i.e., $(D_{\dR}(\bL),\nabla)\cong (\calO_X,d)^{\rank \bL}$.
 \item Let $x$ be a classical point of $X$. We say that $\bL$ is \emph{semistable} (resp.~\emph{(potentially) crystalline}) at $x$ if there exists a geometric point $\overline{x}$ above $x$ such that the stalk $\bL_{\overline{x}}$ is a semistable (resp.~(potentially) crystalline) representation of the absolute Galois group of $k(x)$. This is equivalent to the condition that the stalk $\bL_{\overline{x}}$ is a semistable (resp.~(potentially) crystalline) representation for every geometric point $\overline{x}$ above $x$.
\end{enumerate}
\end{defn}

Theorem~\ref{thm:main thm in introduction} in the introduction is a special case of the following $p$-adic monodromy theorem for de Rham local systems:

\begin{thm}\label{thm:p-adic monodromy}
 Let $X$ be a smooth rigid analytic variety over $K$, and let $\bL$ be a de Rham isogeny $\Z_p$-local system on $X_{\et}$.
Then for every classical point $x\in X$, there exists an open neighborhood $U\subset X_{k(x)}$ of $x$ and a finite extension $L$ of $k(x)$
such that $\bL|_{U_L}$ is horizontal de Rham and semistable at every classical point.
Moreover, if $\bL_{\overline{x}}$ is potentially crystalline, then $U$ and $L$ can be chosen in such a way that $\bL|_{U_L}$ is horizontal de Rham and crystalline at every classical point.
\end{thm}

Let us first make a remark and discuss its consequence.

\begin{rem}
 If $x$ runs over all the classical points of $X$, the images in $X$ of $U$'s in the theorem form a collection of open subsets of $X$. In general, the union of such open subsets may not be the entire $X$.
\end{rem}

\begin{cor}\label{cor: bound of potential degree}
 Let $X$ be a quasi-compact smooth rigid analytic variety over $K$, and let $\bL$ be a de Rham $\Q_p$-local system on $X_{\et}$.
Then there exists a finite extension $L$ of $K$ such that for every $K$-rational point $x\in X$ and every geometric point $\overline{x}$ above $x$,
the representation $\bL_{\overline{x}}$ of the absolute Galois group of $k(x)$ becomes semistable over $L$.
\end{cor}

\begin{proof}[Proof of Corollary~\ref{cor: bound of potential degree}]
Note that the set of $K$-valued points of $X$ is compact with respect to the analytic topology;
the general case is reduced to the case where $X=\bB^n_K$, and the assertion is obvious in this case.
Since $\bL$ is \'etale locally an isogeny $\Z_p$-local system and \'etale maps are open,
 the corollary follows from Theorem~\ref{thm:p-adic monodromy} and the compactness of the set of $K$-valued points of $X$.
\end{proof}

We prove Theorem~\ref{thm:p-adic monodromy} in the rest of this section.
For this, we show that $(D_{\dR}(\bL),\nabla)$ is locally horizontal de Rham.
Then the assertion follows from Theorems~\ref{thm:p-adic monodromy for horizontal de Rham representations} and \ref{thm:horizontal de Rham and one point crystalline}.

\subsection{Preliminaries on smooth rigid analytic varieties}

The following lemma is easy to prove and we omit the proof.

\begin{lem}\label{lem:basics on polydisks}
 Let $x\in \bB^n_K$ be a point with residue class field $K$.
\begin{enumerate}
 \item 
There exists a $\Spa(K,\calO_K)$-isomorphism 
\[
 \bB^n_K\ra \bB^n_K
\]
sending $x$ to $0$.
 \item For every open neighborhood $V$ of $x$ in $\bB^n_K$,
there exists an open immersion of the form
\[
 \bB_K^n \hra V
\]
sending $x$ to $0$.
\end{enumerate}
\end{lem}

\begin{prop}\label{prop:local structure of smooth rigid analytic varieties}
 Let $X$ be an $n$-dimensional smooth rigid analytic variety over $K$.
Let $L$ be a finite extension of $K$ and $x\in X$ a classical point with residue class field $L$.
Then there exists an open immersion 
\[
 \bB^n_L\ra X_L
\]
sending $0\in \bB^n_L$ to $x\in X_L$.
\end{prop}

\begin{proof}
By replacing $X$ by $X_L$, we may assume $L=K$.
By shrinking $X$ if necessary, we may assume that the structure morphism $X\ra \Spa(K,\calO_K)$ factors as
\[
 X\stackrel{f}{\lra}\bB^n_K\ra \Spa(K,\calO_K),
\]
where $f$ is \'etale and the second is the structure morphism (\cite[Corollary 1.6.10]{Huber-etale}).
We may further assume that $f(x)=0$ by Lemma~\ref{lem:basics on polydisks}(i).

By \cite[Corollary 1.7.4, Lemma 1.6.4]{Huber-etale}, the map of local rings
\[
 f^\ast\colon\calO_{\bB^n_K,0}\ra \calO_{X,x}
\]
is finite with the same residue class field $K$.
Recall that $\calO_{\bB^n_K,0}$ is Henselian (cf.~\cite[p.9, Corollaire 2]{Houzel}, \cite[Theorem 2.1.5]{Berkovich-etale}).
Hence we conclude that $\calO_{\bB^n_K,0}\cong \calO_{X,x}$.

It follows that there exist an affinoid open neighborhood $U$ of $x\in X$,  an affinoid open neighborhood $V$ of $0\in \bB^n_K$ and a morphism $g\colon V\ra U$ such that $f\circ g$ is the inclusion $V\subset \bB^n_K$.
Note that $g$ is injective and that $g^\ast\colon \calO_{U,x}\ra \calO_{V,0}$ is an isomorphism. After shrinking $V$ and $U$ if necessary, we may further assume that $g$ is an isomorphism by \cite[7.3.3 Proposition 5]{BGR}.
By Lemma~\ref{lem:basics on polydisks}(ii),
we get an open immersion
\[
 \bB^n_K\ra V\stackrel{g}{\lra} U \ra X
\]
sending $0$ to $x$. 
\end{proof}

\subsection{Vector bundles with integrable connections on the polydisk $\bB_K^n$}

Let $A:=K\langle T_1,\ldots,T_n\rangle$ and let $d\colon A\ra \Omega_A$ denote the universal $K$-derivation of $A$ (\cite[Definition 1.6.1]{Huber-etale}).
Concretely, $\Omega_A$ is a finite free $A$-module of rank $n$ with generators $dT_1,\ldots,dT_n$, and $d\colon A\ra \Omega_A$ sends $f$ to $df=\sum_{i=1}^n\frac{\partial f}{\partial T_i}\,dT_i$ for $f\in A$ (\cite[1.6.2]{Huber-etale}). In other words, $\Omega_A=A\otimes_{A^\circ}\widehat{\Omega}_{A^{\circ}}$ (cf. Subsection~\ref{subsection:de Rham representations}).

The following result will be well known to experts.
For example, see \cite[Appendix III]{DGS}, \cite[Proposition 9.3.3]{Kedlaya-PDE} for the case $n=1$.

\begin{thm}\label{thm:pDF on polydisk}
Let $(M,\nabla)$ be a finite free $A$-module equipped with an integrable connection.
Then there exists $l\in\N$ such that the induced connection on
\[
 M\otimes_A K\langle T_1/p^l,\ldots,T_n/p^l\rangle
\]
has a full set of solutions, i.e., it is isomorphic to $(K\langle T_1/p^l,\ldots,T_n/p^l\rangle, d)^{\rank_A M}$.
\end{thm}

\begin{proof}
 Consider the natural inclusion
\[
 M\subset M\otimes_A K[[T_1,\ldots,T_n]]
\]
and extend the integrable connection to the latter.
By a classical result (cf.~\cite[Proposition 8.9]{Katz_Nilpotent}),
we know that $M\otimes_A K[[T_1,\ldots,T_n]]$ has a full set of solutions.
In fact, we can construct an explicit $K$-linear surjection
\[
P\colon M \ra \bigl(M\otimes_A K[[T_1,\ldots,T_n]]\bigr)^{\nabla=0}
\]
using the Taylor expansion, which we now explain.

For each $i=1,\ldots,n$, we set $D_i:=\nabla_{\frac{\partial}{\partial T_i}}\colon M\ra M$.
Then the integrability of $\nabla$ is equivalent to $D_i\circ D_{i'}=D_{i'}\circ D_i$ for every $i,i'=1,\ldots,n$.
For each $\underline{j}=(j_1,\ldots,j_n)\in \N^n$, we set
\[
 \lvert \underline{j}\rvert=j_1+\cdots+j_n,\quad
\underline{j}!:=(j_1!)\cdots (j_n!),\quad
T^{\underline{j}}:=T_1^{j_1}\cdots T_n^{j_n},
\]
and
\[
 D_{\underline{j}}:=\prod_{i=1}^n\bigl(D_i\bigr)^{j_i}\colon M\ra M.
\]
We define $P\colon M \ra M\otimes_A K[[T_1,\ldots,T_n]]$
by
\[
 P(f):=\sum_{\underline{j}\in\N^n}\frac{(-1)^{\lvert \underline{j}\rvert}}{\underline{j}!}T^{\underline{j}}D_{\underline{j}}(f).
\]
Then we can check $P(f)(0,\ldots,0)=f(0,\ldots,0)$, 
$ \Image P\subset \bigl(M\otimes_A K[[T_1,\ldots,T_n]]\bigr)^{\nabla=0}$, 
$\Ker P=(T_1,\ldots,T_n)M$, and that $P$ induces an isomorphism
\[
P\colon M/(T_1,\ldots,T_n) \stackrel{\cong}{\lra} \bigl(M\otimes_A K[[T_1,\ldots,T_n]]\bigr)^{\nabla=0}.
\]

Let $r=\rank_AM$ and choose an $A$-basis of $M$ so that we identify $M=A^r$.
For each $f\in A^r$, we can uniquely write $P(f)$ as
\[
 P(f)=\sum_{\underline{j}\in \N^n}A_{\underline{j}}(f)T^{\underline{j}},\quad A_{\underline{j}}(f)\in K^r.
\]
More concretely, we see
\[
 A_{\underline{j}}(f)
=\biggl(\biggl(\frac{\partial}{\partial T}\biggr)^{\underline{j}}(P(f))\biggr)\biggm|_{(0,\ldots,0)}
=\biggl(\biggl(\frac{\partial}{\partial T}\biggr)^{\underline{j}}\biggl(
\sum_{\underline{j}'\leq \underline{j}}\frac{(-1)^{\lvert \underline{j}'\rvert}}{\underline{j}'!}T^{\underline{j}'}D_{\underline{j}'}(f)
\biggr)\biggr)\biggm|_{(0,\ldots,0)},
\]
where 
$\bigl(\frac{\partial}{\partial T}\bigr)^{\underline{j}}:=\bigl(\frac{\partial}{\partial T_1}\bigr)^{j_1}\cdots\bigl(\frac{\partial}{\partial T_n}\bigr)^{j_n}$
for $\underline{j}=(j_1,\ldots,j_n)$, and $\underline{j}'\leq \underline{j}$ refers to the inequality entrywise.
In this way, we get a $K$-linear map $A_{\underline{j}}\colon A^r\ra K^r$.

Consider the Gauss norm on $A$ and equip $A^r(=M), M_r(A)$ with the norm defined by the maximum of the Gauss norm of each entry. Similarly, equip $K^r$ with the $p$-adic norm.
We normalize these norms so that the norm of $p$ is $p^{-1}$.
We will estimate the operator norm $\lVert A_{\underline{j}}\rVert$.
For each $i=1,\ldots,n$, write
\[
 D_i=\frac{\partial}{\partial T_i}+B_i, \quad B_i\in M_r(A).
\]
Choose $C'>0$ such that $\lVert B_i\rVert\leq C'$ for $i=1,\ldots,n$.
Set $C=\max\{1,C'\}$.
Note that both $\frac{\partial}{\partial T_i}\colon A^r\ra A^r$ and the multiplication by $T_i$ on $A^r$ have operator norm $1$ and that
the $p$-adic norm of $\underline{j}!$ is greater than or equal to $\bigl(\frac{1}{p^{1/(p-1)}}\bigr)^{\lvert \underline{j}\rvert}$.
It follows that
\[
 \biggl\lVert \frac{(-1)^{\lvert \underline{j}\rvert}}{\underline{j}!}T^{\underline{j}}D_{\underline{j}}\biggr\rVert
\leq C^{\lvert \underline{j}\rvert}p^{\lvert \underline{j}\rvert/(p-1)},
\]
and thus
\[
 \lVert A_{\underline{j}}\rVert
=\biggl\lVert \biggl(\biggl(\frac{\partial}{\partial T}\biggr)^{\underline{j}}\biggl(
\sum_{\underline{j}'\leq \underline{j}}\frac{(-1)^{\lvert \underline{j}'\rvert}}{\underline{j}'!}T^{\underline{j}'}D_{\underline{j}'}
\biggr)\biggr)\biggm|_{(0,\ldots,0)}\biggr\rVert
\leq C^{\lvert \underline{j}\rvert}p^{\lvert \underline{j}\rvert/(p-1)}.
\]

Choose $l\in\N$ such that
\[
 Cp^{1/(p-1)}p^{-l}<1.
\]
By the estimate on $\lVert A_{\underline{j}}\rVert$, we have
$\Image P\subset M\otimes_AK\langle T_1/p^l,\ldots,T_n/p^l\rangle$.
In other words, we have
\[
 \bigl(M\otimes_A K[[T_1,\ldots,T_n]]\bigr)^{\nabla=0}\subset M\otimes_AK\langle T_1/p^l,\ldots,T_n/p^l\rangle.
\]

We again identify $M\!=\!A^r$ and $M\otimes_AK\langle T_1/p^l,\ldots,T_n/p^l\rangle\!=\!K\langle T_1/p^l,\ldots,T_n/p^l\rangle^r$. Choose 
\[
 s_1=(s_{11},\ldots,s_{1r}),\ldots, s_r=(s_{r1},\ldots,s_{rr})\in M\otimes_AK\langle T_1/p^l,\ldots,T_n/p^l\rangle
\]
such that $s_1,\ldots,s_r$ form a $K$-basis of 
$\bigl(M\otimes_A K[[T_1,\ldots,T_n]]\bigr)^{\nabla=0}$ and such that
the values at $(T_1,\ldots,T_n)=(0,\ldots,0)$ satisfy
\[
 s_1(0,\ldots,0)=(1,0,\ldots,0),\ldots,s_r(0,\ldots,0)=(0,\ldots,0,1).
\]
Set 
\[
 s:=\det(s_{ij})_{i,j=1,\ldots,r}\in K\langle T_1/p^l,\ldots,T_n/p^l\rangle.
\]
Then $s$ satisfies $s(0,\ldots,0)=1\neq 0$.
In particular, there exists $l'\geq l$ such that $s$ is invertible in $K\langle T_1/p^{l'},\ldots,T_n/p^{l'}\rangle$.

Let us consider the induced connection $\nabla$ on $M\otimes_AK\langle T_1/p^{l'},\ldots,T_n/p^{l'}\rangle$.
Since $s_1,\ldots,s_r\in M\otimes_AK\langle T_1/p^{l'},\ldots,T_n/p^{l'}\rangle$, we have
\[
 \bigl(M\otimes_A K\langle T_1/p^{l'},\ldots,T_n/p^{l'}\rangle\bigr)^{\nabla=0}\subset M\otimes_AK\langle T_1/p^{l'},\ldots,T_n/p^{l'}\rangle.
\]
Moreover, since $s$ is invertible in $K\langle T_1/p^{l'},\ldots,T_n/p^{l'}\rangle$, we conclude
that the natural map
\[
 \bigl(M\otimes_A K\langle T_1/p^{l'}\!\!,\ldots,\!T_n/p^{l'}\!\rangle\bigr)\!^{\nabla=0}\!
\otimes_KK\langle T_1/p^{l'}\!\!,\ldots,\!T_n/p^{l'}\!\rangle
\!\ra\! M\otimes_AK\langle T_1/p^{l'}\!\!,\ldots,\!T_n/p^{l'}\!\rangle
\]
is an isomorphism.
This means that $(M\otimes_A K\langle T_1/p^{l'},\ldots,T_n/p^{l'}\rangle,\nabla)$ has a full set of solutions.
\end{proof}

\begin{rem}
Note that $A$ is a UFD by \cite[5.2.6 Theorem 1]{BGR}.
Hence any finite projective $A$-module is finite free.
\end{rem}

\subsection{Proof of Theorem~\ref{thm:p-adic monodromy}}

We prove Theorem~\ref{thm:p-adic monodromy}.
By Proposition~\ref{prop:local structure of smooth rigid analytic varieties} and Theorem~\ref{thm:pDF on polydisk}, 
there exists an open neighborhood $U\subset X_{k(x)}$ of $x$
such that $U\cong \bB^n_{k(x)}$ and $(D_{\dR}(\bL)|_U,\nabla)$ has a full set of solutions.
Hence $\bL|_U$ is horizontal de Rham because $D_{\dR}(\bL)|_U=D_{\dR}(\bL|_U)$.
By shrinking $U$ again, we may assume that $U$ is isomorphic to $\bT_{k(x)}^n$.
Write $\bT^n_{k(x)}=\Spa(A,A^\circ)$. Then $R=A^\circ$ satisfies the conditions in Set-up~\ref{set-up: Brinon's rings} relative to $k(x)$.
Let $V\in \Rep_{\Q_p}(\calG_A)$ be the representation corresponding to $\bL|_U$.
Then $V$ is horizontal de Rham by Lemma~\ref{lem:de Rham local system vs representation}.
By Theorem~\ref{thm:p-adic monodromy for horizontal de Rham representations}, there exists a finite extension $L$ of $k(x)$ such that $V|_{\calG_{A_L}}$ is horizontal semistable.
In particular, $\bL|_{U_L}$ is horizontal de Rham and semistable at every classical point.
The second assertion follows from Theorem~\ref{thm:horizontal de Rham and one point crystalline}.

\section{Applications: potentially crystalline loci and potentially good reduction loci }\label{section:potentially crystalline loci and potentially good reduction loci}

As an application of our $p$-adic monodromy theorem, we discuss potentially crystalline loci of de Rham local systems and cohomologically potentially good reduction loci of smooth proper families of relative dimension at most two.

Let $K$ be a complete discrete valuation field of characteristic zero with perfect residue field of characteristic $p$.
In Subsections~\ref{subsection:monodromy weight} and \ref{subsection:potentially good reduction locus}, we will further assume that the residue field of $K$ is finite.

\subsection{Potentially crystalline loci}
\begin{defn}
 Let $X$ be a smooth rigid analytic variety over $K$, and let $\bL$ be a de Rham $\Q_p$-local system on $X_{\et}$.
We say that an open subset $U$ of $X$ is \emph{a potentially crystalline locus} of $\bL$ if
$U$ satisfies the following property: for every classical point $x$ of $X$,
\begin{enumerate}
 \item if $x\in U$, then $\bL$ is potentially crystalline at $x$;
 \item if $x\not\in U$, then $\bL$ is not potentially crystalline at $x$.
\end{enumerate}
\end{defn}

\begin{thm}\label{thm:potentially crystalline loci}
For every smooth rigid analytic variety $X$ over $K$ and every de Rham $\Q_p$-local system $\bL$ on $X_{\et}$,  there exists a potentially crystalline locus of $\bL$.
\end{thm}

\begin{proof}
Observe that if there exists an \'etale covering $\{X_i\}_i$ of $X$ such that each $\bL|_{X_i}$ admits a potentially crystalline locus $U_i\subset X_i$, then the union of the images of $U_i$'s in $X$ is a potentially crystalline locus of $X$.
By passing to an \'etale cover of $X$, we may assume that $\bL$ is a de Rham isogeny $\Z_p$-local system on $X_{\et}$.

 Let $x$ be a classical point of $X$ such that $\bL$ is potentially crystalline at $x$.
By Theorem~\ref{thm:p-adic monodromy}, there exists an open neighborhood $U'_x\subset X_{k(x)}$ of $x$ such that $\bL|_{U'_x}$ is potentially crystalline at every classical point. Let $U_x$ denote the image of $U'_x$ in $X$; it is an open subset of $X$.
Then the union of such $U_x$'s satisfies the desired properties.
\end{proof}

\begin{rem}
 Our definition takes only classical rank $1$ points into consideration, and thus a potentially crystalline locus of $\bL$ is not unique.
On the other hand, the union of potentially crystalline loci is also a potentially crystalline locus and thus we have the largest potentially crystalline locus.
\end{rem}

\subsection{The monodromy-weight conjecture and $\ell$-independence}\label{subsection:monodromy weight}

We will discuss cohomologically potentially good reduction loci of smooth proper maps in the next subsection.
For this, let us briefly review basic results and conjectures on cohomologies of smooth proper varieties over $p$-adic fields.

Let $K$ be a complete discrete valuation field of characteristic zero with \emph{finite} residue field $k$ of characteristic $p$. Set $\Gal_K:=\Gal(\overline{K}/K)$. Let $q=p^h$ denote the cardinality of $k$.
Set $P_0:=W(\overline{k})[p^{-1}]$. It is equipped with an action of $\Gal_K$. Write $\sigma\colon P_0\ra P_0$ for the Witt vector Frobenius lifting $x\mapsto x^p$.

Let $I_K$ denote the inertia subgroup of $\Gal_K$ and let $W_K$ denote the Weil group of $K$; they fit in the following exact sequence:
\[
 0\lra I_K\lra W_K\stackrel{v}{\lra} h\Z\lra 0,
\]
where $v$ is normalized so that $h\in h\Z\subset h\widehat{\Z}=\Gal_k$ is the geometric Frobenius $a\mapsto a^{-q}$ of $k$. Set $W_K^+:=\{g\in W_K\mid v(g)\geq 0\}$.

Let $\ell$ be a prime.
For $\ell\neq p$, let $t_\ell\colon I_K\ra \Z_\ell(1)$ be the map sending $g\in I_K$ to $(g(\pi^{\ell^{-m}})/\pi^{\ell^{-m}})_m\in\Z_\ell(1)$, where $\pi$ is a uniformizer of $K$; the map $t_\ell$ is independent of the choice of a uniformizer $\pi$ nor its $\ell$-power roots.

Let $X$ be a smooth proper scheme purely of dimension $n$ over $K$.
For each $\ell$, the $\ell$-adic \'etale cohomology $H^m_{\et}(X_{\overline{K}},\Q_\ell)$ is an $\ell$-adic representation of $\Gal_K$. 

Let us first assume $\ell\neq p$.
By Grothendieck's monodromy theorem, there exist a nilpotent endomorphism $N\in \End(H^m_{\et}(X_{\overline{K}},\Q_\ell))(-1)$
and an open subgroup $I\subset I_K$ such that for $g\in I$, the action of $g$ on $H^m_{\et}(X_{\overline{K}},\Q_\ell)$ is given by $\exp(t_\ell(g)N)$.
The endomorphism $N$ defines the monodromy filtration $\Fil^N_\bullet$ on $H^m_{\et}(X_{\overline{K}},\Q_\ell)$; it is the unique separated and exhaustive increasing filtration such that
$N(\Fil^N_s)\subset\Fil^N_{s-2}(-1)$ for $s\in\Z$ and such that 
the induced map $N^s\colon \gr_s^NH^m_{\et}(X_{\overline{K}},\Q_\ell)\ra \gr_{-s}^NH^m_{\et}(X_{\overline{K}},\Q_\ell)(-s)$ is an isomorphism for $s\geq 0$.

\begin{conj}[$\ell$-adic monodromy-weight conjecture for $\ell\neq p$]
 For $s\in\Z$ and $g\in W_K$, the eigenvalues of $g$ on $\gr_s^NH^m_{\et}(X_{\overline{K}},\Q_\ell)$ are algebraic numbers and the complex absolute values of their conjugates are
$p^{(m+s)v(g)/2}$.
\end{conj}

\begin{rem}\label{rem:known cases of monodromy-weight conjecture}
 When $X$ has good reduction, the conjecture follows from Deligne's purity theorem on weights in \cite{WeilII}.
The conjecture is also known for $m\leq 2$; 
When $X$ has a strictly semistable model and $\dim X=2$, it is 
\cite[Satz 2.13]{Rapoport-Zink-local-zeta}.
Using \cite{deJong-alteration}, we can reduce the general case to this case (see \cite[Lemma 3.9]{Saito-weightSS} or the proof of Theorem~\ref{thm:p-adic monodromy weight} below). 
\end{rem}

We next discuss the case $\ell=p$.
For this we need to recall the general recipe of Fontaine \cite{Fontaine-exposeVIII}.

Let $V$ be a de Rham $\Q_p$-representation of $\Gal_K$.
By the $p$-adic monodromy theorem \cite[Th\'eor\`eme 0.7] {Berger}, 
$V$ is potentially semistable.
We set 
\[
 \widehat{D}_{\pst}(V):=\varprojlim_{H\subset I_K}(B_{\st}\otimes_{\Q_p}V)^H,
\]
where $H$ runs over open subgroups of $I_K$.
It is a $P_0$-vector space of dimension $\dim_{\Q_p}V$ equipped with a semilinear action of $\Gal_K$, a Frobenius $\varphi\colon \widehat{D}_{\pst}(V)\ra \widehat{D}_{\pst}(V)$ (i.e., a $\sigma$-semilinear injective map commuting with $\Gal_K$-action), and a monodromy operator $N\colon \widehat{D}_{\pst}(V)\ra \widehat{D}_{\pst}(V)$ (i.e., a $P_0$-linear map commuting with $\Gal_K$-action) satisfying $N\varphi=p\varphi N$. Such an object is called a \emph{$p$-module of Deligne} in \cite[1.1]{Fontaine-exposeVIII}.
It follows that $\varphi$ is bijective and that $N$ is nilpotent.

We associate to $\widehat{D}_{\pst}(V)$ a Weil--Deligne representation $W\widehat{D}_{\pst}(V)$ over $P_0$ as follows: Set $W\widehat{D}_{\pst}(V):=\widehat{D}_{\pst}(V)$ as a $P_0$-vector space equipped with the monodromy operator $N$. Define a $P_0$-linear action of $W_K$ on $W\widehat{D}_{\pst}(V)$ by letting $g\in W_K$ act as $g\cdot \varphi^{n(g)}$.
Since $N$ acts nilpotently, 
there exists a unique separated and exhaustive increasing filtration $\Fil^N_\bullet$ on $W\widehat{D}_{\pst}(V)$ such that $N(\Fil^N_s)\subset \Fil^N_{s-2}$ for $s\in\Z$ and such that
$N^s\colon \gr^N_s\ra\gr^N_{-s}$ is an isomorphism for every $s\geq 0$.

Let us return to the $p$-adic representation $H^m_{\et}(X_{\overline{K}},\Q_p)$.
It is de Rham by \cite{Tsuji-Cst} and \cite{deJong-alteration} (cf.~\cite[Theorem A1]{Tsuji-survey}). We denote by $H_{\pst,p}^m(X)$ the associated Weil--Deligne representation 
$W\widehat{D}_{\pst}(H^m_{\et}(X_{\overline{K}},\Q_p))$ of $K$.

\begin{conj}[$p$-adic monodromy-weight conjecture]\label{conj:p-adic monodromy weight}
 For $s\in\Z$ and $g\in W_K$, the eigenvalues of $g$ on $\gr_s^NH_{\pst,p}^m(X)$ are algebraic numbers and the complex absolute values of their conjugates are
$p^{(m+s)v(g)/2}$.
\end{conj}

\begin{rem}
This is part of the conjecture stated in \cite[p.347]{Jannsen-FontaineJannsenConjecture} and \cite[Conjecture 5.1]{Jannsen-survey}.
\end{rem}

The following result will be well known to experts, but we give a proof for completeness.

\begin{thm}\label{thm:p-adic monodromy weight}
 Conjecture~\ref{conj:p-adic monodromy weight} holds for $m\leq 2$.
\end{thm}

\begin{proof}
 We follow the proof of \cite[Lemma 3.9]{Saito-weightSS}.
Note that for a finite extension $L$ of $K$, $H_{\pst,p}^m(X)$ and $H_{\pst,p}^m(X_L)$
are isomorphic as Weil--Deligne representations of $L$ (after rescaling the monodromy operator; see \cite[2.2.6]{Fontaine-exposeVIII}).
It follows that the assertion for $H_{\pst,p}^m(X)$ is equivalent to the one for $H_{\pst,p}^m(X_L)$.

We may assume that $X$ is geometrically connected over $K$.
By \cite[Theorem 6.5]{deJong-alteration}, there exist a finite extension $L$ of $K$ inside $\overline{K}$, a smooth projective scheme $X'$ over $L$, and a proper surjective and generically finite morphism $f\colon X'\ra X$. The morphism $f$ induces a $\Gal_L$-equivariant injection $f^\ast\colon H^m_{\et}(X_{\overline{K}},\Q_p)\ra H^m_{\et}(X'_{\overline{K}},\Q_p)$ and thus an injective map
$f^\ast\colon H^m_{\pst,p}(X_L)\ra H^m_{\pst,p}(X')$ of Weil--Deligne representations of $L$.
In this case, it suffice to prove the assertion for the latter, and thus we may further assume that $X$ is projective by replacing $X$ by $X'$ and $K$ by $L$. By replacing $X$ by its hyperplane section, we may also assume $\dim X\leq 2$.

By applying \cite[Theorem 6.5]{deJong-alteration} again, we may assume that $X$ admits a projective and strictly semistable model $\calX$ over $\calO_K$.
 Denote the special fiber $\calX\otimes_{\calO_K} k$ by $Y$. Let $M_{\calX}$ be the fine log structure on $\calX$ associated to the simple normal crossing divisor $Y\subset \calX$ and let $M_Y$ be the inverse image of $M_{\calX}$ on $Y$.
Write $H^m_{\logcris}(Y/W)$ for the log-crystalline cohomology of $(Y,M_Y)/(k,\N\oplus k^\times)$ over $W:=W(k)$ and set $H^m_{\logcris}(\calX):=H^m_{\logcris}(Y/W)\otimes_WK_0$, where $K_0:=W[p^{-1}]$.
Then $H^m_{\logcris}(\calX)$ has a natural structure of $(\varphi,N)$-module relative to $K$\footnote{A choice of a uniformizer of $K$ equips $H^m_{\logcris}(\calX)$ with a filtration over $K$ by the comparison isomorphism to the log-de Rham cohomology, but we do not need this structure.}.
By \cite[Theorem 0.2]{Tsuji-Cst}, we have an isomorphism $D_{\st}(H^m_{\et}(X_{\overline{K}},\Q_p))\cong H^m_{\logcris}(\calX)$ of $(\varphi,N)$-modules relative to $K$\footnote{Depending on the convention of the monodromy operators, we may need to change the sign of one of the monodromy operators, but this does not affect our argument.}.
It follows that $\widehat{D}_{\pst}(H^m_{\et}(X_{\overline{K}},\Q_p))\cong H^m_{\logcris}(\calX)\otimes_{K_0}P_0$ as $p$-modules of Deligne.

Note that $\varphi^h$ acts $K_0$-linearly on $H^m_{\logcris}(\calX)$ and the monodromy operator defines the monodromy filtration $\Fil^N_\bullet$ on $H^m_{\logcris}(\calX)$.
By the definition of the Weil--Deligne representation $H^m_{\pst,p}(X)$ and the above isomorphism of $p$-modules of Deligne, it suffices to prove that 
 the eigenvalues of $\varphi^h$ on $\gr_s^NH^m_{\logcris}(\calX)$ are algebraic numbers and the complex absolute values of their conjugates are
$p^{(m+s)h/2}=q^{(m+s)/2}$.

For $r\geq 1$, let $Y^{(r)}$ denote the disjoint union of all $r$-fold intersections of the different irreducible components of $Y$.
Then we have the $p$-adic weight spectral sequence
\[
 E_1^{-l,m+l}=\mspace{-20mu}\bigoplus_{j\geq\max\{-l,0\}}\mspace{-20mu}H^{m-2j-l}_{\cris}(Y^{(2j+l+1)}/W)(-j-l)\otimes_WK_0\implies
H^m_{\logcris}(\calX).
\]
This spectral sequence is $\varphi$-equivariant
(see \cite[(9.11.1)]{Nakkajima-padicWSS} and \cite[\S 2 (B)]{Nakkajima-signs} for the details).
By the purity of weights of the $E_1$-term (\cite{Katz-messing}),
the spectral sequence degenerates at $E_2$, inducing the filtration $\Fil^W_\bullet$ on $H^m_{\logcris}(\calX)$ such that 
 the eigenvalues of $\varphi^h$ on $\gr_s^WH^m_{\logcris}(\calX)$ are algebraic numbers whose complex conjugates have absolute value $q^{(m+s)/2}$.
By \cite[Th\'eor\`eme 5.3, Corollaire 6.2.3]{Mokrane} and \cite[Theorem 8.3, Corollary 8.4]{Nakkajima-signs},
two filtrations $\Fil^N_\bullet$ and $\Fil^W_\bullet$ on $H^m_{\logcris}(\calX)$ agree.
This completes the proof.
\end{proof}

We now want to compare the results for $\ell$ and $p$.
For this, we need the following theorem on $\ell$-independence.

\begin{thm}[{\cite[Theorem B, Theorem 3.1]{Ochiai-l-indep}}]
 Let $X$ be a smooth proper scheme over $K$ of dimension $n$.
For every $g\in W_K^+$, the alternating sum
\[
 \sum_{m=0}^{2n}\tr(g^\ast;H_{\et}^m(X_{\overline{K}}, \Q_\ell))
\]
is a rational integer that is independent of $\ell$ with $\ell\neq p$.
Moreover, it is equal to 
\[
 \sum_{m=0}^{2n}\tr(g^\ast;H_{\pst,p}^m(X)).
\]
\end{thm}

\begin{cor}\label{cor:l-independence of each cohomology}
 Let $X$ be a smooth proper scheme over $K$ of dimension at most two.
For every $m\in\N$ and $g\in W_K^+$, the trace
$\tr(g^\ast;H_{\et}^m(X_{\overline{K}}, \Q_\ell))$
is a rational integer that is independent of $\ell$ with $\ell\neq p$.
Moreover, it is equal to $\tr(g^\ast;H_{\pst,p}^m(X))$.
\end{cor}

\begin{proof}
The cases $m=0, \dim X$ are obvious.
The same assertion for an abelian variety with $m=1$ is also known (cf.~\cite[Remarque 2.4.6 (iii)]{Fontaine-exposeVIII} and \cite[Corollary 2.2]{Noot}).
From these results, it remains to deal with the case where $\dim X=2$ and $m=2$; this follows from the above theorem on $\ell$-independence of the alternating sum.
\end{proof}

\begin{rem}
For $\ell\neq p$, we can also associate to the $\ell$-adic representation $H^m_{\et}(X_{\overline{K}},\Q_\ell)$ a Weil--Deligne representation of $K$ over $\Q_\ell$, which we denote by $H_{\pst,\ell}^m(X)$.
 It is expected that for every smooth proper scheme $X$ over $K$, the Weil--Deligne representations $H_{\pst,\ell}^m(X)$ with $\ell$ ranging all the primes including $p$ are compatible in the sense of \cite[2.4.2]{Fontaine-exposeVIII}. See \cite[2.4]{Fontaine-exposeVIII} for the details.
\end{rem}

\begin{cor}\label{cor:potentially good reduction for ell and p}
 Let $X$ be a smooth proper scheme of dimension at most two over $K$ and let $m\in\N$.
For a rational prime $\ell\neq p$,
the $\ell$-adic representation $H^m_{\et}(X_{\overline{K}},\Q_\ell)$ is potentially unramified if and only if the $p$-adic representation $H^m_{\et}(X_{\overline{K}},\Q_p)$ is potentially crystalline.
\end{cor}

\begin{proof}
The former condition in the assertion is equivalent to the condition that $N=0$ on $H^m_{\et}(X_{\overline{K}},\Q_\ell)$.
Since $\ell$-adic monodromy-weight conjecture is known for $X$ (Remark~\ref{rem:known cases of monodromy-weight conjecture}), it is also equivalent to the condition that, for every $g\in W_K^+$, the eigenvalues of $g$ on $H^m_{\et}(X_{\overline{K}},\Q_\ell)$, which are algebraic numbers, satisfy the property that the complex absolute values of their conjugates are $p^{mv(g)/2}$.
Similarly, by Theorem~\ref{thm:p-adic monodromy weight}, the latter condition in the assertion is equivalent to the condition that, for every $g\in W_K^+$, the eigenvalues of $g$ on $H^m_{\pst,p}(X)$, which are algebraic numbers, satisfy the property that the complex absolute values of their conjugates are $p^{mv(g)/2}$.
Hence the assertion easily follows from Corollary~\ref{cor:l-independence of each cohomology}.
\end{proof}

\subsection{Cohomologically potentially good reduction locus}\label{subsection:potentially good reduction locus}

Let $K$ be a complete discrete valuation field of characteristic zero with \emph{finite} residue field of characteristic $p$.

\begin{defn}
 Let $f\colon Y\ra X$ be a smooth proper morphism between smooth algebraic varieties over $K$ and let $X^{\ad}$ denote the adic space associated to $X$. 
We say that an open subset $U$ of $X^{\ad}$ is \emph{a cohomologically potentially good reduction locus}
of $f$ of degree $m$ if $U$ satisfies the following property:
for every classical point $x$ of $X^{\ad}$ and every geometric point $\overline{x}$ above $x$,  
\begin{enumerate}
 \item if $x\!\in\!U$, then $(R^m\!f_\ast\Q_\ell)_{\overline{x}}$ is potentially unramified for $\ell\neq p$ and 
$(R^m\!f_\ast\Q_p)_{\overline{x}}$ is potentially crystalline as Galois representations of $k(x)$;
 \item if $x\!\not\in \!U$, then $(R^m\!f_\ast\Q_\ell)_{\overline{x}}$ is not potentially unramified for $\ell\neq p$ and $(R^m\!f_\ast\Q_p)_{\overline{x}}$ is not potentially crystalline as Galois representations of $k(x)$. 
\end{enumerate}
\end{defn}

\begin{rem}
Note that our conditions do not characterize cohomologically potentially good reduction loci.
On the other had, the union of two cohomologically potentially good reduction loci is also a cohomologically potentially good reduction locus and thus we can consider the largest cohomologically potentially good reduction locus if there is at least one.
\end{rem}

\begin{thm}\label{thm:potentially good reduction locus}
Let $f\colon X\ra S$ be a smooth proper morphism between smooth algebraic varieties over $K$. Assume that the relative dimension of $f$ is at most two.
Then there exists a cohomologically potentially good reduction locus of $f$ of any degree.
\end{thm}

\begin{proof}
By Corollary~\ref{cor:potentially good reduction for ell and p}, it follows from Theorem~\ref{thm:potentially crystalline loci} or \cite[Proposition 5.2 (2)]{Kisin-LocalConstancy}.
\end{proof}

\begin{rem}
 As mentioned in the introduction, Imai and Mieda introduce a notion of the potentially good reduction locus of a Shimura variety (cf.~\cite[Definition 5.19]{Imai-Mieda}). It is a quasi-compact constructible open subset of the associated adic space where every $\ell$-adic automorphic \'etale local system  is potentially unramified for $\ell\neq p$ and every $p$-adic automorphic \'etale local system is potentially crystalline at every classical point.
Due to the constructibility, such a locus is unique if it exists. 
They prove the existence of the potentially good reduction locus of Shimura varieties of preabelian type, and use them to study the \'etale cohomology of Shimura varieties (cf.~\cite[Theorem~6.1]{Imai-Mieda}).
\end{rem}


\providecommand{\bysame}{\leavevmode\hbox to3em{\hrulefill}\thinspace}
\providecommand{\MR}{\relax\ifhmode\unskip\space\fi MR }
\providecommand{\MRhref}[2]{%
  \href{http://www.ams.org/mathscinet-getitem?mr=#1}{#2}
}
\providecommand{\href}[2]{#2}

\end{document}